\documentclass[pdftex, noinfoline]{imsart}

\RequirePackage[OT1]{fontenc}
\RequirePackage{amsthm,amsmath, amssymb}
\RequirePackage[numbers]{natbib}
\RequirePackage[colorlinks,citecolor=blue,urlcolor=blue]{hyperref}

\usepackage{amsfonts,dsfont} 
\usepackage{pstricks}
\usepackage{comment}
 
\usepackage{colordvi}
\usepackage{graphicx}

\definecolor{vert}{rgb}{0.09,0.7,0.17}
\definecolor{violet}{rgb}{0.69,0.13,0.69}
\definecolor{blendedblue}{rgb}{0.2,0.2,0.7}

\newcommand{\vp}{\vspace{.5cm}}

\newtheorem{theorem}{Theorem}
\newtheorem{proposition}{Proposition}
\newtheorem{corollary}{Corollary}
\newtheorem{lemma}{Lemma}

\newtheorem{remark}[theorem]{Remark}

\newtheorem{heu}{Heuristic}
\newcommand{\R}{\mathbb{R}} 
\newcommand{\1}{1
} 
 
\renewcommand{\P}{P}
\newcommand{\E}{E} 
\DeclareMathOperator{\var}{Var} 

\renewcommand{\l}{\ell}
\newcommand{\FDR}{\mbox{FDR}}
\newcommand{\FNR}{\mbox{FNR}}
\newcommand{\FNP}{\mbox{FNP}}
\newcommand{\BFDR}{\mbox{BFDR}}

\newcommand{\Cst}{C}
\newcommand{\MCI}{\mbox{\tt MCI}}
\newcommand{\SC}{\mbox{\tt SC}}

\newcommand{\mtc}{\mathcal}
\newcommand{\mbf}{\mathbf}

\newcommand{\ol}[1]{\overline{#1}}

\newcommand{\ind}[1]{{\mbf{1}\{#1\}}}

\newcommand{\al}{\alpha}

\newcommand{\eps}{\varepsilon}

\newcommand{\cH}{{\mtc{H}}}

\newcommand{\cS}{{\mtc{S}}}

\usepackage{graphics}
\newcommand{\FDP}{\mbox{FDP}}

\newcommand{\be}{\beta}

\newcommand{\ga}{\gamma}
\newcommand{\Ga}{\Gamma}

\newcommand{\La}{\Lambda}

\newcommand{\te}{\theta}
\newcommand{\ta}{\tau}
\newcommand{\veps}{\varepsilon}
\newcommand{\vphi}{\varphi}

\newcommand{\RR}{\mathbb{R}}
\newcommand{\leqa}{\lesssim}
\newcommand{\geqa}{\gtrsim}

\newcommand{\bfi}{\ol\Phi}

\newcommand{\given}{\,|\,}

\newcommand{\EM}{\ensuremath}
\newcommand{\cA}{\EM{\mathcal{A}}}
\newcommand{\cC}{\EM{\mathcal{C}}}
\newcommand{\cD}{\EM{\mathcal{D}}}

\newcommand{\cG}{\EM{\mathcal{G}}}
\newcommand{\cI}{\EM{\mathcal{I}}}

\newcommand{\cL}{\EM{\mathcal{L}}}
\newcommand{\cM}{\EM{\mathcal{M}}}

\newcommand{\EBayesqseuillage}{{\tt EBayesq$.0\,$}}
\newcommand{\vphiqvalseuillage}{\vphi^{\mbox{\tiny $q$-val$.0$}}}
\newcommand{\EBayesqplus}{{\tt EBayesq.hybrid$\,$}}

\usepackage{xr-hyper}

\begin{document}

\begin{frontmatter} 

\title{On spike and slab empirical Bayes multiple testing}
\runtitle{Spike and slab empirical Bayes multiple testing}
\thankstext{T1}{Work partly supported by grants of the ANR, projects ANR-16-CE40-0019 (SansSouci) and ANR-17-CE40-0001 (BASICS)}

\begin{aug}
\author{\fnms{Isma\"el} \snm{Castillo}\ead[label=e1]{ismael.castillo@upmc.fr}}
\and
\author{\fnms{\'Etienne} \snm{Roquain}\ead[label=e2]{etienne.roquain@upmc.fr}}
\address{
  Sorbonne Universit\'e, Laboratoire de Probabilit\'es, Statistique et Mod\'elisation, LPSM,\\ 4, Place Jussieu, 75252 Paris cedex 05, France\\
  \printead{e1,e2}
 }
 
 \affiliation{Sorbonne  Universit\'e}
\end{aug}

\begin{abstract}
This paper explores a connection between empirical Bayes posterior distributions and false discovery rate (FDR) control. In the Gaussian sequence model,  this work shows that empirical Bayes-calibrated spike and slab posterior distributions  allow a correct FDR  control under sparsity. Doing so, it offers a frequentist theoretical validation of  empirical Bayes methods 
in the context of multiple testing.
Our theoretical results are  illustrated with numerical experiments.
\end{abstract}
\begin{keyword}[class=AMS]
\kwd[Primary ]{}{62C12, 62G10}
\end{keyword}
\begin{keyword}
 \kwd{Frequentist properties of Bayesian procedures}  \kwd{False discovery rate} \kwd{sparsity} \kwd{multiple testing} \end{keyword}
\end{frontmatter}

\section{Introduction}
\label{sec:intro}

\subsection{Context} 

In modern high dimensional statistical models, several aims are typically pursued, often at the same time:  
{\em testing} of hypotheses on the parameters 
of interest, {\em estimation} and {\em uncertainty} quantification, among others. Due to their flexibility, in particular in the choice of the prior, Bayesian posterior distributions are routinely used to provide solutions to a variety of such inference problems. However, although practitioners may often directly read off quantities such as the posterior mean or credible sets once they have simulated posterior draws, the question of mathematical justification of the use of such quantities, in particular from a frequentist perspective, has recently attracted a lot of attention. While the seminal papers \cite{ggv}, \cite{sw01} set the stage for the study of posterior estimation rates in general models, the case of estimation in high dimensional models has been considered only recently from the point of view of estimation, see \cite{js04}, \cite{cv12}, \cite{vkv14} among others, while results on frequentist coverage of credible sets are just starting to emerge, see e.g. \cite{belnur15}, \cite{vsvv17d}. Some of the previous approaches rely on automatic data-driven calibration of the prior parameters, following the so-called {\em empirical Bayes} approach, notably \cite{js04}, estimating the proportion of significant parameters, and \cite{jiangzhang09}, where the full distribution function of the unknowns is estimated. 

 Our interest here is on the issue of {\em multiple testing} of hypotheses. 
 Typically, the problem is to identify the active variables among a large number of candidates.  This task appears in a wide variety of applied fields as genomics, neuro-imaging, astrophysics, among others. 
Such data typically involve more than thousands of variables with only a small part of them being significant (sparsity). 

In this context, a typical aim is to control the false discovery rate (FDR), see \eqref{deffdr} below, that is, to find a selection rule that ensures that the averaged proportion of errors among the selected variables is smaller than some prescribed level $\alpha$. 
This multiple testing type I error rate, introduced in \cite{BH1995}, became quickly popular with the development of high-throughput technologies because it is ``scalable" with respect to the dimension: the more rejections are possible, the more false positives are allowed.
A common way to achieve this goal is to compute the $p$-values (probability under the null that the test statistic is larger than the observed value) and to run the  Benjamini-Hochberg (BH) procedure \cite{BH1995}, which is often considered as a benchmark procedure. 
In the last decades, an extensive literature aimed at studying the BH method, by showing that it (or versions of it) controls the FDR in various frameworks, see \cite{BY2001,BKY2006,Sar2007,FDR2007}, among others.

In a fundamental work \cite{fdr06}, Abramovich, Benjamini, Donoho and Johnstone proved that a certain hard thresholding rule deduced from the BH procedure 
-- keeping only observations with significant $p$-values -- satisfies remarkable risk properties: it is minimax adaptive simultaneously for a range of losses and sparsity classes over a broad range of sparsity parameters. 
In addition, similar results hold true for the misclassification risks, see \cite{BCFG2011,NR2012}. 
These results in particular suggest a link between FDR controlling procedures and adaptation to sparsity.
Here, we shall follow a questioning that can be seen as `dual' to the former one: starting from a commonly used Bayesian procedure that is known to optimally  adapt to the sparsity in terms of risk over a broad range of sparsity classes (and even, under appropriate self-similarity type conditions, to produce adaptive confidence sets), we ask whether a uniform FDR control can be guaranteed.

\subsection{Setting} 

In this paper, we consider the Gaussian sequence model. One observes, for $1\le i\le n$, 
\begin{equation}\label{model}
X_i = \theta_{0,i} + \varepsilon_i, 
\end{equation}
for an unknown $n$-dimensional vector $\theta_{0}= (\theta_{0,i})_{1\leq i \leq n} \in \R^n$ and $\varepsilon_i$ i.i.d. $\mathcal{N}(0,1)$. This model can be seen as a stylized version of an high-dimensional model. 
The problem is to test
$$
\mbox{$H_{0,i}: ``\theta_{0,i}=0"$ against $H_{1,i}: ``\theta_{0,i}\neq 0"$},  
$$
 simultaneously over $i\in\{1,\dots,n\}$.
We also introduce the assumption that the vector $\theta_{0}$ is $s_n$-sparse, that is, is supposed to belong to  the set
\begin{equation}\label{equ:l0sn}
\ell_0[s_n] =  \{\theta\in \R^n\::\: \# \{1\leq i\leq n \::\: \theta_i\neq 0\} \le s_n\},
\end{equation}
for some sequence $s_n\in\{0,1,\dots,n\}$, typically much smaller than $n$, measuring the sparsity of the vector. 

\subsection{Bayesian multiple testing methodology}\label{sec:BMTM}

From the point of view of posterior distributions, one natural approach for testing is simply based on comparing posterior probabilities of the hypotheses under consideration. Yet, to do so, a choice of prior needs to be made, and for this reason it is important to carefully design a prior that is flexible enough to adapt to the unknown underlying structure (and, here, sparsity) of the model. This is one of the reasons behind the use of {\em empirical Bayes} approaches, that aim at calibrating the prior in a fully automatic, data-driven, way. Empirical Bayes methods for multiple testing have been in particular advocated by Efron (see e.g. \cite{Efron2008} and references therein) in a series of works over the last 10-15 years, reporting excellent behaviour of such procedures -- we describe two of them in more detail in the next paragraphs -- in practice. Fully Bayes methods, that bring added flexibility by putting prior on sensible hyperparameters, are another alternative. In the sequel {\it Bayesian multiple testing procedures} will be referred to as BMT for brevity.

Several popular BMT procedures rely on two quantities that can be seen as 
possible Bayesian counterparts of standard $p$-values: 
\begin{itemize}
\item the $\ell$-value: the probability that the null is true conditionally on the fact that the test statistics is {\it equal} to the observed value, see e.g. \cite{ETST2001};
\item the $q$-value: the probability that the null is true conditionally on the fact that the test statistics is {\it larger} than the observed value, introduced in \cite{Storey2003}.
\end{itemize}
(Note that the $\ell$-value is usually called ``local FDR". Here, we used another terminology to avoid any confusion between the procedure and the FDR.) Obviously, 
 these quantities are well defined only if 
 the trueness/falseness of a null hypothesis is random,  which is obtained by introducing an appropriate prior distribution.
 
Once the prior is calibrated (in a data-driven way or not), the $q$-values (resp. $\ell$-values) can be computed and combined to produce BMT procedures. For instance, existing strategies reject null hypotheses with:
\begin{itemize}
\item  a $\ell$-value smaller than a fixed cutoff $t=0.2$ \cite{Efron2007};
\item a $q$-value smaller than the  nominal level  $\alpha$ \cite{Efron2008};
\item averaged $\ell$-values smaller than the nominal level $\alpha$ \cite{muelleretal04,SC2007,SC2009}.   
\end{itemize} 
For alternatives see, e.g., \cite{AA2006,sarkaretal08}. 
In particular, one popular fact is that the use of Bayesian quantities ``automatically corrects for the multiplicity of the tests", see, e.g., \cite{Storey2003}; while using $p$-values requires to use a cutoff $t$ that decreases with the dimension $n$, using $\ell$-values/$q$-values can be used with a cutoff $t$ close to the nominal level $\alpha$, without any further correction.  
This is well known to be valid from a decision theoretic perspective for the Bayes FDR, that is, for the FDR integrated w.r.t. the prior distribution, as we recall in Proposition~\ref{prop:BFDR} below. 
When the hyper-parameters are estimated from the data within the BMT, the Bayes FDR is still controlled to some extent, as proved in \cite{SC2007,SC2009}. However,  controlling the Bayes FDR does not give theoretical guarantees for the usual frequentist FDR, that is, for the FDR at the true value of the parameter, as the pointwise FDR may deviate from an integrated version thereof. 

\subsection{Frequentist control of BMT}

In this paper,
our main aim is to study whether BMT procedures have valid frequentist multiple testing properties.

A first hint has already been given in \cite{Storey2003,Efron2008}: it turns out that the BH procedure can loosely be seen as a ``plug-in version" of the procedure rejecting the $q$-values smaller than $\alpha$ (namely, the theoretical c.d.f. of the $p$-values is estimated by its empirical counterpart). Since the BH procedure controls the (frequentist) FDR, this 
might suggest a possible connection between BMT and successful frequentist multiple testing procedures.

In regard to the rapidly increasing literature on frequentist validity of Bayesian procedures from the {\em estimation} perspective, the multiple testing question for BMT procedures has been less studied so far from the theoretical, frequentist, point of view. This is despite a number of very encouraging simulation performance results, see e.g. \cite{muelleretal04, caoetal09, guindanietal09, martintokdar12}.  
A recent exception is the interesting preprint \cite{salomond17} that shows a frequentist FDR control for a BMT based on a continuous shrinkage prior; yet, this  control holds under a certain signal-strength assumption only. One main question we ask in the present work is whether a fully {\em uniform} control (over sparse vectors) of the frequentist FDR is possible for some posterior-based BMT procedures.  
Also, while the constants in the risk bounds  are not made explicit  in \cite{salomond17}, 
we would like to clarify whether the final FDR control is made at, or close to, the required level $\alpha$. The FDR control results below will also be complemented by appropriate type II-error controls.

\subsection{Spike and slab prior distributions and sparse priors}

Let $w\in(0,1)$ be a fixed hyper-parameter.
Let us define the  prior distribution $\Pi=\Pi_{w,\gamma}$ 
on $\R^n$ as 
\begin{equation}\label{prior}
\Pi_{w,\gamma} = ((1-w)\delta_0+w \cG)^{\otimes n},
\end{equation}
where $\cG$ is a distribution with a symmetric density $\gamma$ on $\R$.  Such a prior is a tensor product of a mixture of a Dirac mass at $0$ (spike), that reflects the sparsity assumption, and of an absolutely continuous distribution (slab), that models nonzero coefficients. This is arguably one of the most natural priors on sparse vectors and has been  considered in many key contributions on Bayesian sparse estimation and model selection, see, e.g., \cite{MitchellBeauchamp}, \cite{GeorgeFoster}. 
 
Of course, an important question is that of  the choice of $w$ and $\ga$. A popular choice of $w$ is data-driven and based on a marginal maximum likelihood empirical Bayes method (to be described in more details below). The idea is to make the procedure  learn the intrinsic sparsity while also incorporating some automatic multiplicity correction, as discussed e.g. in \cite{scottberger10, bgt08}.  Following such an approach in a fundamental paper, Johnstone and Silverman \cite{js04} show that, provided $\ga$ has tails at least as heavy as Laplace, the posterior median of the empirical Bayes posterior is rate adaptive for a wide range of sparsity parameters and classes, is fast to compute and enjoys excellent behaviour in simulations (the corresponding R--package {\tt EBayesThresh} \cite{js05} is widely used). Namely, if $\|\cdot\|$ denotes the euclidian norm and  $\hat\te=\hat\te(X)$ is the coordinate-wise median of the empirical Bayes posterior distribution, there exists $c_1>0$ such that
 \begin{align}
\sup_{\te_0\in\ell_0[s_n]} 
E_{\te_0} \|\hat \te-\te_0\|^2 & \leq c_1s_n\log(n/s_n).
 \label{riskpoint}
\end{align}
Thus, asymptotically (in the regime $s_n,n\to\infty$, $s_n/n\to 0$), it matches up to a constant the minimax risk for this problem (\cite{djhs92}). In the recent work \cite{cm17}, the convergence of the empirical Bayes full posterior distribution (not only aspects such as median or mean) is considered, and similar results can be obtained, under stronger conditions on the tails of $\ga$ (for instance $\ga$ Cauchy works). More precisely, for $\Pi(\cdot\given X)=\hat\Pi(\cdot\given X)$ the empirical Bayes posterior, one can find a constant $C_1>0$ such that
 \begin{align}
\sup_{\te_0\in\ell_0[s_n]} 
E_{\te_0} \int \|\te-\te_0\|^2 d\Pi(\te\given X) & \leq C_1s_n\log(n/s_n).
 \label{riskcontrol}
\end{align}

 Further, under some conditions, one can show that certain credible sets from the posterior distributions are also adaptive confidence sets in the frequentist sense \cite{cs18}. Alternatively, one can also follow a hierarchical approach and put a prior on $w$. The paper \cite{cv12} obtains adaptive rates for such a fully Bayes procedure over a variety of sparsity classes, and presents a polynomial time algorithm to compute certain aspects of the posterior. 

Empirical Bayes approaches have also been successfully applied to a variety of different sparse priors such as empirically recentered Gaussian slabs as in \cite{belnur15, belghos16}, or the horseshoe \cite{vsvv17r, vsvv17d}, both studied in terms of estimation and the possibility to construct adaptive confidence sets. In \cite{jiangzhang09}, an empirical Bayes approach based on the `empirical' cdf of the $\theta$s is shown to allow for optimal adaptive estimation over various sparsity classes. For an overview on the rapidly growing literature on  sparse priors, we refer to the discussion paper \cite{vsvv17d}. 

Yet, most of the previous results are concerned with estimation or confidence sets, although a few of them report empirical false discoveries, e.g. \cite{vsvv17d}, Figure 7, without theoretical analysis though.

\subsection{Aim and results of the paper} 
Here we wish to find --  if this is at all possible --  a posterior-based procedure using a prior $\Pi$ (possibly an empirical Bayes one i.e. $\Pi=\hat\Pi$), that can perform {\em simultaneous inference} in that a) it behaves optimally up to constants in terms of the quadratic risk in the sense of \eqref{riskpoint} (or \eqref{riskcontrol}), 
b) its frequentist FDR at {\em any} sparse vector is bounded from above by (a constant times) a given nominal level. 
More precisely,  given a nominal level $t\in(0,1)$  and  $\vphi_t$ a multiple testing procedure deduced from $\Pi$
 ($\ell$-values or $q$-values procedure, as listed in Section~\ref{sec:BMTM}) we want to {\em validate} its use in terms of a uniform control of its false discovery rate $\text{FDR}(\te_0,\vphi_t)$, see \eqref{deffdr} below, over the whole parameter space. That is, we ask whether we can find $C_2>0$ independent of $t$ such that, for $n$ large enough, 
\begin{align}
\sup_{\te_0\in\ell_0[s_n]} \text{FDR}(\te_0,\vphi_t) & \le C_2\:t. \label{fdrcontrol}
\end{align}
Our main results are as follows: for a sparsity $s_n=O(n^\upsilon)$ with $\upsilon\in(0,1)$,
\begin{itemize}
\item Theorem~\ref{th1} shows that  \eqref{fdrcontrol} holds with $C_2$ arbitrary small for the BMT procedure rejecting the nulls whenever the corresponding $\ell$-value is smaller than $t$. 
\item Theorem~\ref{th2} shows that  \eqref{fdrcontrol} holds for some $C_2>0$ for the BMT procedure rejecting the nulls whenever the corresponding $q$-value is smaller than $t$ (with a slight modification if only few signals are detected).
\end{itemize}
These results hold for spike and slab priors,  
 for $\ga$ being Laplace or Cauchy,  or even for slightly more general heavy-tailed distributions.  The hyperparameter $\hat w$ is chosen according to a certain empirical Bayes approach to be specified below (with minor modifications with respect to the choice of \cite{js04}). 
In addition, 
 it is important to evaluate the amplitude of $C_2>0$ in  \eqref{fdrcontrol}. 
Our numerical experiments support the fact that, roughly, $C_2=1$. Furthermore, Theorem~\ref{theorem:psharpFDRcontrol} shows that  for some  
subset $\cL_0[s_n]\subset \ell_0[s_n]$ (containing strong signals), we have
for the $q$-value BMT, for any (sequence) $\te_0\in \cL_0[s_n]$,
\begin{equation}
\label{fdrcontrolexact}
\lim_n \ \text{FDR}(\te_0,\vphi_t)  = t\:,
\end{equation}
so the FDR control is exactly achieved asymptotically in that case. 

Finally, we provide a control of the type II error of the considered procedures by showing in Theorem \ref{theorem:power} that if FNR$(\te_0,\vphi)$ denotes the average number of non-discoveries of a procedure $\vphi$, for $\te_0\in\cL_0[s_n]$ as above,
\begin{equation}
\label{fnrcontrol}
 \lim_{n}  \text{FNR}(\te_0,\vphi_t) =0,
\end{equation} 
where $\vphi_t$ can either be the $\ell$-values or $q$-values procedure at level $t$.  

It follows from these results (combined with previous results of \cite{js04,cm17}) that the posterior distribution associated to a spike and slab prior, with $\ga$ Cauchy and a suitably empirical Bayes--calibrated $w$, is appropriate to perform several  tasks: \eqref{fdrcontrol}-\eqref{fdrcontrolexact}-\eqref{fnrcontrol} (multiple testing), \eqref{riskcontrol}--\eqref{riskpoint} (posterior concentration in $L^2$-distance). The posterior can also be used to build honest adaptive confidence sets (\cite{cs18}). The present work, focusing on the multiple testing aspect, then  completes the inference picture for spike and slab empirical Bayes posteriors, confirming their excellent behaviour in simulations.

\subsection{Organisation of the paper}

In Section \ref{sec:bmt}, we introduce Bayesian multiple testing procedures associated to spike and slab posterior distributions
 as well as the considered empirical Bayes choice of $w$. 
In Section  \ref{sec:EBw}, our main results are stated, while Section \ref{sec:simu} contains numerical experiments,  Section~\ref{sec:other}  presents some related BMT procedures and Section \ref{sec:disc} gives a short discussion.  
Preliminaries for the proofs are given in Section \ref{sec:prelim}, while the proof of 
Theorems~\ref{th1} and \ref{th2} 
can be found in Section \ref{sec:proof:th1}. The supplementary file \cite{CR2017supp} gathers a number of lemmas used in the proofs, as well as the proofs of Propositions \ref{prop:BFDR}--\ref{prop:typeIerror} and Theorems \ref{theorem:psharpFDRcontrol}, \ref{theorem:power} and \ref{thm-mci}. The sections and equations of this supplement are referred to with an additional symbol ``S-'' in the numbering.

\subsection{Notation}
 
In this paper, we use the following notation:

\begin{itemize}
\item for $F$ a cdf, we set $\ol{F}=1-F$
\item $\phi(x)=(2\pi)^{-1/2} e^{-x^2/2}$ and $\Phi(x)=\int_{-\infty}^x \phi(u)du$
\item $u_n  \asymp v_n$ means that there exists constants $C,C'>0$ such that  $|v_n| c \leq |u_n|\leq C |v_n|$ for $n$ large enough; 
\item $u_n  \leqa v_n$ means that there exists constants $C>0$ such that   $|u_n|\leq C |v_n|$ for $n$ large enough;
\item $f(y)  \asymp g(y)$, for $y\in A$ means that there exists constants $C,C'>0$  such that for all $y\in A$,   $c |g(y)| \leq |f(y)| \leq C |g(y)|$;
\item $f(y)  \asymp g(y)$, as $y\to\infty$ means that there exists constants $C,C'>0$  such that  $c |g(y)| \leq |f(y)| \leq C |g(y)|$ for $y$ large enough; 
\item $u_n\sim v_n$ means $u_n-v_n=o(u_n)$.
\end{itemize} 

Also, for $\ta\in\RR^n$, the symbol $\E_\ta$ (resp. $\P_\ta$) denotes the expectation (resp. probability) under $\te_0=\ta$ in the model \eqref{model}. 
The support of $\theta_0\in\RR^n$ is denoted by 
$S_{\te_0}=\{i\::\: \te_{0,i}\neq 0\}$ or sometimes $S_0$ for simplicity. The cardinality of the support $S_{\te_0}$ is denoted by
$\sigma_0=|S_0|$.

\subsection{Relevance and novelty of the approach}

We now briefly emphasize connections with existing works, and discuss several merits of the proposed  approach.
First, studying theoretical properties of BMT procedures is motivated by the fact that they are routinely used in practice  
since  Efron's seminal papers \cite{ETST2001,Efron2008}; in the context of genomic applications, we refer for instance to a recent series of works by  Stephens and co-authors \cite{Steph2016,GS2018} and references therein. Second,  
we note that just a few other procedures to date  theoretically allow both  estimation at minimax rate and uniform FDR control: besides the BH procedure \cite{BH1995,fdr06}, the SLOPE procedure \cite{bogdanslope15}, \cite{sucandes16} also enjoys these two properties in a regression context. In addition, the Bayesian maximum a posteriori (MAP) rule \cite{abramovichetal07} has a minimax estimation rate and shares connections with the BH rule \cite{AA2006} for some specific choice of the prior. Third, let us mention that Sun, Cai and coauthors have also investigated a generic $\l$-value-based approach (see Section~\ref{rem:SC} for more details) that allows to control the FDR in structured settings where the BH procedure can be suboptimal \cite{SC2009,CS2009,CSW2019}. Nevertheless, the proposed FDR control is not uniform from the frequentist perspective, and is restricted to a specific asymptotical setting. 
Interestingly, using the present spike and slab prior in these contexts seems promising to get uniform FDR control while improving upon the BH procedure. During the submission process of the manuscript, a first encouraging attempt has  being made by the second author in the discussion part of the paper \cite{CSW2019} (see page $218$ therein).

To summarize, the present work aims at providing guarantees for a widely used class of 
 $\ell$-value/$q$-value-based BMT procedures,
deploying a spike and slab prior with suitably heavy tails and empirical Bayes choice of the weight. Further, by doing so, and combining with results from recent parallel investigations \cite{cm17,cs18}, our work demonstrates that  the corresponding posterior distribution produces simultaneously optimal estimation rates, confidence sets and uniform FDR control (as well as FNR control over appropriately large signals), thereby achieving a complete inference picture along the three canonical inferential goals of ``estimation, testing (here, multiple) and confidence sets".  We are not aware of any another method that produces simultaneously these (frequentist) inferences in the present setting.

\section{Preliminaries}
 \label{sec:bmt}

\subsection{Procedure and FDR}

A multiple testing procedure is a measurable function of the form $\vphi(X)=(\vphi_i(X))_{1\leq i\leq n}\in\{0,1\}^n$, where each $\vphi_i(X)=0$ (resp. $\vphi_i(X)=1$) codes for accepting $H_{0,i}$ (resp. rejecting $H_{0,i}$). 
For any such procedure $\vphi$, we let
\begin{equation} \label{deffdr}
\FDR(\theta_{0},\vphi) = \E_{\theta_0}\left[\frac{\sum_{i=1}^n \ind{\theta_{0,i}=0} \vphi_i(X)}{1\vee \sum_{i=1}^n  \vphi_i(X)}\right]  .
\end{equation}
A procedure $\vphi$ is said to control the FDR at level $\alpha$ if $\FDR(\theta_{0},\vphi)\leq \alpha$ for any $\theta_{0}$ in $\R^n$. 
Note that under $\theta_{0}=0$, we have $\FDR(\theta_{0},\vphi)=\P_{\theta_{0}=0}(\exists i\::\:\vphi_i(X)=1)$, which means that an $\alpha$--FDR controlling procedure provides in particular a (single) test of level $\alpha$ of the full null ``$\theta_{0,i}=0$ for all $i$". 
As already mentioned, in the framework of this paper, our goal  
is a control of the FDR around the pre-specified target level, as in  \eqref{fdrcontrol} or \eqref{fdrcontrolexact} (where $t=\al$).

\subsection{Prior, posterior, $\ell$-values and $q$-values}

Recall the definition of the prior distribution $\Pi=\Pi_{w,\gamma}$ from \eqref{prior} and let
\begin{equation} 
g(x)=
\int \gamma(x-u)\phi(u)du. \label{equg}
\end{equation}
The posterior distribution $\Pi[\cdot\given X]=\Pi_{w,\ga}[\cdot\given X]$ of $\theta$ is explicitly given by
\begin{align}
\theta \:|\: X&\,\sim\, \bigotimes_{i=1}^n\  \l_i(X)\, \delta_0 + (1-\l_i(X))\, \cG_{X_i},\label{posteriordistribution}
\end{align}
where $\cG_{x}$ is the distribution  with density $\gamma_{x}(u) := \phi(x-u) \gamma(u)/g(x)$ and
\begin{align}
\l_i(X)&=\l(X_i;w,g);\label{lvalues}\\
\l(x;w,g)&=\Pi(\theta_1=0 \:|\: X_1=x)= \frac{(1-w)\phi(x)}{(1-w)\phi(x) + w g(x)}.\label{lformula}
\end{align} 
The quantities $\l_i(X)$, $1\leq i \leq n$, given by \eqref{lvalues} are called the {\it $\ell$-values}. 
Note that, although we do not emphasize it in the notation for short, the $\ell$-values depend also on $w$ and $g$. The $\ell$-value measures locally, for a given observation $X_i$, the probability that the latter comes from pure noise.  This is why it is sometimes called `local-FDR', see \cite{ETST2001}. 

If one has in mind a range of values 
--i.e. those that exceed a given amplitude--, a different measure is given by the {\em q-values} defined by
\begin{align}
q_i(X)&=q(X_i;w,g);\label{qvalues}\\
q(x;w,g)&=\Pi(\theta_1=0 \:|\: |X_1|\geq |x|)= \frac{(1-w)\ol{\Phi}(|x|)}{(1-w)\ol{\Phi}(|x|) + w \:\ol{G}(|x|)};\label{qvaluesfunc}\\
\ol{G}(s)&=\int_{s}^{+\infty} g(x)dx.
\end{align}
The identity \eqref{qvaluesfunc} relating the $q$-value to $\ol{\Phi}, \ol{G}$ is proved in Section~\ref{sec:proplqval}.

\subsection{Assumptions}\label{sec:assumpnota}

We follow throughout the paper assumptions similar to those of \cite{js04}.
The prior $\gamma$ is assumed to be unimodal, symmetric and so that 
\begin{align} 
|\log\ga(x) - \log\ga(y)| & \le \La|x-y|,\quad x,y\in\RR; \label{asump1}\\
\ga(y)^{-1}\int_y^\infty\ga(u)du & \asymp y^{\kappa-1},\quad\text{as } y\to\infty, \:\: \kappa\in[1,2]; \label{asump2}\\
y\in \R&\to y^2\ga(y)\mbox{ is bounded}. \label{asump3}
\end{align}

Conditions  \eqref{asump1}, \eqref{asump2} and \eqref{asump3}  above are for instance true when $\gamma$ is Cauchy ($\kappa=2$, $\Lambda=1$) or Laplace ($\kappa=1$, $\Lambda$ is the scaling parameter). As we show in Remark~\ref{rem:explicit}, explicit expressions exist for $g$, see \eqref{equg}, in the Laplace case. In the Cauchy case, the integral is not explicit, but in practice  (to avoid approximating the integral) one can work with the quasi-Cauchy prior,  see \cite{js05}, that satisfies the above conditions and corresponds to
\begin{align}
\gamma(x)&=(2\pi)^{-1/2} (1- |x| \ol{\Phi}(x)/\phi(x))\label{gammaquasiCauchy};\\
 g(x)&=(2\pi)^{-1/2} x^{-2} (1-e^{-x^2/2}).\label{gquasiCauchy}
\end{align}
The condition  \eqref{asump3} is mostly for simplicity to get unified proofs, but heavier tails could be consider as well, by adapting estimates of \cite{cs18}.

\subsection{Bayesian Multiple Testing procedures (BMT)}

We define the multiple procedures defined from the $\ell$-values/$q$-values in the following way:
\begin{align}\label{def:lvalproc}
\vphi^{\mbox{\tiny $\ell$-val}}_i(t;w,g)&=\mathds{1}_{\{\l_i(X)\leq t\}},\:\:1\leq i\leq n;\\
\vphi^{\mbox{\tiny $q$-val}}_i(t;w,g)&=\mathds{1}_{\{q_i(X)\leq t\}},\:\:1\leq i\leq n, \label{def:qvalproc}
\end{align}
where $t\in(0,1)$ is some threshold, that possibly depends on $X$.
As we will see in Section~\ref{sec:BMT}, these two procedures, denoted $\vphi^{\mbox{\tiny $\ell$-val}}(t)$, $\vphi^{\mbox{\tiny $q$-val}}(t)$ for brevity, simply correspond to (hard) thresholding procedures that select the $|X_i|$'s larger than some ({\em random}) threshold. The value of the threshold is driven by the posterior distribution in a very specific way: it depends on $\gamma$,  $t$, and on the whole data vector $X$ through the empirical Bayes choice of the hyper-parameter $w$, that automatically ``scales" the procedure according to the sparsity of the data.

\subsection{Controlling the Bayes FDR}

If the aim is to control the FDR at some level $\alpha$, a first result indicates that choosing $t=\alpha$ in $\vphi^{\mbox{\tiny $\ell$-val}}(t)$ and $\vphi^{\mbox{\tiny $q$-val}}(t)$ may be  appropriate, because the corresponding procedures control the Bayes FDR, that is, the FDR where the parameter $\theta$ has been integrated with respect to the prior distribution (see, e.g., \citep{sarkaretal08}). 
More formally, for any multiple testing procedure $\vphi$, and hyper-parameters $w$ and $\gamma$, define
\begin{align}
\BFDR(\vphi;w,\gamma)&=\int_{\R^n} \FDR(\theta,\vphi) d\Pi_{w,\gamma}(\theta) \label{equBFDR}.
\end{align}
Then the following result holds.

\begin{proposition}\label{prop:BFDR}
Let $\alpha\in (0,1)$ and $w\in(0,1)$ and consider any density $\gamma$ satisfying the assumptions of Section~\ref{sec:assumpnota}. Let $\vphi^\l=\vphi^{\mbox{\tiny $\ell$-val}}(\alpha;w,g)$ as defined in \eqref{def:lvalproc} and $\vphi^q=\vphi^{\mbox{\tiny $q$-val}}(\alpha;w,g)$ as defined in \eqref{def:qvalproc}.
Then we have 
\begin{align}
\BFDR(\vphi^\l;w,\gamma) &\leq \alpha \:\P(\exists i\::\: \l_i(X)\leq \alpha)\label{equBFDRlval}\:\\
&\leq \alpha \:\P(\exists i\::\: q_i(X)\leq \alpha)=\BFDR(\vphi^q;w,\gamma) \leq \alpha \label{equBFDRqval}.
\end{align}
\end{proposition}

This result can be certainly considered as well known, as \eqref{equBFDRlval} (resp. \eqref{equBFDRqval}) is similar in essence to Theorem~4 of \cite{SC2009} (resp. Theorem~1 of \cite{Storey2003}). It is essentially a consequence of Fubini's theorem, see Section~\ref{sec:prop:BFDR} for a proof.
While Proposition~\ref{prop:BFDR} justifies the use of $\ell$/$q$-values from the purely Bayesian perspective, it does not bring any information about $\FDR(\theta_0,\vphi^\l)$ and $\FDR(\theta_0,\vphi^q)$ at an arbitrary sparse vector $\theta_0\in \R^n$. 

\subsection{Marginal maximum likelihood}

In order to choose the hyper-parameter $w$, we explore now the choice made in \cite{js04}, following the popular marginal maximum likelihood method.
Let us introduce the auxiliary functions
\begin{align}\label{functionbeta}
 \beta(x) &= \frac{g}{\phi}(x)-1 ;\:\:\:
\beta(x,w)=\frac{\be(x)}{1+w\be(x)} .
 \end{align}
A useful property is that $\beta$  is increasing on $[0,\infty)$ from $\beta(0)\in (-1,0)$ to infinity, see Section \ref{sec:assumg}.  
The marginal likelihood for $w$ is by definition the marginal density of $X$, given $w$, in the Bayesian setting. Its logarithm is equal to  
$$ L(w)=
 \sum_{i=1}^n \log \phi(X_i) + \sum_{i=1}^n \log\left( 1+w \be(X_i)\right),$$
 which is a differentiable function on $[0,1]$. The derivative $\cS$ of $L$, the score function, can be written as
\begin{equation}\label{equ:score}
 \cS(w) = \sum_{i=1}^n \frac{\be(X_i)}{1+w\be(X_i)} =  \sum_{i=1}^n \be(X_i,w). 
 \end{equation}
The function $w\in[0,1]\to \cS(w)$ is (a.s.) decreasing and thus $w\in[0,1]\to L(w)$ is (a.s.) strictly concave. Hence, almost surely, the maximum of the function $L$ on a compact interval exists,  is  unique, and we can define the marginal  maximum likelihood estimator $\hat w$ by 
\begin{equation}\label{defw}
 \hat w \ =\ \underset{w\in\left[\frac1n,1\right]}{\text{argmax}}\  L(w)\:\:\:\: \mbox{ (a.s.)}.
\end{equation} 
This choice of $\hat w$ is close to the one in \cite{js04}. The only difference is in the lower bound, here $1/n$, of the maximisation interval, which differs from the choice in \cite{js04} by a slowly varying term. This difference is important for multiple testing in case of weak or zero signal (in contrast to the estimation task, for which this different choice does not modify the results).
Another slightly different choice of interval, still close to $[1/n,1]$, will also be of interest below.
In addition, if $\hat w\in(1/n,1)$,  it solves the equation $\cS(w)=0$ in $w$.  
However, note that in general the maximiser $\hat w$ can be at the boundary and thus may not be a zero of $\cS$.

\section{Main results}\label{sec:EBw}

Let us first describe the $\ell$-value algorithm.\\

\begin{center}
\fbox{\begin{minipage}{0.9\textwidth}
   \begin{center}
{ Algorithm$\,$ {\tt EBayesL} }
   \end{center}
{\tt Input:} $X_1,\ldots,X_n$, slab prior $\ga$, target confidence $t$\\
{\tt Output:} BMT procedure $\vphi^{\mbox{\tiny $\l$-val}}$
   \begin{enumerate}
     \item Find the maximiser $\hat w$ given by \eqref{defw}.
      \item Compute $\hat\ell_i(X)=\ell(X_i;\hat w,g)$ given by \eqref{lformula}.
      \item Return, for $1\le i\le n$,
      \begin{equation}
       \vphi_i^{\mbox{\tiny $\l$-val}}=\ind{\hat\ell_i(X)\le t}.
       \label{def:ebayesl}
      \end{equation} 
    \end{enumerate}
\end{minipage}}
\end{center}

\vp

\begin{theorem}\label{th1}
Consider the parameter space $\ell_0[s_n]$ given by \eqref{equ:l0sn} with sparsity $s_n \le n^\upsilon$ for some $\upsilon\in(0,1)$. 
Let $\ga$ be a unimodal symmetric slab density that satisfies \eqref{asump1}--\eqref{asump3} with $\kappa$ as in \eqref{asump2}. Then the algorithm {\tt EBayesL} produces as output the BMT $\vphi^{\mbox{\tiny $\l$-val}}$ defined in \eqref{def:ebayesl} that satisfies the following: there exists a constant $C=C(\gamma,\upsilon)$ such that for any $t\le 3/4$, there exists an integer $N_0=N_0(\gamma,\upsilon,t)$ such that, for any $n\ge N_0$,
\begin{equation}\label{resultth1}
 \sup_{\te_0\in\ell_0[s_n]} 
\FDR(\theta_0,\vphi^{\mbox{\tiny $\l$-val}}) 
\le C\frac{\log\log{n}}{(\log{n})^{\kappa/2}}. 
\end{equation} 
\end{theorem}

Theorem~\ref{th1} is proved in Section~\ref{sec:proof:th1}. The proof relies mainly on two different arguments: first, 
a careful analysis of the concentration of $\hat w$, which requires to distinguish between two regimes (weak/moderate or strong signal, basically); second, the study of the FDR of the $\l$-value procedure taken at some sparsity parameter $w$ (not random but depending on $n$) in each of these two regimes. 
This requires to analyse the mathematical behavior of a number of functions of $w,\theta_0$, uniformly over a wide range of possible sparsities, which is one main technical difficulty of our results. In particular, the concentration of $\hat w$ is obtained uniformly over all sparse vectors with polynomial sparsity, without any strong-signal or self-similarity-type assumption, as would typically be the case for obtaining adaptive confidence sets. Such assumptions would of course simplify the analysis significantly, but the point here is precisely that a uniform FDR control is possible for rate-adaptive procedures without any assumption on the true sparse signal. The uniform concentration of $\hat w$ is expressed implicitly and requires  sharp estimates, contrary to rate results for which a concentration in a range of values is typically sufficient. In particular, some of our lemmas in the supplementary file \cite{CR2017supp} are refined versions of lemmas in \cite{js04}.

As a corollary, \eqref{resultth1} entails
\begin{align*}
\varlimsup_n \sup_{\te_0\in\ell_0[s_n]}
\FDR(\theta_0,\vphi^{\mbox{\tiny $\l$-val}})&  =0,
\end{align*}
and this for any chosen threshold $t\in (0,1)$ in $\vphi^{\mbox{\tiny $\l$-val}}$. 
From a pure $\alpha$-FDR controlling point of view, 
while making a vanishing small proportion of errors is obviously desirable, it implies that $\vphi^{\mbox{\tiny $\l$-val}}$ is, as far as the FDR is concerned, somewhat conservative, in the sense that it does not spend all the allowed type I errors ($0$ instead of $\alpha$) and thus will make too few (true) discoveries at the end. It turns out that in the present setting 
 $\ell$-values are not quite on the ``exact"  scale for FDR control. 
An alternative is to consider the $q$-value scale, as we now describe.\\

\begin{center}
\fbox{\begin{minipage}{0.9\textwidth}
   \begin{center}
{ Algorithm $\,${\tt EBayesq} }
   \end{center}
{\tt Input:} $X_1,\ldots,X_n$, slab prior $\ga$, target confidence $t$\\
{\tt Output:} BMT procedure $\vphi^{\mbox{\tiny $q$-val}}$
   \begin{enumerate}
     \item  Find the maximiser $\hat w$ given by \eqref{defw}.
      \item Compute $\hat{q}_i(X)=q(X_i;\hat w,g)$.
      \item Return, for $1\le i\le n$,
      \begin{equation}
       \vphi_i^{\mbox{\tiny $q$-val}}=\ind{\hat{q}_i(X)\le t}.
              \label{def:ebayesq}
      \end{equation} 
    \end{enumerate}
\end{minipage}}\\
\end{center}
\vp

We also consider the following variant of the procedure {\tt EBayesq}, which is mostly the same, except that it does not allow for too small estimated weight $\hat w$. Set, for $L_n$ tending slowly to infinity, 
\begin{equation} \label{defom}
 \omega_n = \frac{L_n}{n \overline{G}(\sqrt{2.1\log{n}})}. 
\end{equation} 
For instance, for $\ga$ Cauchy or quasi-Cauchy, we have $\omega_n\asymp (L_n/n)\sqrt{\log{n}}$ while for $\ga$ Laplace(1) we have $\omega_n\asymp (L_n/n)\exp\{C\sqrt{\log{n}}\}$.
\vp

\begin{center}
\fbox{\begin{minipage}{0.9\textwidth}
   \begin{center}
{ Algorithm $\,$\EBayesqseuillage }
   \end{center}
{\tt Input:} $X_1,\ldots,X_n$, slab prior $\ga$, target confidence $t$, sequence $L_n$\\
{\tt Output:} BMT procedure $\vphiqvalseuillage$
   \begin{enumerate}
     \item[1.-2.]  Same as for  {\tt EBayesq}, returning $\hat q_i(X)$.
      \item[3.] Return, for $1\le i\le n$, and $\omega_n$ as in \eqref{defom},
      \begin{equation}
       \vphiqvalseuillage_i=\ind{\hat {q}_i(X)\le t}\ind{\hat w>\omega_n}.
       \label{def:ebayesQ}
      \end{equation} 
    \end{enumerate}
\end{minipage}}\\
\end{center}
\vp

\begin{theorem}\label{th2}
Consider the same setting as Theorem~\ref{th1}.
Then the algorithm {\tt EBayesq}  produces the BMT procedure $\vphi^{\mbox{\tiny $q$-val}}$  in \eqref{def:ebayesq} that satisfies the following:  there exists a constant $C=C(\gamma,\upsilon)$ such that for any $t\le 3/4$, there exists an integer $N_0=N_0(\gamma,\upsilon,t)$ such that, for any $n\ge N_0$,
\begin{align}\label{result1th2}
  \sup_{\te_0\in\ell_0[s_n]} 
\FDR(\theta_0,\vphi^{\mbox{\tiny $q$-val}}) & \le  C t\log(1/t). 
\end{align}
 In addition, the algorithm \EBayesqseuillage  produces the BMT procedure $\vphiqvalseuillage$  in \eqref{def:ebayesQ} that satisfies, for $\omega_n$ as in \eqref{defom} with $L_n\to \infty$, $L_n\le \log{n}$, $t\le 3/4$ and $C, N_0$ as before (but with possibly different numerical values), for any $n\ge N_0$,
\begin{align}\label{result2th2}
  \sup_{\te_0\in\ell_0[s_n]} 
\FDR(\theta_0,\vphiqvalseuillage) & \le  C t. 
\end{align}
\end{theorem}

The proof of Theorem~\ref{th2} is technically close to that of Theorem~\ref{th1} and is given in Section~\ref{sec:proof:th1}, see also Section~\ref{sec:heu} for an informal heuristic that serves as guidelines for the proof. 
The statements of Theorem~\ref{th2} are however of different nature, because the $q$-value threshold $t$ appears explicitly in the bounds \eqref{result1th2}-\eqref{result2th2}, that do not vanish as $n$ tends to infinity. 

The two bounds \eqref{result1th2} and \eqref{result2th2}  differ from a $\log(1/t)$ term, which may become significant for small $t$.  
This term appears in the  case where the signal is weak (only few rejected nulls), for which the calibration $\hat w$ is slightly too large. This may not be the case using a different type of sparsity--adaptation, or a different estimate $\hat w$.  Indeed, this phenomenon disappears when using \EBayesqseuillage,  
since $\hat w$ is then set to $0$ when it is not large enough, in which case the FDR control is shown to be guaranteed, and we retrieve a dependence in terms of a constant times the target level $t$.
 
A consequence of Theorem~\ref{th2} is that an $\alpha$--FDR control can be achieved with {\tt EBayesq}/\EBayesqseuillage  procedures by taking $t=t(\alpha)$ sufficiently small (although not tending to zero). Again, it is important to know how small the constant $C>0$ can be taken in \eqref{result1th2} and \eqref{result2th2}. 
When the signal is strong enough, the following result shows that $C=1$ and the $\log(1/t)$ factor can be removed in \eqref{result1th2}.

Let us first introduce a set $\cL_0[s_n]$ of `large' signals, for arbitrary $a>1$, 
\begin{align} 
\cL_0[s_n]= \bigg\{ \theta\in \l_0[s_n] \::\: |\theta_{i}| \geq &\:a\sqrt{2\log (n/s_n)} \text{ for } i\in S_{\te},\ \ |S_{\te}|=s_n\bigg\}.\label{L0}
\end{align}  

\begin{theorem}\label{theorem:psharpFDRcontrol}
Consider $\cL_0[s_n]=\cL_0[s_n;a]$ defined by \eqref{L0} with an arbitrary $a>1$, for $s_n \rightarrow \infty$ and $s_n\le n^{\upsilon}$ for some $\upsilon\in(0,1)$. 
Assume that $\ga$ is a unimodal symmetric slab density that satisfies \eqref{asump1}--\eqref{asump3} with $\kappa$ as in \eqref{asump2}. Then, for any pre-specified level $t\in (0,1)$, 
 {\tt EBayesq}  produces the BMT procedure $\vphi^{\mbox{\tiny $q$-val}}$  in \eqref{def:ebayesq} 
such that
\begin{align}
\lim_n \sup_{\theta_0 \in \cL_0[s_n]} \FDR(\theta_0,\vphi^{\mbox{\tiny $q$-val}}) = \lim_n \inf_{\theta_0 \in \cL_0[s_n]} \FDR(\theta_0,\vphi^{\mbox{\tiny $q$-val}}) = t\:. \label{FDRcontrolsharp}
\end{align}
In addition,  \EBayesqseuillage with $L_n\to \infty$, satisfies the same property whenever $s_n/n\ge 2\omega_n$, for $\omega_n$ as in \eqref{defom}, which is in particular the case if $s_n$ grows faster than a given power of $n$ and $L_n\le \log{n}$.
\end{theorem}

Theorem~\ref{theorem:psharpFDRcontrol}, although focused on a specific regime, shows that empirical Bayes procedures are able to produce an asymptotically exact FDR control. 
Again, this may look surprising at first, as the prior slab density $\gamma$ is not particularly linked to the true value of the parameter $\theta_{0}\in \cL_0[s_n]$ in \eqref{FDRcontrolsharp}. This puts forward a strong adaptive property of the spike and slab prior for multiple testing.

We conclude this section by giving results on the type II risk of the introduced multiple testing procedures. 
This is done by controlling the average 
number of false negatives (also called false non-discoveries) among the non-zero coordinates,
which is called below False Negative Rate (FNR).
For a given multiple testing procedure $\vphi$, following \cite{erychen}, we let \begin{equation} \label{deffnr}  
\FNR(\theta_{0},\vphi) 
= \E_{\theta_0}\left[\frac{\sum_{i=1}^n \ind{\theta_{0,i}\neq 0} (1-\vphi_i(X))}{1\vee \sum_{i=1}^n \ind{\theta_{0,i}\neq 0}}\right].
\end{equation}
Clearly, in the present setting, controlling this quantity is only possible under signal strength  assumptions.  
Below, we provide such  a control over 
the class $\cL_0[s_n]$ defined in \eqref{L0} above, and for the procedures $\vphi^{\mbox{\tiny $\ell$-val}}$ and $\vphi^{\mbox{\tiny $q$-val}}$ (results for $\vphiqvalseuillage$ are the same as for $\vphi^{\mbox{\tiny $q$-val}}$ under the conditions of Theorem \ref{theorem:psharpFDRcontrol} and are omitted).
\begin{theorem}\label{theorem:power}
Let $t\in (0,1)$ be any pre-specified level. 
Consider the setting and notation of Theorem \ref{theorem:psharpFDRcontrol} and recall the $\ell$-values procedure from Theorem~\ref{th1}. The BMT procedures $\vphi^{\mbox{\tiny $\ell$-val}}$ and $\vphi^{\mbox{\tiny $q$-val}}$ verify
\begin{align}\label{power_qval} 
\lim_n \sup_{\theta_0 \in \cL_0[s_n]} 
 \FNR(\theta_0,\vphi^{\mbox{\tiny $\ell$-val}}) = 
 \lim_n \sup_{\theta_0 \in \cL_0[s_n]} 
 \FNR(\theta_0,\vphi^{\mbox{\tiny $q$-val}})
& = 0\:. 
\end{align}
\end{theorem}

\begin{corollary}\label{cor:power}
In the setting of Theorem \ref{theorem:power}, for any pre-specified level $t\in (0,1)$, the multiple testing procedures $\vphi^{\mbox{\tiny $\ell$-val}}$ and $\vphi^{\mbox{\tiny $q$-val}}$ satisfy 
\begin{align}\label{sumerr_qval} 
\ \lim_n & \left[\sup_{\theta_0 \in \ell_0[s_n]} \FDR(\theta_0,\vphi^{\mbox{\tiny $\ell$-val}}) 
+ \sup_{\theta_0 \in \cL_0[s_n]} \FNR(\theta_0,\vphi^{\mbox{\tiny $\ell$-val}}) \right]
 = 0\:. \\
\ \lim_n & \sup_{\theta_0 \in \cL_0[s_n]} \FDR(\theta_0,\vphi^{\mbox{\tiny $q$-val}}) = t,\quad 
\lim_n \sup_{\theta_0 \in \cL_0[s_n]} \FNR(\theta_0,\vphi^{\mbox{\tiny $q$-val}}) 
& = 0\:. 
\end{align}
\end{corollary}

Let us consider, similarly to \cite{erychen}, the (multiple testing) classification risk $\mathfrak{R}(\theta_0,\vphi) = \FDR(\theta_0,\vphi)+\FNR(\theta_0,\vphi)$ for any $\theta_0\in\R^n$ and procedure $\vphi$. 
It follows from Corollary~\ref{cor:power} that  for any  $a>1$ and $t<1$, 
\begin{equation*}
\lim_n \sup_{\theta_0 \in \cL_0[s_n;a]} \{\mathfrak{R}(\theta_0,\vphi^{\mbox{\tiny $\ell$-val}})\}=0,\qquad 
\lim_n \sup_{\theta_0 \in \cL_0[s_n;a]} \{\mathfrak{R}(\theta_0,\vphi^{\mbox{\tiny $q$-val}})\}=t,
\end{equation*}
so the procedure $\vphi^{\mbox{\tiny $\ell$-val}}$ is consistent for this risk on this range of signals, while $\vphi^{\mbox{\tiny $q$-val}}$ controls it at level $t<1$.

We can legitimately ask if this property is optimal in some sense. We establish below that  the  
classification task is impossible (that is, the risk is at least $1$)  below the boundary $\sqrt{2\log (n/s_n)}$, at least over a fairly large class of procedures.

Define the class $\cC$ of two-sided thresholding-based multiple testing procedures $\varphi$ of the form
$$
\varphi_i(X)=\ind{X_i\geq \tau_1(X) \mbox{ or } -X_i\geq \tau_2(X)},\:\:1\leq i\leq n,
$$
for some measurable $\tau_1(X), \tau_2(X)\geq 0$. The following result adapts a result of \cite{erychen} to the two-sided context. 

\begin{proposition}\label{prop:lb}
Consider $\cL_0[s_n]=\cL_0[s_n;a]$ defined by \eqref{L0} with an arbitrary $a<1$, for $s_n \rightarrow \infty$ and $s_n\le n^{\upsilon}$ for some $\upsilon\in(0,1)$. Consider the class of two-sided thresholding-based multiple testing procedures $\mathcal{C}$ defined above. Then, for $\mathfrak{R}$ is the FDR$+$FNR classification risk defined above,
\begin{align*}
\varliminf_n \inf_{\varphi\in \mathcal{C}} \sup_{\theta_0\in \mathcal{L}_0[s_n;a]} 
\mathfrak{R}(\theta_0,\vphi)
\geq 1.
\end{align*}
\end{proposition}

The proof of Proposition~\ref{prop:lb} is given in Section~\ref{sec:lb}. Let us underline that therein, much sharper results are provided, which allow to derive explicit convergence rates for the classification impossibility for a signal strength just below $\sqrt{2\log (n/s_n)}$.

Finally, we have established that the procedures $\vphi^{\mbox{\tiny $\ell$-val}}$ and $\vphi^{\mbox{\tiny $q$-val}}$ both achieve asymptotically the optimal classification boundary $\sqrt{2\log (n/s_n)}$: they asymptotically control the risk on $\cL[s_n;a]$ for arbitrary $a>1$ (at levels $0$ and $t$ respectively), while any such control is impossible if $a<1$.

\begin{remark}\label{remarkg}
Our results can be extended to the case where $g$ is not of the form  \eqref{equg} (that is, not necessarily of the form of a convolution with the standard gaussian), but satisfies some weaker properties, see Section~\ref{sec:assumg}. This extended setting corresponds to  a ``quasi-Bayesian" approach where the $\ell$-values (resp. $q$-values) are directly given by the formulas \eqref{lvalues} (resp. \eqref{qvalues}), without specifying a slab prior $\gamma$.
\end{remark}

\section{Numerical experiments}\label{sec:simu}

In this section, our theoretical findings are illustrated via numerical experiments. 
A motivation here is also to evaluate how the parameters $s_n$, $\theta_0 \in \ell_0[s_n]$, and the hyper-parameter $\gamma$ (or $g$) affect the FDR control, in particular the value of the constant in the bound of Theorem~\ref{th2}.  

For this we consider $n=10^4$, $s_n\in\{10,10^2,10^3\}$ and the following two possible scenarios for $\theta_0\in \ell_0[s_n]$:
\begin{itemize}
\item constant alternatives: $\theta_{0,i}=\mu$ if $1\leq i \leq s_n$ and $0$   otherwise; or
\item randomized alternatives:  $\theta_{0,i}$ i.i.d. uniformly distributed on $(0,2\mu)$ if $1\leq i \leq s_n$ and $0$   otherwise.
\end{itemize}
The parameter range for $\mu$ is taken equal to $\{0.01,0.5,1, 2,\dots,10\}$.
The marginal likelihood estimator $\hat{w}$ given by \eqref{defw} is computed by using a modification of the function \texttt{wfromx} of the package \texttt{EbayesThresh} \cite{js05}, that accommodates the lower bound $1/n$ in our definition (instead of $w_n=\zeta^{-1}(\sqrt{2\log n})$, see \eqref{zeta}, in the original version).
The parameter $\gamma$ is either given by the quasi-Cauchy prior \eqref{gammaquasiCauchy}-\eqref{gquasiCauchy} or by the Laplace prior of scaling parameter $a=1/2$ (see Remark~\ref{rem:explicit} for more details).  
For any of the above parameter combinations, the FDR of the procedures {\tt EBayesL}, {\tt EBayesq} (defined in Section~\ref{sec:EBw}) is evaluated empirically via $2000$ replications.

Figure~\ref{fig1} displays the FDR of the procedures  {\tt EBayesL} ($\ell$-values) and {\tt EBayesq} ($q$-values). Concerning {\tt EBayesL},  in all situations, the FDR is small while not exactly equal to the value $0$, which seems to indicate that the bound found in Theorem~\ref{th1} is not too conservative. Moreover, the   quasi-Cauchy version seems more conservative than the Laplace version, which corroborates  our theoretical findings (in our bound \eqref{resultth1}, we have the factor $(\log n)^{-1}$ for quasi-Cauchy and $(\log n)^{-1/2}$ for Laplace). 
As for {\tt EBayesq}, when the signal is large, the FDR curves are markedly close to the threshold value $t$ when $s_n/n$ is small, which is in line with Theorem~\ref{theorem:psharpFDRcontrol}. However, for a weak sparsity $s_n/n=0.1$, the FDR values are slightly inflated (above the threshold $t$), which seems to indicate that  the asymptotical regime is not yet reached for this value.
Looking now at the whole range of signal strengths, one notices the presence of a `bump` in the regime of intermediate values of $\mu$, especially for the Laplace prior. However, this bump seems to disappear when $s_n/n$ decreases. 
We do not known presently whether this bump is vanishing with $n$ or if this corresponds to a necessary additional constant $C=C(\gamma,\upsilon)>1$ (or $\log(1/t)$) in the achieved FDR level, but we suspect that this is related to the fact that the intermediate regime was the most challenging part of our proofs.
Overall, the Cauchy slab prior seems to have a particularly suitable behavior. This was not totally surprising for us as it already showed more stability than the Laplace prior in the context of estimation with the full empirical Bayes posterior distribution, as seen in \cite{cm17}.

Finally, we provide additional experiments in the supplement, see Section~\ref{sec:addnum}. The findings can be summarized as follows.
\begin{itemize}
\item The curves behave qualitatively similarly for randomized alternatives (second scenario).
\item The procedure  \EBayesqseuillage (with $L_n=\log\log n$) has a global behavior similar to {\tt EBayesq}, with  more conservativeness for weak signal (as expected).
\item It is possible to uniformly improve \EBayesqseuillage by considering the following modification (named \EBayesqplus below): 
 if $w\leq \omega_n$, instead of rejecting no null, \EBayesqplus performs a standard Bonferroni correction, that is, rejects the $H_{0,i}$'s such that $p_i(X)\leq t/n$. Note that a careful inspection of the proof of Theorem~\ref{th2} (\EBayesqseuillage part) shows that the bound \eqref{result2th2} is still valid for \EBayesqplus\hspace{-0.05cm}.
\end{itemize}

\begin{figure}
\begin{tabular}{ccc}
\vspace{-0.5cm}
&quasi-Cauchy & Laplace\\
\vspace{-1cm}
\rotatebox{+90}{\hspace{3cm}$s_n/n=0.1$} &\includegraphics[scale=0.35]{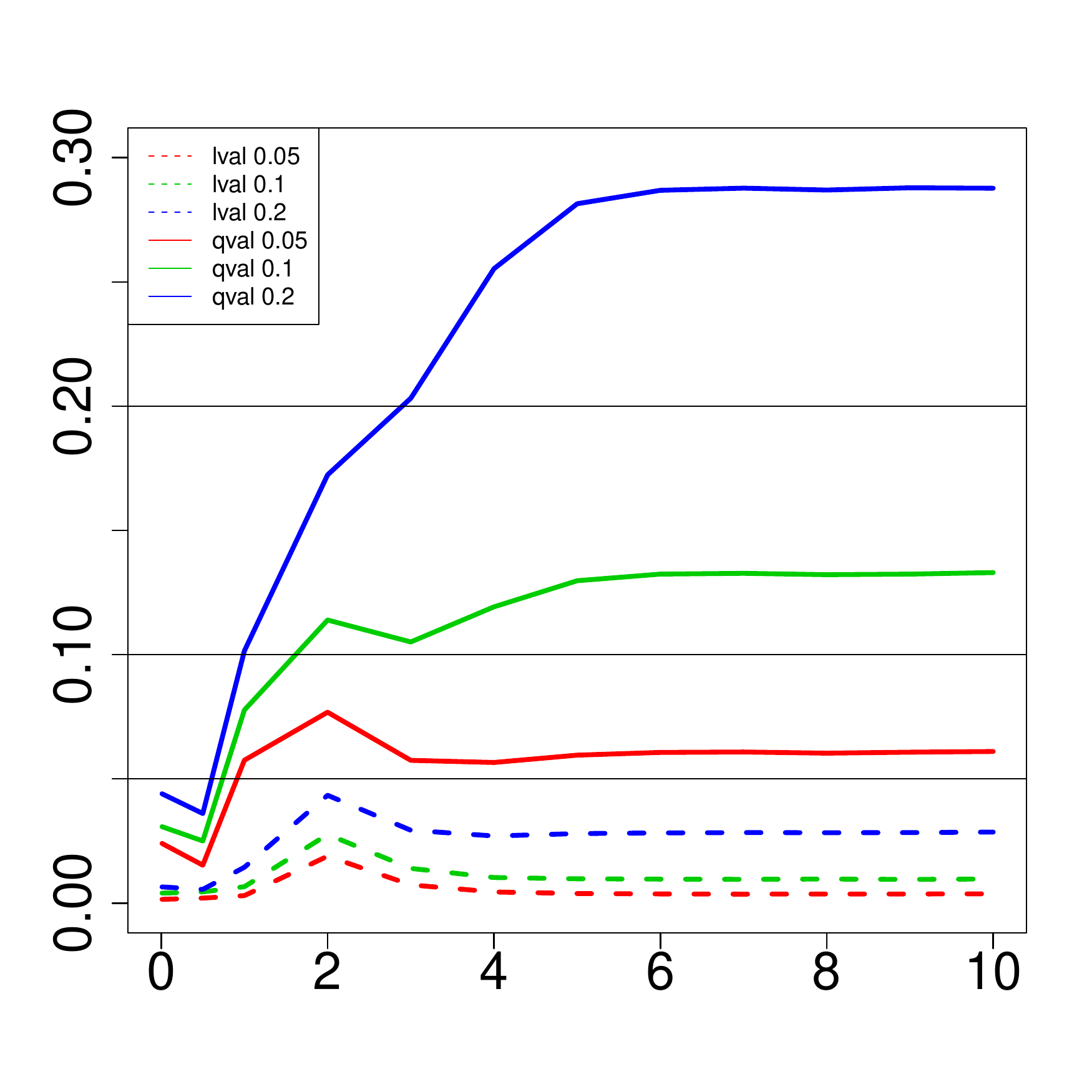}
&\includegraphics[scale=0.35]{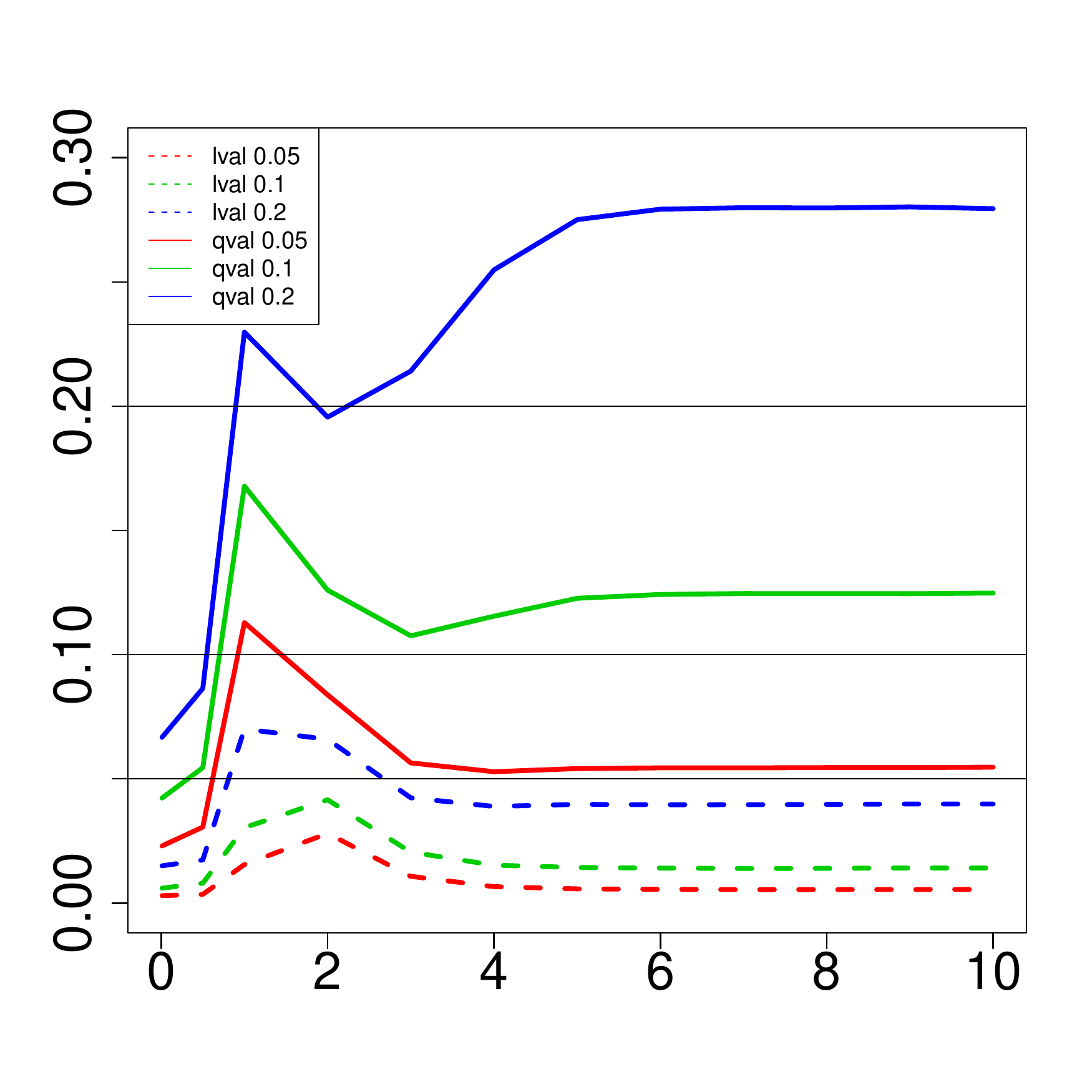}
\\
\vspace{-1cm}
\rotatebox{+90}{\hspace{3cm}$s_n/n=0.01$} &\includegraphics[scale=0.35]{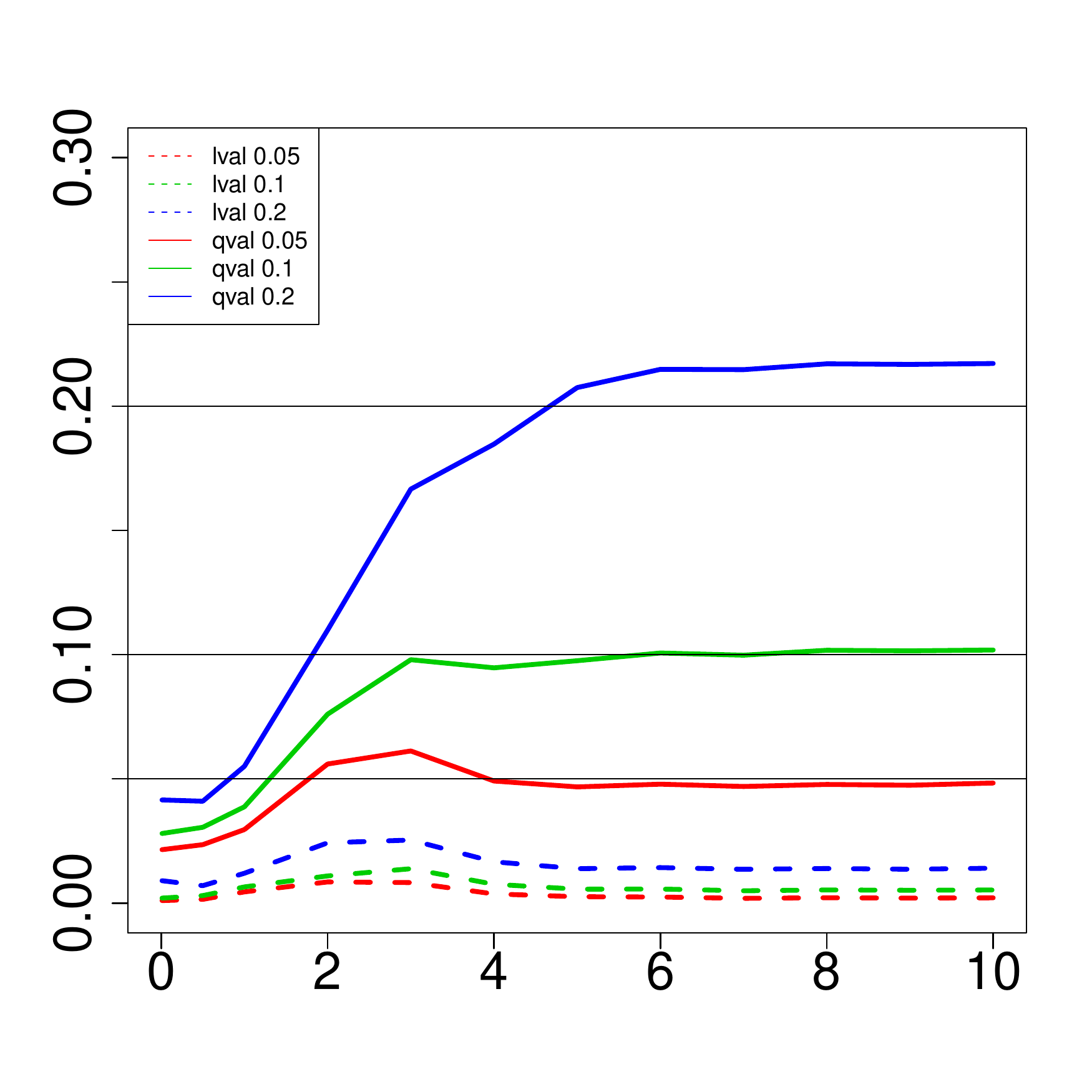}
&\includegraphics[scale=0.35]{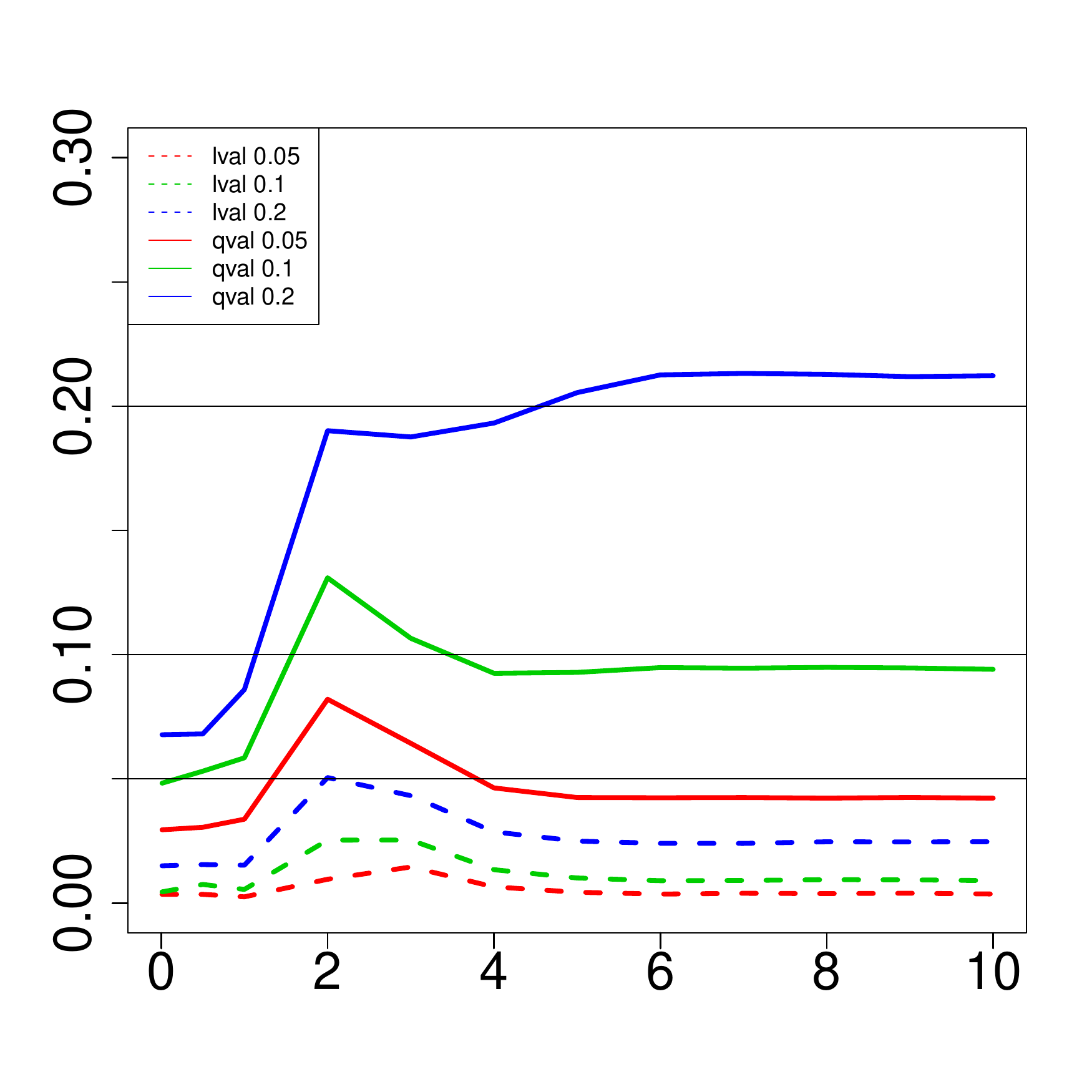}
\\
\vspace{-0.5cm}
\rotatebox{+90}{\hspace{3cm}$s_n/n=0.001$} &\includegraphics[scale=0.35]{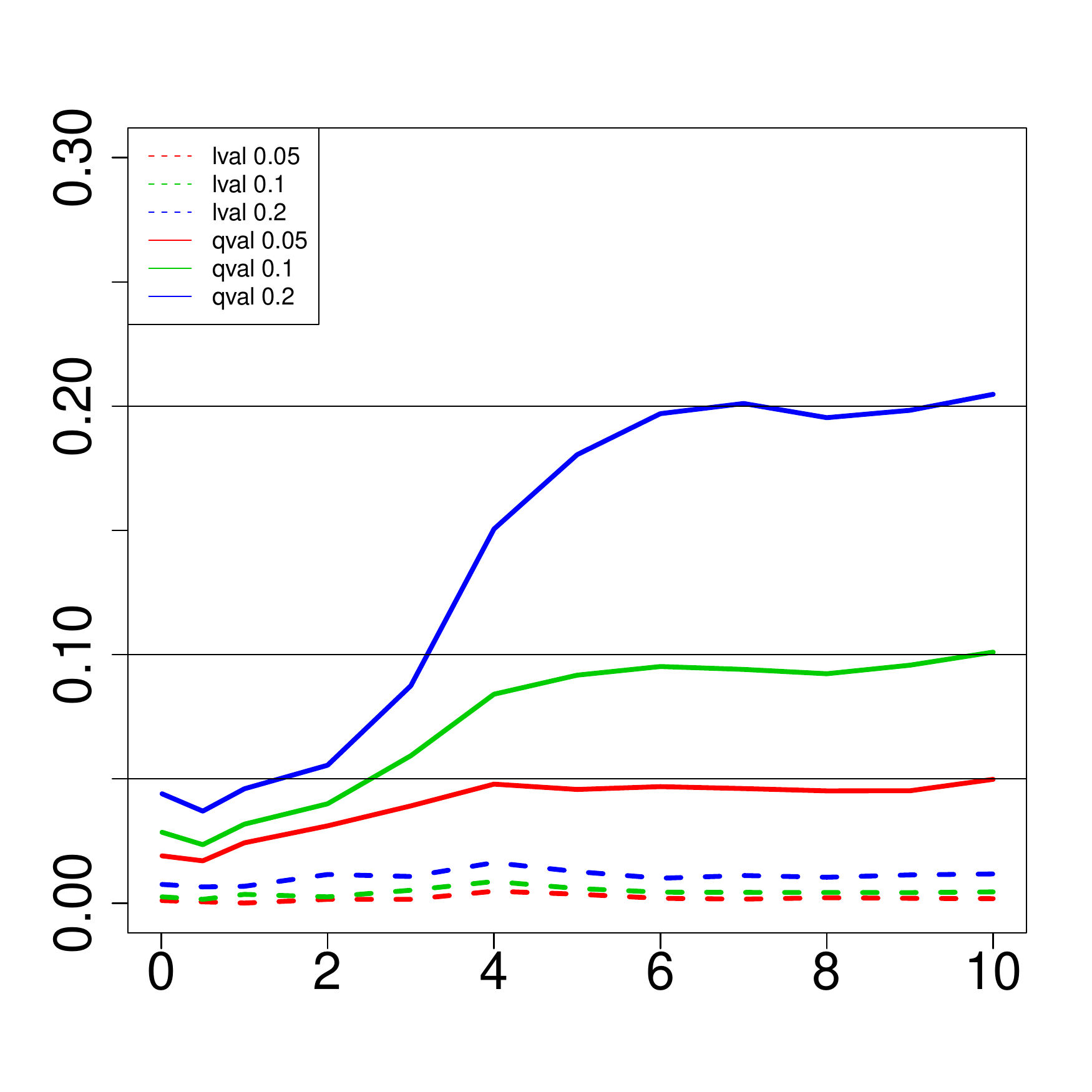}
&\includegraphics[scale=0.35]{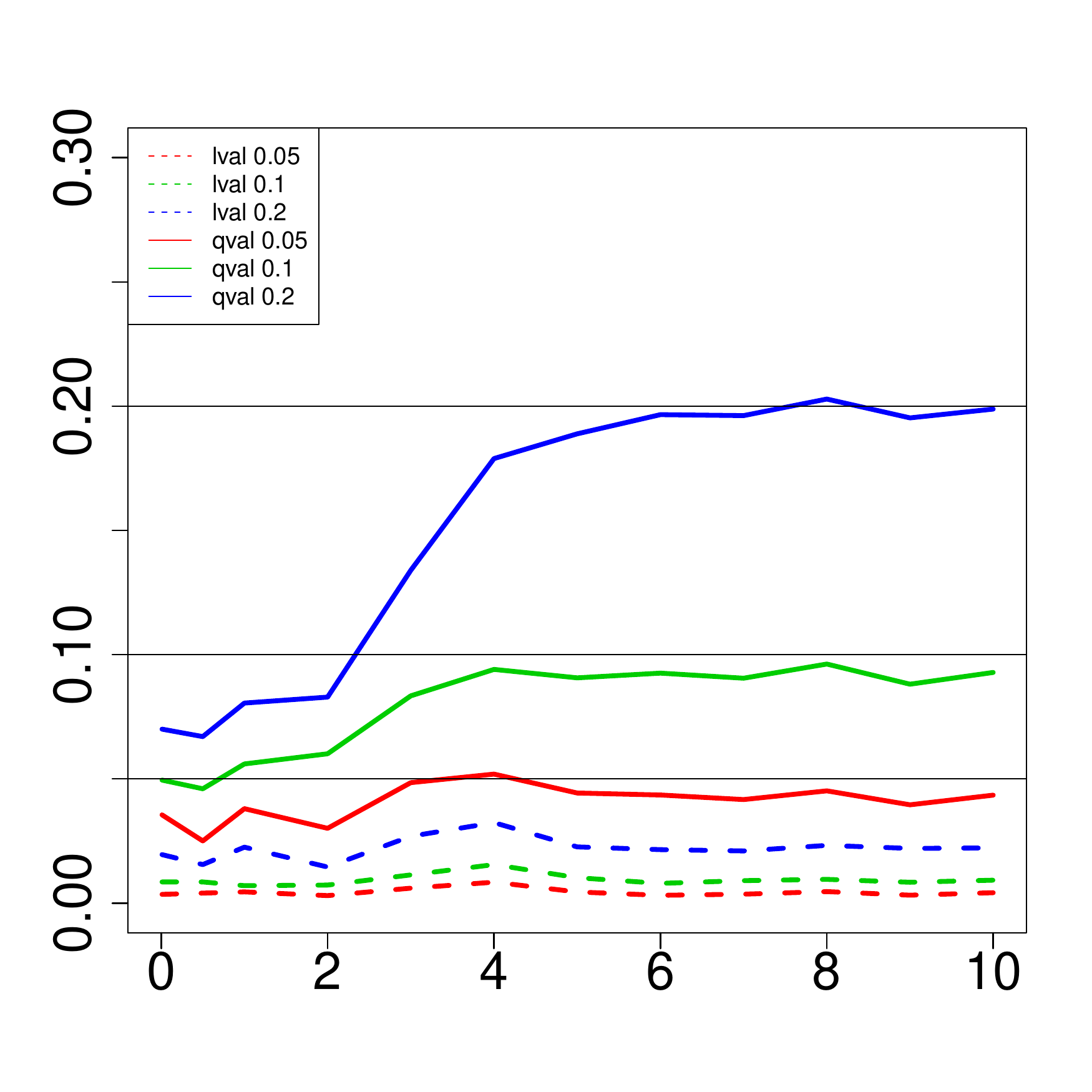}
\end{tabular}
\caption{\label{fig1}
FDR of {\tt EBayesL} and {\tt EBayesq} procedures with threshold $t\in\{0.05,0.1,0.2\}$; $n=10,\,000$; $2000$ replications; alternative all equal to $\mu$ (on the $X$-axis).}
\end{figure}

\section{Further procedures} \label{sec:other} Two other popular Bayesian multiple testing procedures are now briefly discussed as well as their links to both $\ell$- and $q$-value procedures. 

\subsection{MCI procedures}\label{rem:MCI} Given a posterior distribution, one may test the presence of signal on a coordinate by looking at whether $0$ belongs to a certain interval on this coordinate with high posterior probability. We refer to any such procedure based on marginal {\em credible} intervals as MCI procedure for short. Let $\cI_i(X)=\cI_i(t,X)$ be an interval with credibility at least $1-t$ for coordinate $i$ for the empirical Bayes posterior, then by definition 
$\Pi[\cI_i(X) \given X] \ge 1- t$.
Hence, $0\notin \cI_i(X)$ implies, for $\hat\ell_i(X)$ as in \eqref{def:ebayesl},
\[  \hat\ell_i(X) = \Pi[\te_i=0 \given X]  \le \Pi[\te_i\notin \cI_i(X) \given X]  \le  1-(1-t)=t.\]
One deduces that any MCI procedure at level $1-t$ is more conservative than the $\ell$-value procedure at level $t>0$. 
For a natural quantile-based MCI procedure and spike-and-slab priors, it can be shown that the converse is also true up to taking a slightly lower level, say $t-\epsilon$, any $\epsilon>0$, for the $\ell$-value procedure, see Section \ref{sec:MCIanalysis}. This property means that in the present setting this quantile-based MCI procedure is essentially equivalent to the $\ell$-value procedure, which leads to Theorem \ref{thm-mci} below, proved in Section \ref{sec:MCIanalysis}.
 
Let $z_{i}^t(X)$ denote the quantile at level $t\in(0,1)$ of the marginal empirical Bayes posterior distribution of the $i$-th coordinate, 
\[ z_i^t(X)=\inf \left\{z\in\RR:\ \Pi_{\hat w,\ga}[\te_i\le z\given X] \ge t\right\},\]
and define a procedure $\vphi^m$ at level $t$  as follows. For $i=1,\ldots,n$, 
\begin{align}
\vphi^m_i & = \ind{0\notin [z_{i}^t(X),z_{i}^{1-t}(X)] }, \qquad \qquad t\in(0,1/2),  \label{mval} \\
& = \ind{0< z_{i}^{t}(X)} + \ind{z_{i}^{1-t}(X)< 0}=\ind{0\notin \cI_i(X)},\nonumber
\end{align}
where $\cI_i(X)=(z_i^t(X),+\infty)$ if $X_i\ge0$, and $\cI_i(X)=(-\infty,z_i^{1-t}(X))$ if $X_i<0$. Note that such an interval $\cI_i(X)$ is an MCI at level $1-t$, as its credibility is indeed $1-t$ in both cases.  

\begin{theorem} \label{thm-mci}
For $t< 1/2$, under the assumptions of Theorem \ref{th1}, the conclusion of Theorem \ref{th1} holds for the MCI procedure $\vphi^m$ at level $t$.
\end{theorem}
In particular, the FDR of the $\vphi^m$ procedure goes to $0$ uniformly over sparse vectors. Control of FDR+FNR can be obtained as well, in a similar way as for the $\ell$-value procedure in  Section \ref{sec:EBw}. The procedure $\vphi^m$ can be shown to be very close to the $\ell$-values procedure at level $t$, see Section \ref{sec:MCIanalysis} for a justification and the proof of Theorem \ref{thm-mci}. 

\subsection{Averaging $\ell$-values}\label{rem:SC}

Another type of procedures, advocated by Sun and Cai in a series of works  (e.g., \citep{SC2007,SC2009}), are those based on averaged $\ell$-values. In the Bayesian spike and slab context, it gives rise to the procedure, denoted here by $\SC$ (at a target confidence $t$), that rejects the $\hat{k}$ smallest $\ell$-values, where $\hat{k}$ is the maximum of the $k$ such that $k^{-1}\sum_{k'=1}^k \hat{\ell}_{(k')}(X)\leq t$, where $\hat{\ell}_{(1)}(X)\leq \dots \leq \hat{\ell}_{(n)}(X)$ are the ordered elements of $\{\hat{\ell}_i(X),1\leq i\leq n \}$, the latter being the empirical Bayes $\ell$-values used in  {\tt EBayesL} (Section~\ref{sec:EBw}). We provide insight into the behavior of $\SC$ in Section~\ref{sec:SCanalysis}, both  theoretically and numerically. In a nutshell, we observe a qualitative behavior similar to {\tt EBayesq}, with an FDR tending to $t$ under strong signal strength. Nevertheless, the convergence rate to the target level $t$ seems slow (decreasing at a logarithmic order in $n/s_n$), because of a specific remainder term, see Lemma~\ref{lem:final}.

\section{Discussion} \label{sec:disc}

Our results show that spike and slab priors produce posterior distributions with particularly suitable multiple testing properties. One main challenge in deriving the results was to build bounds that are uniform over sparse vectors. 
We demonstrate that such a uniform control is possible up to a constant term away from the target control level. This constant is very close to $1$ in simulations, and can even be shown to be $1$ asymptotically for some subclass of sparse vectors. 

The results of the paper are meant as a theoretical validation of the common practical use of posterior-based quantities for (frequentist) FDR control. 
While the main purpose here was validation, it is remarkable that a uniform control of the FDR very close to the target level can be obtained for the  spike and slab BMT procedure in the present unstructured sparse high-dimensional model.

While many studies focused on controlling the Bayes FDR with Bayesian multiple testing procedures, 
this work 
paves the way for a frequentist FDR analysis of such procedures in different settings. 
In our study, the perhaps most surprising fact is how well marginal maximum likelihood estimation combines with FDR control under sparsity: as shown in our proof (and summarized in our heuristic) the score function is linked to a peculiar equation that makes perfectly the link between the numerator and the denominator in the FDR of the $q$-value--based multiple testing procedure. This phenomenon has not been noticed before to the best of our knowledge. 
We suspect that this link is only part of a more general picture, in which the concentration of the score process  in general sparse high dimensional models plays a central role. 
While this exceeds the scope of this paper, generalizing our results to such settings  is a very interesting direction for future work.

\section{Preliminaries for the proofs} \label{sec:prelim}

\subsection{Working with general $g$}\label{sec:assumg}

As noted in Remark~\ref{remarkg}, the results of Theorems~\ref{th1}, \ref{th2} and \ref{theorem:psharpFDRcontrol} are also true under slightly more general assumptions, that do not impose that $g$ is coming from a $\gamma$ by a convolution product. 
Namely, let us assume 
\begin{align}\label{assumpgdebase} 
\mbox{
\begin{tabular}{c}$g$ is a positive, symmetric, differentiable density\\ that decreases on a vicinity of $+\infty$\end{tabular}
}
\end{align}
 ($g$ decreasing on a vicinity of $+\infty$ means that $x\to g(x)$ is decreasing for $x> M$, for a suitably large constant $M=M(g)$). Assume moreover that 
\begin{align} 
|(\log g)'(y)| & \le \La , \mbox{ for all $y\in\R$}, \:\: \La>0; \label{log-lip}  \\
 \ol{G}(y) & \asymp g(y) \:y^{\kappa-1},\quad\text{as } y\to\infty, \:\: \mbox{ for some }\kappa\in[1,2]; \label{tails}\\
y\in \R&\to (1+y^2)g(y)\mbox{ is bounded} ; \label{tails2}\\
g/\phi &\mbox{ is increasing on $[0,\infty)$ from $(g/\phi)(0)<1$ to $\infty$}; \label{increasing}
\end{align}
By Lemma~\ref{lemmetrivialisant}, it is worth to note that \eqref{increasing} implies
\begin{align} 
\ol{G}/\ol{\Phi} &\mbox{ is increasing on $[0,\infty)$ from $1$ to $\infty$.} \label{increasing2}
\end{align}

In the case where $g$ is of the form of a convolution with $\gamma$, see  \eqref{equg}, conditions \eqref{log-lip}, \eqref{tails} and \eqref{tails2} are easy consequences of the fact $g(y) \asymp \gamma(y)$ when $y\to\infty$  and condition \eqref{increasing} follows from the fact that for all fixed $u>0$,  the function $x\in [0,\infty)\to (\phi(x+u)+\phi(x-u))/\phi(x)$ is increasing, see Lemma~1 of \cite{js04} for a detailed derivation.

A consequence of \eqref{log-lip} is that $g$ and $\ol{G}$ have at least Laplace tails
\begin{align}
 g(y) &\geq  g(0)  e^{-\Lambda y},\:\:\: y\geq 0\label{tailsg};\\
  \ol{G}(y) &\geq  g(0) \Lambda^{-1} e^{-\Lambda y},\:\:\: y\geq 0\label{tails3}.
\end{align}

\subsection{BMT as thresholding-based procedures}\label{sec:BMT}

Recall the definitions \eqref{def:lvalproc} and \eqref{def:qvalproc}.
Let,  for any $w$ and $t$ in $[0,1)$,
\begin{align}
r(w,t)=\frac{wt}{(1-w)(1-t)}.  \label{equinterm}
\end{align}
The following quantity plays the role of threshold for $\ell$-values,
\begin{align} 
\xi = (\phi/g)^{-1} : (0,(\phi/g)(0)] \rightarrow [0,\infty),\label{equxi}
\end{align}
i.e. $\xi$ is the decreasing continuous inverse of $\phi/g$ (that exists thanks to \eqref{increasing}).
Simple algebra
 shows that for $w,t\in [0,1)$ with $r(w,t)\leq \phi(0)/g(0)$,
\begin{align} 
\l_i(X)\leq t &\:\Leftrightarrow\: |X_i|\geq \xi(r(w,t)). \label{liandxi}
\end{align}
When $u$ becomes small, the order magnitude of $\xi(u)$ is given in Lemma~\ref{prop:BMTlval}: $\xi(u)$ slightly exceeds $\left(-2\log u\right)^{1/2}$ but not by much, which comes from the fact that $g$ has heavy tails.

Another quantity close to $\xi$ we shall use in the sequel is the threshold $\zeta$ introduced in \cite{js04} and defined as, for any $w\in(0,1]$,
\begin{equation} \label{zeta}
 \zeta(w) = \beta^{-1}(w^{-1}).
\end{equation}
Combining the definitions leads, see \eqref{zetatoxi} for details, to $\zeta(w)=\xi(w/(1+w))$ and $\xi(w)\le \zeta(w)$. 
Similarly,  let us introduce a threshold for $q$-values as
\begin{align} 
\chi = (\ol{\Phi}/\ol{G})^{-1} : (0,1] \rightarrow [0,\infty),\label{equchi}
\end{align}
which is the decreasing continuous invert of $\ol{\Phi}/\ol{G}$ (that exists thanks to \eqref{increasing2}). For all $w\in [0,1)$ and $t\in [0,1)$ with $r(w,t)\leq 1$,
\begin{align} 
q_i(X)\leq t &\:\Leftrightarrow\: |X_i|\geq \chi(r(w,t)). \label{liandchi}
\end{align}

Lemma~\ref{prop:BMTqval} shows that, for small $u$, the order of magnitude of  $\chi(u)$ is slightly more than $\ol{\Phi}^{-1} \left(u\right)$ but not by much, which comes from the fact that $\ol{G}$ has heavy tails. Also, Lemma \ref{lem:qlvalue} together with \eqref{liandxi}-\eqref{liandchi} imply 
\begin{equation} \label{xichi}
\chi(u) \le \xi(u), \qquad \mbox{for $u\le 1$}.
\end{equation}

\subsection{Single type I error rates}

The single type I error rates of our procedures are evaluated by the following result (proved in Section~\ref{sec:prop:typeIerror}).

\begin{proposition}\label{prop:typeIerror}
Consider any function $g$ satisfying the assumptions of Section~\ref{sec:assumg}.
Consider
$r(\cdot,\cdot)$ as in \eqref{equinterm}, $\xi$ as in \eqref{equxi} and $\chi$ as in \eqref{equchi}. Then the following bounds hold. 
 For all $t,w$ such that $r(w,t)\leq (\phi/g)(0)$,
\begin{align}
\P_{\theta_0=0}(\l_i(X)\leq t) &\leq 2 r(w,t)\frac{g( \xi(r(w,t)))}{\xi(r(w,t))}.\label{ubtypeIerrorlvalue}
\end{align}
Also, for all $t,w$ such that $r(w,t)\leq (\phi/g)(1)$, 
\begin{align}
\P_{\theta_0=0}(\l_i(X)\leq t) &\geq r(w,t)\frac{g( \xi(r(w,t)))}{\xi(r(w,t))}. \label{lbtypeIerrorlvalue}
\end{align}
For $q$-values, we have, for all $t,w$ such that $r(w,t)\leq 1$,
\begin{align}
\P_{\theta_0=0}(q_i(X)\leq t) & = r(w,t)\: 2\ol{G}\left(\chi(r(w,t))\right).
\label{eqtypeIerrorqvalue}
\end{align}
\end{proposition}

As a result, for a fixed $w$, we see that heavier tails of $g$ result in larger type I error rate. This is well--expected, as the heavier the tails of $g$, the more mass the prior puts on large values.

\section{Proof of the main results}\label{sec:proof:th1}

\subsection{Notation}

The following moments are useful when studying the score function $\cS$. Let us set
\begin{equation} \label{mtilde}
\tilde m(w) = - \E_0\beta(X,w)= \int_{-\infty}^{\infty} \be(t,w)\phi(t)dt
\end{equation}
and further denote 
\begin{align} 
m_1(\ta,w) & = \E_\ta[\be(X,w)] = \int_{-\infty}^{\infty} \be(t,w)\phi(t-\tau)dt. \label{m1}\\
m_2(\ta,w) & = \E_\ta[\be(X,w)^2]= \int_{-\infty}^{\infty} (\be(t,w))^2\phi(t-\tau)dt. \label{m2}
\end{align}
These expectations are well defined and studied in detail in Appendix~\ref{sec:moment}, refining previous results established in \cite{js04}.

In order to study the FDR of a procedure $\vphi$, we introduce the notation
\begin{align} \label{not_VS}
V(\vphi) & = \sum_{i:\ \te_{0,i}=0} \vphi_i\:\:,\qquad
S(\vphi)  = \sum_{i:\ \te_{0,i}\neq 0} \vphi_i,
\end{align} 
counting for $\vphi$ the number of false and true discoveries, respectively.

\subsection{Heuristic}\label{sec:heu}

Why should the marginal empirical Bayes choice of $w$ lead to a correct control of the FDR? Here is an informal argument that will give a direction for our proofs. We consider the case of $\vphi^{\mbox{\tiny $q$-val}}$ here as it is expected to reject more nulls than $\vphi^{\mbox{\tiny $\l$-val}}$ and thus to have a larger FDR.

First, let us note that, when there is enough signal, one can expect $\hat{w}$ to be   approximately equal to the solution $ w^\star$ of the score equation in expectation $\E_{\theta_0}(\cS(w^\star))=0$, that is,  by using \eqref{equ:score},
$$
\sum_{i:\theta_{0,i}\neq 0}  m_1(\theta_{0,i},w^\star)=(n-s_n)  \tilde{m}(w^\star),
$$
where $\tilde{m}$ and $m_1$ are defined by \eqref{mtilde} and \eqref{m1}, respectively, if there $\te_0$ has exactly $s_n$ nonzero coordinates. 
As seen in Section~\ref{sec:moment}, up to $\log$-terms, 
\begin{align*}
\sum_{i:\theta_{0,i}\neq 0}  m_1(\theta_{0,i},w^\star) &\approx \sum_{i:\theta_{0,i}\neq 0}  \frac{\ol{\Phi}(\zeta(w^\star) - \theta_{0,i})+\ol{\Phi}(\zeta(w^\star) + \theta_{0,i})}{w^\star};\\
\tilde{m}(w^\star)&\approx 2\ol{G}(\zeta(w^\star)).
\end{align*}
Now consider the FDR  and assume that all quantities are well concentrated (in particular, take the expectation both in the numerator and denominator in \eqref{deffdr}). Then, by using \eqref{eqtypeIerrorqvalue}, we have, denoting $\vphi^{\mbox{\tiny $q$-val}}(\alpha;\hat{w},g)$ the $q$-value procedure at level $\alpha$ with parameters $\hat{w}, g$,
\begin{align*}
\FDR( \theta_0,\vphi^{\mbox{\tiny $q$-val}}(\alpha;\hat{w},g))&\approx 
\FDR( \theta_0,\vphi^{\mbox{\tiny $q$-val}}(\alpha;w^\star,g))\\
&\approx \frac{\sum_{i : \theta_{0,i} = 0} \P_{\theta_{0,i}}(q_i^\star(X)\leq \alpha) }{\sum_{i : \theta_{0,i} = 0} \P_{\theta_{0,i}}(q_i^\star(X)\leq \alpha) + \sum_{i : \theta_{0,i} \neq 0} \P_{\theta_{0,i}}(q_i^\star(X)\leq \alpha)} \\
&\approx
 \frac{(n-s_n) r(w^\star ,\alpha)\: 2\ol{G}\left(\zeta(w^\star)\right) }{(n-s_n)r(w^\star,\alpha)\: 2\ol{G}\left(\zeta(w^\star)\right)  +  \sum_{i : \theta_{0,i} \neq 0} \P_{\theta_{0,i}}(q_i^\star(X)\leq \alpha) },
 \end{align*}
 where we denoted $q_i^\star(X)=q(X_i;w^\star,g)$ and we used that $\chi(r(w^\star,t))$ is close to $\zeta(w^\star)$, as seen in Section~\ref{sec:thresholdprop}.
 Now, by using the definition of $q_i^\star(X)$, 
 \begin{align*}
 \sum_{i : \theta_{0,i} \neq 0} \P_{\theta_{0,i}}(q_i^\star(X)\leq \alpha)
= &
 \sum_{i : \theta_{0,i} \neq 0}   \ol{\Phi}\left(\chi(r(w^\star,\alpha)) - \theta_{0,i}\right)+ \ol{\Phi}\left(\chi(r(w^\star,\alpha)) 
 + \theta_{0,i}\right)\\
  \approx&
 \sum_{i : \theta_{0,i} \neq 0}   \ol{\Phi}\left(\zeta(w^\star) - \theta_{0,i}\right)+ \ol{\Phi}\left(\zeta(w^\star) + \theta_{0,i}\right),
 \end{align*}
where we used again $\chi(r(w^\star,t))\approx\zeta(w^\star)$.
Now using the above properties of $w^\star$, the latter is 
\begin{align*}
  \approx w^\star \sum_{i:\theta_{0,i}\neq 0}  m_1(\theta_{0,i},w^\star)  &= (n-s_n)  w^\star\tilde{m}(w^\star)\approx (n-s_n)  
   w^\star 2 \ol{G}(\zeta(w^\star)).
  \end{align*}
Putting the previous estimates together yields
\begin{align*}
\FDR( \theta_0,\vphi^{\mbox{\tiny $q$-val}}(\alpha;\hat{w},g))&\approx 
 \frac{(n-s_n) r(w^\star ,\alpha)\: 2\ol{G}\left(\zeta(w^\star)\right) }{(n-s_n)r(w^\star,\alpha)\: 2\ol{G}\left(\zeta(w^\star)\right)  +  (n-s_n)  w^\star 2\ol{G}(\zeta(w^\star))}\\
& =
\frac{ r(w^\star ,\alpha) }{ r(w^\star,\alpha) +  w^\star }
=
\frac{\frac{w^\star}{1-w^\star}\frac{\alpha}{1-\alpha} }{\frac{w^\star}{1-w^\star}\frac{\alpha}{1-\alpha}+ w^\star } 
 \approx  \frac{\frac{\alpha}{1-\alpha} }{\frac{\alpha}{1-\alpha}+ 1} = \alpha.
\end{align*}
We will see that this heuristic holds, up to some constant terms that may come in factor of the target level $\alpha$.\\

We note that one main challenge in the proof below is to show that the above estimates hold true for {\em any} sparse signal, in particular for `intermediate' signals $\te_0$ that are neither close to $0$ nor large enough (e.g. do not belong to $\cL_0[s_n]$ as in \eqref{L0}). Among others, we prove in Lemma \ref{lemw1w2} that $\hat w\in[w_2,w_1]$ with $w_2\asymp w_1$, thereby obtaining a sharp concentration of the marginal maximum likelihood estimate (uniformly over sparse vectors) that was not observed before in high dimensional settings, to the best of our knowledge. To derive some of the approximations $\approx$ above, we also sharpen several of the estimates for the moments   $m_1, \tilde{m}$ obtained in \cite{js04}, see e.g. Lemmas \ref{lemun} and \ref{m1binf} for sharp upper and lower bounds on $m_1$.

\subsection{ Proof of Theorems \ref{th1} and \ref{th2} }          

We prove results for $\ell$- and $q$-values together.  The proof for \EBayesqseuillage  is given at the end of this section.
First, let $w_0$ be the solution of the equation, 
\begin{equation} \label{defw0}
n w_0 \tilde{m}(w_0) = M, 
\end{equation}
for $M$ to be chosen below in the range $[1,\log{n}]$ (more precisely, equal to either $C\log(1/t)$ or $Ct^{-1}\log\log{n}$ for  a constant $C$ independent of $t$ and large enough; both bounds belong to the previous interval for $n$ large enough). For any $M\in[1,\log{n}]$, this equation has always a unique solution, as  $\tilde{m}$ is continuous increasing (see Lemma \ref{lemm12}) so the map $w\to w\tilde{m}(w)$ increases from $0$ at $w=0$ to a constant at $w=1$, and in particular has a continuous inverse. This implies that $w_0$ goes to $0$ with $n$, which we use freely in the sequel. Also, we note that $w_0$ is larger than $1/n$ for $C$ in the choice of $M$ large enough. Indeed, $w_0\ge \tilde{m}(1)^{-1}M/n$ by monotonicity of $\tilde{m}$. But $\tilde{m}(1)$ is at most a constant, so, provided $M$ is large enough, $w_0\ge 1/n$. Thus $w_0$ is always inside the interval $[n^{-1},1]$ over which the maximiser $\hat{w}$ is defined.

Let $\nu\in(0,1)$ be a fixed constant and  $\theta_0\in \l_0[s_n]$. Recall that $S_0$ denotes the support of $\theta_0$ and that $\sigma_0=|S_{0}|$ denotes the exact number of nonzero coefficients of $\te_0$, so that $0\le \sigma_0\le s_n$. The next equation, depending on the configuration $\theta_0$,  and on the just defined $w_0$, plays a key role in the proof:
\begin{equation} \label{equw1}
 \sum_{i\in S_0} m_1(\te_{0,i},w)  = (1-\nu) (n-\sigma_0) \tilde{m}(w),          \:\:\: w\in [w_0,1).
\end{equation}  
This equation may or may not have a solution, depending on the true $\te_0$ and the values of $n$ and $\nu$. 
We will now assume $n\ge N_0$ for some universal constant $N_0$ to be determined below.

\subsubsection{Case 1: \eqref{equw1} has no solution}\label{sec:nosolution}

For a given value of $n$, let us consider the case where 
\eqref{equw1} has no solution in $w\in[w_0,1)$. 

First, the maps $w\in[0,1]\to \tilde{m}(w)$  and $w\in[0,1]\to m_1(\mu,w)$ ($\mu\in\R$) 
are continuous, see Lemmas \ref{lemm12} and  \ref{lemmtilde} and, for any $\mu\in\R$,
$$
|m_1(\mu,1)| \leq \int \left|\frac{\beta(x)}{1+\beta(x)}\right| \phi(x-\mu) dx\leq \max_{x\in\R}\left|\frac{\beta(x)}{1+\beta(x)}\right|,
$$ 
so that
 $ \sum_{i\in S_0} m_1(\te_{0,i},1) \le C \sigma_0 < (1-\nu) (n-\sigma_0) \tilde{m}(1)$ for $n\ge N_0$, where we use $\sigma_0\le s_n\le  n^{\upsilon}$ and $\tilde{m}(1)>0$ and $N_0=N_0(g,\upsilon)$. This means 
\begin{equation} \label{spaeq_case1}
\sum_{i\in S_0} m_1(\te_{0,i},w)  < (1-\nu) (n-\sigma_0) \tilde{m}(w),\:\:\: \mbox{ for } w\in [w_0,1),
\end{equation} 
as otherwise by the intermediate value theorem (e.g. Theorem 4.23 in \cite{rudin}) the graphs of the functions on the two sides of the previous inequality would have to cross on $[w_0,1)$ and \eqref{equw1} would have a solution. 
Lemma \ref{lembern_case1} shows that, under \eqref{spaeq_case1}, we have
\begin{equation}\label{conccase1}
\P_{\theta_0}( \hat{w} > w_0) \leq e^{-  C_0 \nu^2 M},
\end{equation}
for some constant $C_0=C_0(g,\upsilon)$.
Now consider $\vphi$ being either $\vphi^{\mbox{\tiny $\ell$-val}}$ or $\vphi^{\mbox{\tiny $q$-val}}$, and denote by $\vphi(t;\hat{w},g)$ such a procedure with cut-off $t$ and parameters $\hat w, g$, as defined in \eqref{def:ebayesl}-\eqref{def:ebayesq}. Let us 
upper-bound  the FDR by the so-called family-wise error rate by distinguishing the two cases $\hat w\le w_0$ and $\hat w>w_0$:
\begin{align}
\FDR(\theta_0,\vphi(t;\hat{w},g)) &\leq \P_{\theta_0}(\exists i \::\: \theta_{0,i}=0 ,\: \vphi_i(t;\hat w,g) =1) \nonumber\\
&\leq \P_{\theta_0}(\exists i\::\: \theta_{0,i}=0 ,\:\vphi_i(t;w_0,g) =1) + \P_{\theta_0}(\hat w>w_0)\nonumber\\
&\leq (n-\sigma_0)\P_{\theta_{0,i}=0}(\vphi_i(t;w_0,g) =1) + e^{-  C_0 \nu^2 M},\label{equ:intermth1}
 \end{align} 
 where we use that $w\to \vphi_i(t;w,g)$ is nondecreasing, see Lemma~\ref{lem:noninclandq}, together with a union bound. 
 
 \paragraph{$\l$-value part}
Let $\xi_0=\xi(r(w_0,t))$ and $\zeta_0=\zeta(w_0)$, 
then \eqref{ubtypeIerrorlvalue} leads to (provided $r(w_0,t)\le (\phi/g)(0)$, which holds for e.g. $t\le 3/4$ and $w_0\le 1/4$) 
\begin{align*}
\FDR(\te_0,\vphi^{\mbox{\tiny $\ell$-val}}(t;\hat{w},g)) 
&\leq 2 \frac{n w_0}{1-w_0} \frac{t}{1-t} \frac{g\left(\xi_0\right)}{\xi_0}+ e^{-  C_0 \nu^2 M}.
 \end{align*}
Combining the definition of $w_0$ and Lemma \ref{lemmtilde},  taking $n$ large enough so that $w_0$ is appropriately small, with $t\le 3/4$,
\begin{align*}
\FDR(\te_0,\vphi^{\mbox{\tiny $\ell$-val}}(t;\hat{w},g)) 
&\leq \frac{5  M }{  \xi_0}
\frac{g\left(\xi_0\right)}{\ol{G}(\zeta_0) }t
+ e^{-  C_0 \nu^2 M}.
 \end{align*}
Noting that $|\xi_0-\zeta_0|\lesssim 1$,    $g(\xi_0)\le D g(\zeta_0)$ and $\ol{G}(\zeta_0)\asymp \zeta_0^{\kappa-1}g(\zeta_0)$ by Lemma \ref{lemxizeta} and \ref{lemmtilde}, 
one obtains
 \begin{equation}\label{lvalcase1}
  \FDR(\te_0,\vphi^{\mbox{\tiny $\ell$-val}}(t;\hat{w},g)) 
 \leq \frac{C(g)  M }{  \zeta_0^\kappa}t
+ e^{-  C_0 \nu^2 M}.
\end{equation} 

 \paragraph{$q$-value part}
 
For the $q$-value case, we come back to \eqref{equ:intermth1} and use \eqref{eqtypeIerrorqvalue}  instead of 
\eqref{ubtypeIerrorlvalue} to get, setting $\chi_0= \chi( r(w_0,t))$,
\begin{align*}
\FDR(\te_0,\vphi^{\mbox{\tiny $q$-val}}(t;\hat{w},g)) &\leq 2 \frac{nw_0}{1-w_0} \frac{t}{1-t} \: \ol{G}\left(\chi_0\right) + e^{-  C_0 \nu^2 M}.
 \end{align*} 
As a result, by \eqref{defw0} and Lemma~\ref{lemmtilde}, one gets for $n$ large enough, $t\le 3/4$,
\begin{align*}
\FDR(\te_0,\vphi^{\mbox{\tiny $q$-val}}(t;\hat{w},g))&\leq  5 Mt  \: \frac{ \ol{G}(\chi_0)}{ \ol{G}\left(\zeta_0\right)} + e^{-  C_0 \nu^2 M}.
 \end{align*}
Now, by the last assertion of Lemma~\ref{lemxizeta}, the ratio in the last display is bounded by  $2$ (say) provided $n$ is large enough, which gives 
\begin{align}\label{qvalcase1}
\FDR(\te_0,\vphi^{\mbox{\tiny $q$-val}}(t;\hat{w},g))&\leq  10 M t   + e^{-  C_0 \nu^2 M}.
 \end{align}
 
\subsubsection{Case 2: \eqref{equw1} has a solution} In this case we denote the solution by $w_1\in[w_0,1)$, so that one can write
\begin{equation} \label{dob1}
 \sum_{i\in S_0} m_1(\te_{0,i},w_1)  = (1-\nu) (n-\sigma_0) \tilde{m}(w_1).\end{equation}
Now consider the slightly different equation in $w$
\begin{equation} \label{equw2}
 \sum_{i\in S_0} m_1(\te_{0,i},w)  = (1+\nu) (n-\sigma_0) \tilde{m}(w),          \:\:\: w\in [0,1).
\end{equation} 
Equation \eqref{equw2} always has a (unique) solution $w_2\in [0,w_1)$. To see this, first note that the case $\te_0=0$  is excluded from \eqref{dob1}, as  $m_1(0,w)=-\tilde{m}(w)<0$ if $w\neq 0$. By Lemma \ref{lemm12}, $w\to m_1(\mu,w)$ and $w\to \tilde{m}(w)$ are continuous and respectively decreasing and increasing  (both strictly), and $\tilde{m}(0)=0$, while it can be seen that $m_1(\mu,0)>0$ if $\mu\neq 0$, see Lemma~\ref{lemm12}.
 On the other hand, the value at $w=1$ of the left hand side of \eqref{equw2} is at most $\sigma_0 C/w\leqa \sigma_0$, and so is of smaller order  than $(1+\nu)(n-\sigma_0)\tilde{m}(1)\asymp n$. 
 
The purpose of $w_1,w_2$ is to provide (implicit) deterministic upper and lower bounds for the random $\hat w$: this is the content of Lemma \ref{lembernw12}.  Additionally, the key Lemma  \ref{lemw1w2} shows that, in case where the solution $w_1$ of \eqref{dob1} exists, we have $w_1\asymp w_2$; that is, the bounds are of the same order.

 \paragraph{$q$-value part}

Recall the notation \eqref{not_VS}. We focus on the case of $q$-values first. We come back to the case of $\ell$-values at the end, its proof being similar.  For simplicity, we write $V_q(w)=V(\vphi^{\mbox{\tiny $q$-val}}(t;w,g))$ and $S_q(w)=S(\vphi^{\mbox{\tiny $q$-val}}(t;w,g))$.
By definition of the FDR,
\begin{align*}
\lefteqn{\FDR(\theta_0,\vphi^{\mbox{\tiny $q$-val}}(t;\hat{w},g)) 
 = \E_{\te_0}\left[ \frac{V_q(\hat w)}{(V_q(\hat w)+S_q(\hat w))\vee 1}\right]}  \\
&& \le \E_{\te_0}\left[ \frac{V_q(\hat w)}{(V_q(\hat w)+S_q(\hat w))\vee 1}
\ind{w_2\le \hat w\le w_1}\right] + \P_{\te_0}\left[\hat w\notin [w_2,w_1]\right].
\end{align*}
The last expectation in the previous display is  now bounded by, using first the monotonicity of the maps  $w\to V_q(w)$, $w\to S_q(w)$, $x\to x/(1+x)$ and $x\to 1/(1+x)$, then bounding the indicator variable by $1$, 
and finally combining with Lemma \ref{lembinom} applied to the {\em independent}  variables $U=V_q(w_1)$  and $T=S_q(w_2)$,
\begin{flalign*}
\lefteqn{\E_{\te_0}\left[ \frac{V_q(\hat w)}{(V_q(\hat w)+S_q(\hat w))\vee 1}\ind{w_2\le \hat w\le w_1}\right]\le \E_{\theta_{0}}\left[  \frac{V_q(w_1)}{(V_q(w_1)+S_q(w_2))\vee 1} \right] };\\
&&\le \exp\{-\E_{\theta_0}S_q(w_2)\} +12 \frac{\E_{\theta_{0}}V_q(w_1)}{\E_{\theta_{0}}S_q(w_2)}.\qquad\qquad\qquad\end{flalign*}
Next, by using the definition of $V_q$, one writes
\begin{align*}
\E_{\theta_{0}}V_q(w_1) &=  \sum_{i:\ \te_{0,i}=0} 2 \ol{\Phi}(\chi(r(w_1,t)))
= 2(n-\sigma_0) \ol{\Phi}(\chi(r(w_1,t))).
\end{align*}
Using the definition of $\chi$, we have $\ol{\Phi}(\chi(u))=\ol{G}(\chi(u))u$ for $u\in(0,1)$, so
\[ \ol{\Phi}(\chi(r(w_1,t))) = r(w_1,t)\ol{G}(\chi(r(w_1,t))). \]
Then \eqref{GchirsurGzeta} in Lemma \ref{lemxizeta}  implies, for small enough $w_1$,
\[ \ol{G}(\chi(r(w_1,t)))\le 2 \ol{G}(\zeta(w_1)). \]
Combining 
 \eqref{w1w2bounding} in Lemma~\ref{lemw1w2}, that is $w_1/\Cst\le w_2\le w_1$, for a constant $C=C(\nu,\upsilon,g)>0$ and Lemma \ref{lemTmuw}, we have (with, say, $c_1=1/2$), 
 \[  (1/2)\ol{G}(\zeta(w_1))\le \ol{G}(\zeta(w_1/\Cst))\le \ol{G}(\zeta(w_2)). \]
Next using Lemma \ref{lemmtilde}, one obtains $\ol{G}(\chi(r(w_1,t)))\leq 3 \:\tilde{m}(w_2)$, so that
\begin{align*}
 \E_{\theta_{0}}V_q(w_1) & \le 
3(n-\sigma_0)\frac{w_1}{1-w_1}\tilde{m}(w_2) \frac{t}{1-t} \\
& \le 3 \Cst (n-\sigma_0)\frac{w_2}{1-\Cst w_2}\tilde{m}(w_2) \frac{t}{1-t}\\
& \le C^* (n-\sigma_0)w_2\tilde{m}(w_2)t,
\end{align*}
because $t\leq 3/4$ for some constant $C^*=C^*(\nu,\upsilon,g)>0$.
On the other hand, by definition of $S_q$, one can write
$$
\E_{\theta_{0}} S_q(w_2)=  \sum_{i : \theta_{0,i} \neq 0}   \ol{\Phi}\left(\chi(r(w_2,t)) - \theta_{0,i}\right)+ \ol{\Phi}\left(\chi(r(w_2,t)) 
 + \theta_{0,i}\right).
$$
Let us introduce the set of indices, for $K_1=2/(1-\upsilon)$, 
\begin{equation}\label{equforproof}
 \cC_0(w,K_1) =  \left\{1\leq i\leq n\::\:\ |\te_{0,i}| \geq \frac{\zeta(w)}{K_1} \right\}.
 \end{equation}
Moreover, $\chi(r(w_2,t)) \leq \zeta(w_2)$ by Lemma \ref{propchizeta}. Hence,
\begin{align} 
\E_{\theta_{0}} S_q(w_2) &\geq  \sum_{i \in \cC_0(w_2,{K_1})}   \ol{\Phi}\left(\zeta(w_2) - \theta_{0,i}\right)+ \ol{\Phi}\left(\zeta(w_2) 
 + \theta_{0,i}\right)\nonumber\\
& \geq  \sum_{i \in \cC_0(w_2,{K_1})}   \ol{\Phi}\left(\zeta(w_2) - |\theta_{0,i}| \right).\label{equforproof2}
\end{align}
First, we apply Corollary~\ref{cor:lemun} with $K=K_1$, $w=w_2$ to bound each term in the sum in terms of $m_1$, noting that $|\te_{0,i}|\ge \zeta(w_2)/K_1$ by definition of the set $\cC_0(w_2,{K_1})$. Next, one uses Lemma \ref{lemsmallsignal} restricting the suprema to $w=w_2$ (which is in the prescribed interval by Lemmas \ref{lemw0}, \ref{lemw1} and \ref{lemw1w2}) and $K=K_1$, 
  to get for $n$ large enough and constants $C=C(\upsilon,g)>0$, $C'=C'(\upsilon,g)>0$, $D=D(\upsilon,g)\in (0,1)$,
\begin{align*}    
\lefteqn{ \sum_{i \in \cC_0(w_2,{K_1})}    \ol{\Phi}\left(\zeta(w_2) - |\theta_{0,i}| \right)
  \geq   C w_2 \sum_{i \in \cC_0(w_2,{K_1})}   m_1(\theta_{0,i},w_2)} &&\\  
& \geq   C w_2 \Big\{\sum_{i \in S_0}   m_1(\theta_{0,i},w_2) - C' n^{1-D}
{\tilde{m}(w_2)}\Big\} \\
&= C  w_2 \Big\{(1+\nu)(n-\sigma_0)\tilde{m}(w_2) -  C' n^{1-D}
{\tilde{m}(w_2)} \Big\},
\end{align*}
where the last equality comes from \eqref{equw2}. 
As a consequence, 
for $n$ large enough, for a positive constant 
$C_*=C_*(\upsilon,g)>0$, we have
\begin{align*}
\E_{\theta_{0}} S_q(w_2)&\geq C_* (n-\sigma_0)w_2\tilde{m}(w_2).
\end{align*}
Combining  the previous bounds leads to
\begin{align*}
\E_{\te_0}\left[ \frac{V_q(\hat w)}{V_q(\hat w)+S_q(\hat w)\vee 1}\ind{w_2\le \hat w\le w_1}\right]
& \le e^{-C_* (n-\sigma_0)w_2\tilde{m}(w_2)}+ 12 \frac{C^*}{C_*} t.
\end{align*}
As $w\to w\tilde{m}(w)$ is increasing, and $w_1/C\le w_2$ by Lemma~\ref{lemw1w2}, we have $w_2\tilde{m}(w_2)\ge (w_1/C)\tilde{m}(w_1/C)$. Recall that $w_1\ge w_0$ by definition, so Lemma~\ref{lemmtilde} together with \eqref{equGbarwM} of Lemma~\ref{lemTmuw} imply 
\[ \tilde{m}(w_1/C)\ge (1/2) \tilde{m}(w_1)\ge (1/2) \tilde{m}(w_0). \]
Combining the obtained inequalities leads to 
\begin{equation}\label{equforproof3}
 (n-\sigma_0)w_2\tilde{m}(w_2)\ge C'(n-\sigma_0)w_0\tilde{m}(w_0)\ge C'M,
 \end{equation}
where the last inequality follows from the definition of $w_0$. Now turning to a bound on the FDR, Lemma \ref{lembernw12} and the above inequality imply, with $\nu=1/2$,
\begin{equation}\label{equuseful}
 \P_{\te_0}[\hat w \notin[w_1,w_2]] \le 2 e^{- C_1\nu^2 n w_2 \tilde{m}(w_2)}\le 2e^{-CM},
 \end{equation}
 for some $C=C(\upsilon,g)>0$.
Conclude that in the considered case, for some constants $c_1=c_1(\upsilon,g), c_2=c_2(\upsilon,g)>0$,
\begin{align}\label{qvalcase2}
 \FDR(\theta_0,\vphi^{\mbox{\tiny $q$-val}}(t;\hat{w},g)) 
&\leq c_2 t  +3 e^{- c_1 M}.
\end{align}

 \paragraph{$\l$-value part}

In the case of $\ell$-values, one can follow a similar argument. 
 We write $V_\ell(w)=V(\vphi^{\mbox{{\tiny $\ell$-val}}}(t;w,g))$ and $S_\ell(w)=S(\vphi^{\mbox{{\tiny $\ell$-val}}}(t;w,g))$. Again, the maps $w\to V_{\ell}(w)$ and $w\to S_\ell(w)$ are monotone. 
So, as above for $q$-values, 
\[ \FDR(\theta_0,\vphi^{\mbox{\tiny $\ell$-val}}(t;\hat{w},g)) 
\le
\exp\{-\E_{\theta_0}S_\ell(w_2)\} +12 \frac{\E_{\theta_{0}}V_\ell(w_1)}{\E_{\theta_{0}}S_\ell(w_2)} + \P_{\te_0}\left[\hat w\notin [w_2,w_1]\right].
\]
By definition of $V_\ell$ and $\xi$, one can write
\begin{align*}
\E_{\theta_{0}}V_\ell(w_1) &=   2(n-\sigma_0) \ol\Phi(\xi(r(w_1,t))).
\end{align*}
The bound $\ol\Phi(u)\le \phi(u)/u$ for $u>0$ (see Lemma~\ref{bphi}), combined with the definition of $\xi$ and 
that $|\xi(r(w_1,t))-\zeta(w_1)|\lesssim 1$ by Lemma \ref{lemxizeta}  leads to
\[ \E_{\theta_{0}}V_\ell(w_1) 
\le  3(n-\sigma_0)\zeta(w_1)^{-1} {r(w_1,t)} g(\xi(r(w_1,t))).
\]
Lemma \ref{lemxizeta} then implies $g(\xi(r(w_1,t)))\le 2 g(\zeta(w_1))$ (say), for $n$ large enough. 
Using $w_1/C\le w_2\le w_1$, and \eqref{equGbarwM} in Lemma \ref{lemTmuw}, we have 
\[ (1/2) g(\zeta(w_1))\le g(\zeta(w_1/C))\le g(\zeta(w_2)). \]
Next using the relation $\zeta^{\kappa-1}g(\zeta)\asymp \tilde{m}(w)$ from Lemma \ref{lemmtilde}, one obtains $g(\xi(r(w_1,t))) \leqa \zeta(w_2)^{1-\kappa}\tilde{m}(w_2)\leqa \zeta(w_1)^{1-\kappa}\tilde{m}(w_2)$, so that
\begin{align*}
 \E_{\theta_{0}}V_\ell(w_1) & \le 
Ct(n-\sigma_0) w_1 \tilde{m}(w_2) \zeta(w_1)^{-\kappa}\\
& \le c^*t (n-\sigma_0)w_2\tilde{m}(w_2)\zeta(w_1)^{-\kappa},
\end{align*} 
for a constant $c^*=c^*(\upsilon,g)>0$.  
On the other hand, by definition of $S_\ell$, 
$$
\E_{\theta_{0}} S_\ell(w_2)=  \sum_{i : \theta_{0,i} \neq 0}   \ol{\Phi}\left(\xi(r(w_2,t)) - \theta_{0,i}\right)+ \ol{\Phi}\left(\xi(r(w_2,t)) 
 + \theta_{0,i}\right).
$$
Lemma \ref{lembfi} now enables to bound from below the two terms in the previous display in terms of $\zeta(w_2)$, and further restricting the sum to  the set of indices $\cC_0(w_2,K_1)$ defined by \eqref{equforproof} with the same choice of $K_1$ leads to
\begin{align*} 
\E_{\theta_{0}} S_\ell(w_2) &\geq   Ct \sum_{i \in \cC_0(w_2,K_1)}   \ol{\Phi}\left(\zeta(w_2) - |\theta_{0,i}| \right).
\end{align*}
Appart from the $Ct$ term in factor, it is the same bound as for $q$-values, see \eqref{equforproof2}. Hence, using the bound obtained above, for $n$ large enough and $c_*=c_*(\upsilon,g)>0$,
\begin{align*}
\E_{\theta_{0}} S_\ell(w_2)&\geq c_* t(n-\sigma_0)w_2\tilde{m}(w_2).
\end{align*}
Combining  the previous bounds leads to
\begin{align*}
\E_{\te_0}\left[ \frac{V_\ell(\hat w)}{V_\ell(\hat w)+S_\ell(\hat w)\vee 1}\ind{w_2\le \hat w\le w_1}\right]
& \le e^{-c_* t(n-\sigma_0)w_2\tilde{m}(w_2)}+ 12 \frac{c^*}{c_*} \frac{1}{\zeta(w_1)^{\kappa}}.
\end{align*}
As in \eqref{equforproof3}, we have   $(n-\sigma_0)w_2\tilde{m}(w_2)\ge C'(n-\sigma_0)w_0\tilde{m}(w_0)\ge C'M$. One concludes that, in Case 2, for some constants $d_1=d_1(\upsilon,g), d_2=d_2(\upsilon,g)>0$ and taking $\nu=1/2$, setting $\zeta(w_1)=\zeta_1$,
\begin{align} 
 \FDR(\theta_0,\vphi^{\mbox{\tiny $\ell$-val}}(t;\hat{w},g)) 
&\leq d_2 \zeta_1^{-\kappa} +  e^{- C'  M c_* t} + 2 e^{-CM}\nonumber\\
&\le d_2 \zeta_1^{-\kappa} + 3e^{-d_1Mt}.\label{lvalcase2}
\end{align}  
\subsubsection{Combining cases 1 and 2}
For $q$-values, for $\nu=1/2$ and $t\leq 3/4$, we get by combining \eqref{qvalcase1} and \eqref{qvalcase2}
\begin{align*}
\FDR(\theta_0,\vphi^{\mbox{\tiny $q$-val}}(t;\hat{w},g)) 
& \le\max\left\{10 M t   + e^{- C_0  M}, c_2 t  +3 e^{- c_1 M}\right\}
\end{align*}
Taking $M= (C_0\wedge c_1)^{-1}\log (1/t)$ gives the upper bound  
\[ \FDR(\theta_0,\vphi^{\mbox{\tiny $q$-val}}(t;\hat{w},g)) \le 
\max\{C't\log(1/t)+e^{-\log(1/t)},c_2t+3e^{-\log{1/t}}\},\]
which is smaller than $Ct\log(1/t)$, giving the result for $q$-values.

In the $\ell$-values case, with $\zeta_1\le \zeta_0$ and setting $\nu=1/2$, we get by combining \eqref{lvalcase1} and \eqref{lvalcase2}
\begin{align*}
\FDR(\theta_0,\vphi^{\mbox{\tiny $\ell$-val}}(t;\hat{w},g)) 
& \le\max\left\{ C  M  \zeta_0^{-\kappa}t
+ e^{-C_0  M}, d_2 \zeta_1^{-\kappa}  +3 e^{- d_1 tM}\right\}\\
& \le d_3\{ M\zeta_1^{-\kappa}t+\zeta_1^{-\kappa}  + e^{- d_4 tM}\}.
\end{align*} 
The announced bound is obtained upon setting $M=t^{-1}d_4^{-1}\log(\zeta_1^\kappa)$ and noting that $\zeta_1^2\leqa \log(1/w_1)\leqa \log{n}$ and $\zeta_1^2\geqa \log(1/w_1)\geqa \log{n}$ 
by using Lemmas \ref{lemw0}, \ref{lemw1} to bound $w_1$  and Lemma~\ref{lemzetagen} to bound  $\zeta(w_1)$.
This concludes the proof of Theorem \ref{th1} for $\ell$-values and Theorem \ref{th2} for $q$-values.

\subsubsection{Proof for \EBayesqseuillage}

 First notice that
 \begin{align}
\FDR(\te_0,\vphiqvalseuillage(t;\hat{w},g)) &= \E_{\te_0} \left[ \frac{\sum_{i=1}^n \ind{\theta_{0,i}=0} \vphiqvalseuillage(t;\hat{w},g)}{1\vee \sum_{i=1}^n \vphiqvalseuillage(t;\hat{w},g)} \right]\nonumber\\
&= \E_{\te_0} \left[ \frac{\sum_{i=1}^n \ind{\theta_{0,i}=0} \vphi^{\mbox{\tiny $q$-val}}(t;\hat{w},g)}{1\vee \sum_{i=1}^n \vphi^{\mbox{\tiny $q$-val}}(t;\hat{w},g)}   \ind{\hat w>\omega_n} \right],\label{equforqvalseuillage}
\end{align} 
by definition of algorithm \EBayesqseuillage.
The strategy of proof is similar to the $q$-value case. Let us take $M$ in the definition \eqref{defw0} of $w_0$ equal to  $L_n$  from the statement of Theorem~\ref{th2}, see  \eqref{defom}, and suppose $L_n\in[1,\log{n}]$. 
Let us show for $n$ large enough,
\begin{equation}\label{equ:omeganw0}
\omega_n\geq w_0 .
\end{equation}
As $\zeta(w_0)\leq \zeta(1/n)\leq \sqrt{2.1\log{n}}$ for $n$ large enough by Lemmas~\ref{lemw0},~\ref{lemzetagen},
 $$\omega_n= \frac{L_n}{n \overline{G}(\sqrt{2.1\log{n}})}\geq \frac{L_n}{n \overline{G}(\zeta(1/n))}\geq \frac{L_n}{n \overline{G}(\zeta(w_0))}.$$
    Now, by using Lemma~\ref{lemmtilde},  for $n$ large enough,
    $$\frac{L_n}{n \overline{G}(\zeta(w_0))}\geq 0.9 \frac{ 2 L_n}{n \tilde{m}(w_0)} \geq \frac{  L_n}{n \tilde{m}(w_0)} = w_0,$$
 leading to \eqref{equ:omeganw0}.
Next, on the one hand, in Case 1, the FDR is bounded by
\begin{align*}
\FDR(\te_0,\vphiqvalseuillage(t;\hat{w},g))
&\leq \P_{\theta_0}(\hat w>\omega_n)\leq \P_{\theta_0}(\hat w>w_0).
 \end{align*} 
 By using~\eqref{conccase1}, the last display is at most $e^{- C_0 \nu^2L_n}$. 
 On the other hand, in Case $2$, we simply use that by \eqref{equforqvalseuillage},
\begin{align*}
 \FDR(\theta_0,\vphiqvalseuillage(t;\hat{w},g)) 
\leq \FDR(\theta_0,\vphi^{\mbox{\tiny $q$-val}}(t;\hat{w},g)) 
&\leq c_2 t  +3 e^{- c_1 L_n}, 
\end{align*}
which concludes the proof.

\section*{Acknowledgments}
The authors would like to thank the Associate Editor and the referees for their helpful and constructive comments. 
This work has been supported by 
ANR-16-CE40-0019 (SansSouci) and ANR-17-CE40-0001 (BASICS).

\bibliographystyle{abbrv}
\bibliography{bibmt}

\pagebreak

$\ $

\vspace{.3cm}

\begin{center}
{\large \bf Supplement to ``On spike and slab empirical Bayes multiple testing"}
\end{center}

\vspace{1cm}

This supplementary file contains additional materials for the proofs as well as the proof of Propositions \ref{prop:BFDR}--\ref{prop:typeIerror} and Theorems \ref{theorem:psharpFDRcontrol}--\ref{theorem:power}; a study of $\ell$-values and $q$-values; inequalities for the thresholds of the corresponding BMT procedures; properties of the moment functions $\tilde{m}$, $m_1$ and $m_2$; an optimality result for the simultaneous control of type I and II testing errors; details on related procedures, including a proof of Theorem \ref{thm-mci},  as well as additional numerical experiments.

\vspace{.5cm}

\section{Intermediate lemmas used in the proof of main results}

 In the sequel we freely use that $s_n\le n^{\upsilon}$ as assumed in the main results of the paper. We assume that the function $g$ satisfies the assumptions from \eqref{assumpgdebase} up to and including \eqref{increasing} (recall that this is in particular the case if $g$ arises from a convolution $g=\gamma\star \phi$ for $\ga$ satisfying \eqref{asump1}--\eqref{asump3}, which is the case in the Bayesian setting with a slab density $\ga$). 

We start by two basic lemmas on $w_0=w_0(n,M)$, $w_1=w_1(n, M, \theta_0,\nu)$, $w_2=w_2(n,M,\theta_0,\nu)$,  quantities introduced in \eqref{defw0}, \eqref{dob1}, \eqref{equw2}, respectively.
 
\begin{lemma} \label{lemw0}
Let $w_0$ as in \eqref{defw0} with $M>1$ arbitrary. Let $\tilde{m}$ be defined by \eqref{mtilde}. Then, for an integer $N_0(g)>0$, and constants $c_1=1/\tilde{m}(1)$, $c_2=c_2(g)$, we have for all $n\ge N_0(g)$,
\[  \frac{n}{M} \tilde{m}\left(Mc_1/n \right) \le \frac{1}{w_0} \le  \frac{n}{M} \tilde{m}\left(\sqrt{Mc_2 /n}\right). \]
In particular, for any $M\in[1,\log{n}]$, for $C_1, C_2$ depending only on $g$,
\[ C_1\frac{\sqrt{\log{n}}}{n}\le w_0 \le \frac{\log{n}}{n}e^{C_2\sqrt{\log{n}}}.\]
\end{lemma}
\begin{proof}
Lemma \ref{lemmtilde} gives $\tilde{m}(w)\geqa w^{c}$ for any $c>0$.
Setting $c=1$ and using the 
equation defining $w_0$, that is $nw_0\tilde{m}(w_0)=M$, leads to $w_0\le (C M/n)^{1/2}$. Reinserting this estimate into  $\tilde{m}$ in the equation defining $w_0$ (by using that $\tilde m$ is increasing by Lemma \ref{lemm12}) gives the first upper bound of the lemma. Next, one notes that $\tilde{m}(w)\le \tilde{m}(1)$, which leads to $w_0\ge M/(n\tilde{m}(1))$. Reinserting this estimate into $\tilde m$ in the equation defining $w_0$ gives the first lower bound of the lemma.
 
To prove the second display of the lemma, one notes that the fact that $\log{g}$ is Lipschitz and $g(u)\leqa (1+u^2)^{-1}$ by \eqref{tails2} imply for $w$  small enough,
\[ \zeta(w)^{\kappa-1}e^{-\La\zeta(w)} \leqa \tilde{m}(w) \leqa \zeta(w)^{\kappa-3}.\]
Using the first display of the lemma together with Lemma \ref{lemzetagen} on $\zeta$ and $1\le M\le \log{n}$ leads to the result.
\end{proof}

\begin{lemma} \label{lemw1}
For $M>0$ and $\nu\in(0,1)$, there exist an integer $N_0=N_0(\nu,\upsilon,g)>0$ and $r=r(\nu,\upsilon,g)
\in(0,1)$ such that for all $n\geq N_0$ and $\theta_0\in \ell_0[s_n]$, 
if a solution $w_1=w_1(n,M,\theta_0,\nu)$ of \eqref{dob1} exists, then
\[ w_0\le w_1\le n^{-r}. \]
\end{lemma}
\begin{proof}
The lower bound follows from the definition of $w_0$ and $w_1$. For the upper bound, one uses  the definition of $w_1$ and the global bound $|m_1(\mu,w)|\le 1/(w\wedge c_1)$ (which follows from Lemma \ref{lembeta}) to get, 
\[ \frac{\sigma_0}{w_1\wedge c_1}\ge (1-\nu)(n-\sigma_0)\tilde{m}(w_1).\]
As $\tilde m$ is increasing and $\tilde{m}(w)\geqa w^c$ for arbitrary $c\in(0,1)$ (see Lemma \ref{lemmtilde}), one gets $(w_1\wedge c_1)^{1+c}\le C\sigma_0/n\le Cs_n/n$. Using $s_n\le n^{\upsilon}$ gives the result. 
\end{proof}

\begin{lemma}[Bernstein $w_0$]  \label{lembern_case1} 
There exist an integer $N_0=N_0(g,\upsilon)>0$ and $C_0=C_0(g)>0$ such that the following holds for all $n\geq N_0$ and $\theta_0\in \ell_0[s_n]$. Let $M\in[1,\log{n}]$ and  $w_0$ as in \eqref{defw0}. Let $\nu\in (0,1)$ and assume
 \eqref{spaeq_case1} (which is implied by the fact that \eqref{equw1} has no solution). 
Then the MMLE estimate $\hat w$ satisfies 
\begin{equation}\label{conccase1supp}
\P_{\theta_0}( \hat{w} > w_0) \leq e^{- C_0 \nu^2 n w_0 \tilde{m}(w_0)} = e^{- C_0 \nu^2M}.
\end{equation}
\end{lemma}
\begin{proof}[Proof of Lemma \ref{lembern_case1}]
One first notes the almost sure equality of  events $\{\hat w > w_0\} = \{\cS(w_0)>0\}$. This follows since 
 $\cS$ is (strictly) decreasing and continuous on $[0,1]$ (except in the case that $g(X_i)=\phi(X_i)$ for all $i$ which happens with probability $0$). Then,  with $P=P_{\te_0}, E=E_{\te_0}$ as shorthand,
\begin{align*}
\P(\hat{w} > w_0) &= \P(\cS(w_0) >0) = \P(\cS(w_0) - \E \cS(w_0)>- \E \cS(w_0))\\
&\leq \P(\cS(w_0) - \E \cS(w_0)> \nu (n-\sigma_0)  \tilde{m}(w_0) ), 
\end{align*}
as $\E \cS(w_0) =  \sum_{i\in S_0} m_1(\te_{0,i},w_0) - (n-\sigma_0)  \tilde{m}(w_0) < -  \nu (n-\sigma_0)  \tilde{m}(w_0)$ using \eqref{spaeq_case1}. 
Now, the score function equals $\cS(w_0)=\sum_{i=1}^n \beta(X_i,w_0)$, a sum of independent variables. One applies Bernstein's inequality (see Lemma~\ref{th:bernstein} and notation therein) to the variables $W_i=\beta(X_i,w_0)-\E\beta(X_i,w_0)$. 
Note that $|W_i|\leq 2/w_0=:\cM$ as $|\be|\le (w_0\wedge c_1)^{-1}=w_0^{-1}$ by Lemma \ref{lembeta} for $n$ large enough (indeed, $w_0$ goes to $0$ with $n$ by Lemma \ref{lemw0}).   
Also, 
\[ V :=\sum_{i=1}^n\var(W_i) \le \sum_{i=1}^n m_2(\te_{0,i},w_0).
 \]
 One splits the last sum in two. Consider $\zeta_0=\beta^{-1}(w_0^{-1})$ the pseudo-threshold associated to $w_0$.
  Using Corollary \ref{m2m1} (recall as noted above that $w_0$ goes to $0$ with $n$), with $M_0$ the constant therein, combined with \eqref{spaeq_case1}, one gets 
\begin{align*} 
 \sum_{i:\ |\te_{0,i}|>M_0} m_2(\te_{0,i},w_0)&\le  \frac{C_2}{w_0}\sum_{i:\ |\te_{0,i}|>M_0} m_1(\te_{0,i},w_0)\\
 & \le \frac{C_2}{w_0} (1-\nu ) (n-\sigma_0) \tilde{m}(w_0)  - \frac{C_2}{w_0}\sum_{i:\ |\te_{0,i}|\leq M_0} m_1(\te_{0,i},w_0)\\
 &\leq \frac{2 C_2}{w_0} (1-\nu ) n \tilde{m}(w_0),
 \end{align*}
 because $\mu\in\R_+\to m_1(\mu,w_0)$ is nondecreasing (see Lemma~\ref{lemm12}) and bounded from below by $-\tilde m(w_0)$.

For small non-zero signals, one uses  Lemma \ref{lemdeux} to get, with 
$\zeta_0:=\zeta(w_0)$,
 \[ \sum_{i:\ 0<|\te_{0,i}|\le M_0} m_2(\te_{0,i},w_0)\le C \sum_{i:\ 0<|\te_{0,i}|\le M_0} \frac{\ol\Phi(\zeta_0-|\te_{0,i}|) }{w_0^2}\le C\sigma_0 \frac{\ol\Phi(\zeta_0-M_0) }{w_0^2},\]
 and one uses $\ol\Phi(\zeta_0-M_0)\le C\phi(\zeta_0-M_0)/\zeta_0\le C' e^{M_0\zeta_0}\phi(\zeta_0)/\zeta_0$. With Lemma \ref{lemmtilde}, one gets  $\phi(\zeta)/\zeta\asymp w g(\zeta)/\zeta\asymp w\tilde{m}(w)/\zeta^\kappa$ for small $w$, so that 
\[ \sum_{i:\ |\te_{0,i}|\le M_0} m_2(\te_{0,i},w_0)\leqa \frac{s_n e^{M_0\zeta_0}}{n\zeta_0^\kappa } \frac{n \tilde{m}(w_0)}{w_0} \leqa \frac{n \tilde{m}(w_0)}{\zeta_0^\kappa w_0},\]
where we use that $s_n e^{M_0\zeta_0}/n\le C$, as follows from $s_n=O(n^\upsilon)$ and $\zeta_0^2\leqa \log{n}$ (combining Lemmas \ref{lemw0} on $w_0$ and Lemma~\ref{lemzetagen}). 
With $A = (n-\sigma_0)\nu\tilde{m}(w_0)$, one gets, for $n\ge N_0$,
\[ \frac{V+\frac13 \cM A}{A^2} \leqa 
\frac{\nu^{-2}}{nw_0\tilde m(w_0)} + \frac{ \nu^{-2}}{n w_0 \tilde m(w_0) \zeta_0^\kappa}  
\leqa \frac{\nu^{-2}}{n w_0 \tilde m(w_0)},\]
An application of Bernstein's inequality (see Lemma~\ref{th:bernstein}) now gives \eqref{conccase1supp}.
\end{proof}
 
\begin{lemma}[Bernstein $w_1, w_2$] \label{lembernw12}
There exist an integer $N_0=N_0(g,\upsilon)>0$ and $C_1=C_1(g)>0$ such that the following holds for all $n\geq N_0$ and $\theta_0\in \ell_0[s_n]$:
for $\nu\in(0,1)$, suppose that a solution $w_1$ of \eqref{dob1} exists, and let $w_2$ be the solution of \eqref{equw2}. Then the MMLE estimate $\hat w$ satisfies  
\begin{equation}\label{concw2}
\P_{\theta_0}( \hat{w} \notin [w_2, w_1]) \leq e^{- C_1\nu^2 n w_1 \tilde{m}(w_1)}+ e^{- C_1\nu^2 n w_2 \tilde{m}(w_2)}.
\end{equation}
\end{lemma}
\begin{proof}
One bounds successively each of the probabilities $\P(\hat w>w_1)$ 
and $\P(\hat w<w_2)$. The first bound is obtained in exactly the same way as in the proof of Lemma \ref{lembern_case1}, with $w_0$ replacing $w_1$. We note the two minor differences: $\E \cS(w_1)=\sum_{i\in S_0} m_1(\te_{0,i},w_1) - (n-\sigma_0)  \tilde{m}(w_1)$ now {\em equals} 
$-  \nu (n-\sigma_0)  \tilde{m}(w_1)$ by the definition \eqref{dob1} of $w_1$. 
Then bounds on $m_2$ can be carried out in the same way -- now evaluated at $w=w_1$ -- as in the proof of Lemma \ref{lembern_case1}. We note that $w_1$ goes to zero with $n$ by Lemma \ref{lemw1}.
This means that we can use the bounds of Lemma \ref{lemdeux} and Corollary \ref{m2m1} as in the proof of Lemma \ref{lembern_case1}. Further, if $\zeta_1:=\zeta(w_1)$, we have $\zeta_1\le \zeta_0$, so one also has $s_ne^{M_0\zeta_1}/n\le C$ using the corresponding bound for $\zeta_0$. This shows the desired result for $w_1$.

For $w_2$, one proceeds similarly.
If $w_2=0$ the result is immediate. 
Otherwise we have $\{\hat w < w_2\} = \{\cS(w_2)<0\}$.   Again, one applies Bernstein's inequality to the score function $\cS(w)=\sum_{i=1}^n \beta(X_i,w)$ and set $W_i=\beta(X_i,w_2)-m_1(\te_{0,i},w_2)$.  
As $W_i$ are centered independent variables with $|W_i|\le \cM$ and $\sum_{i=1}^n\text{Var}(W_i)\le \sum_{i=1}^n E[\beta(X_i,w_2)^2] =: V_2$,  for any $B>0$,
\[ P\left[ \sum_{i=1}^n W_i <-B\right] \le \exp\{-\frac12 B^2/(V_2+\frac13 MB) \}.\]
One can take $\cM=c_3/w$, using Lemma \ref{lembeta}. Set $B = \sum_{i=1}^n m_1(\te_{0,i},w_1)$. 
By definition of $w_2$ in \eqref{equw2}, we have
\[ B=-(n-\sigma_0)\tilde m(w_2) + \sum_{i\in S_0} m_1(\te_{0,i},w_2) 
= \nu(n-\sigma_0)\tilde m(w_2). \] 
The term $V_2$ is bounded in a similar way as in the proof of Lemma \ref{lembern_case1}, using  the bounds of Lemma \ref{lemdeux} and Corollary \ref{m2m1}. As for $w_1$ above, one notes that, if $\zeta_{2}=\zeta(w_2)$,  one has $s_ne^{M_0\zeta_2}/n\le C$ as, using Lemma \ref{lemw1w2}, we have $w_1\leqa w_2$, so that $w_2\geqa 1/n$ and $\zeta_2\leqa \sqrt{\log n}$. 
One obtains $V_2\leqa (nw_2\tilde{m}(w_2))^{-1}$ which leads to 
\[ \frac{V_2+\frac13 \cM B}{B^2} \leqa \frac{\nu^{-2}}{n w_2 \tilde m(w_2)},\]
and the desired bound on $w_2$ is obtained. 
\end{proof}

\begin{lemma} \label{lemw1w2}
Let $\nu\in(0,1)$. There exist some integer $N=N(\nu,\upsilon,g)>0$ and $C=C(\nu,\upsilon,g)>1$ such that, for all  $n\geq N$ and $\theta_0\in \ell_0[s_n]$, if \eqref{dob1} has a solution $w_1$, the solution $w_2$ of \eqref{equw2} verifies 
\begin{align}\label{w1w2bounding}
w_1/C \leq  w_2\leq w_1.
\end{align}
\end{lemma}
\begin{proof}
The behaviour of $w_1, w_2$ for a given specific true signal $\te_0$ is determined through properties of the function
$$H_{\te_0}(w)= \sum_{i\in S_0} m_1(\te_{0,i},w)/\tilde{m}(w).$$ This function is decreasing, as $w\to m_1(\te_{0,i},w), 1\le i\le n,$ and $w\to \tilde{m}(w)^{-1}$ both are, by Lemma~\ref{lemm12}. 
It suffices to show that
 for an appropriately large constant $z\ge 1$ (possibly depending on  $\upsilon,g,\nu$), for $n$ large enough, 
 \begin{equation} \label{proph}
H_{\te_0}\left(\frac{w_1}{z}\right) \ge \frac{1+\nu}{1-\nu} H_{\te_0}(w_1),
\end{equation}
Indeed, by definition of $w_1, w_2$, one has $H_{\te_0}(w_2) = (1+\nu)(n-\sigma_0)=(1+\nu)(1-\nu)^{-1}H_{\te_0}(w_1)$. 
 So, if \eqref{proph} holds, $H_{\te_0}(w_2)\le H_{\te_0}(w_1/z)$ which in turn yields $w_2\ge w_1/z$ by monotonicity. 

Now, \eqref{proph} is obtained in two steps. First, one shows that appropriately small signals do not contribute too much to the sum defining $H_{\te_0}$, so that one can replace the sum in \eqref{proph} by a sum $H_{\te_0}^\circ$, to be defined now, on large signals only. For $w\in(0,1)$ and $K>1$,  set $\cC_0(w,K)=\{1\leq i\leq n\::\: |\theta_{0,i}|\geq \zeta(w)/K\}$ and 
\[ H^\circ_{\te_0}(w, K) = \sum_{i\in \cC_0(w,K)} m_1(\te_{0,i},w)/\tilde m(w). \]
Set $K_2=4/(1-\upsilon)$.  
By Lemmas~\ref{lemw0} and~\ref{lemw1}, both $w_1$ and  $w_1/z$ belong to the interval $[1/n,1/\log n]$, provided $z\leqa (\log n)^{1/4}$ (which will be the case below). 
Let us now use, with $K_1=K_2/2$ and $D>0$, both Lemmas \ref{lemsmallsignal} and \ref{lemhcirc}, and $z=z(\nu,\upsilon, g)$ a constant to be chosen below,
\begin{align*}
H_{\te_0}\Big(\frac{w_1}{z}\Big) & = H_{\te_0}^\circ\Big(\frac{w_1}{z}, K_2\Big) + 
H_{\te_0}\Big(\frac{w_1}{z}\Big)-H_{\te_0}^\circ\Big(\frac{w_1}{z}, K_2\Big)\\
& \ge   Cz^{1/(2K_2)} H_{\te_0}^\circ(w_1,K_2/1.1) - C'n^{1-D}\\
& \ge Cz^{(1-\upsilon)/8} H_{\te_0}^\circ(w_1, K_1) - C'n^{1-D},
\end{align*}
where in the last inequality one uses that $K\to H_{\te_0}^\circ(w,K)$ is nondecreasing by definition. 
Using Lemma \ref{lemsmallsignal} again now shows that, for $D>0$, 
\[ |H_{\te_0}(w_1) - H_{\te_0}^\circ(w_1, K_1)| \le C'n^{1-D}.\]
One deduces that, for $C$ the constant in the one but last display,
\[ H_{\te_0}\Big(\frac{w_1}{z}\Big) \ge Cz^{(1-\upsilon)/8} H_{\te_0}(w_1) +o(n).\]
Since $H_{\te_0}(w_1)\asymp n$ by definition of $w_1$, the latter is bounded from below by $(C/2)z^{(1-\upsilon)/8}H_{\te_0}(w_1)$ for $n$ large enough. Taking $z=\{\max((2/C),1)(1+\nu)/(1-\nu)\}^{8/(1-\upsilon)}$  shows \eqref{proph} and the proof is complete.
\end{proof}

\section{Auxiliary proofs}

\subsection{Proof of Proposition~\ref{prop:BFDR}}\label{sec:prop:BFDR}

For any multiple testing procedure $\vphi$, 
\begin{align*}
\BFDR(\vphi;w,\gamma)=\int_{\R^n} \FDR(\theta,\vphi) d\Pi_{w,\gamma}(\theta)  &= \E_{X,\theta} \left[\frac{\sum_{i=1}^n \ind{\theta_{i}=0} \vphi_i}{1\vee \sum_{i=1}^n  \vphi_i}\right].
\end{align*}
For $\vphi^\l$, using the chain rule $E[\cdot]= E[E[\cdot\given X]]$, one gets
\begin{align*}
\BFDR(\vphi^\l;w,\gamma) &=\E_{X} \left[\frac{\sum_{i=1}^n \l_i(X) \vphi^\l_i}{1\vee \sum_{i=1}^n  \vphi^\l_i}\right]=\E_{X} \left[\frac{\sum_{i=1}^n \l_i(X) \mathds{1}_{\{\l_i(X)\leq \alpha\}}}{1\vee \sum_{i=1}^n  \mathds{1}_{\{\l_i(X)\leq \alpha\}}}\right]\\
&\leq \alpha \:\P(\exists i\::\: \l_i(X)\leq \alpha).
\end{align*} 
For $\vphi^q$, conditioning this time on the variables $\vphi_1^q(X),\ldots,\vphi_n^q(X)$ and using that for the prior $\Pi_{w,g}$ the conditional distribution of $\te_i\given X$ only depends on $X_i$ for all $i$, so that $E[\ind{\theta_{i}=0}\given \vphi_1^q,\ldots,\vphi_n^q]\vphi_i^q = P(\te_i=0\given \vphi_i^q=1)\vphi^q_i$ a.s., one obtains
\begin{align*}
\BFDR(\vphi^q;w,\gamma)
& = \E_{X} \left[\frac{\sum_{i=1}^n \P(\theta_{i}=0\:|\: \vphi_i^q=1) \vphi_i^q}{1\vee \sum_{i=1}^n  \vphi_i^q}\right]\\
&=\E_{X} \left[\frac{\sum_{i=1}^n \P(\theta_{i}=0\:|\: q_i(X)\leq \alpha) \mathds{1}_{\{q_i(X)\leq \alpha\}}}{1\vee \sum_{i=1}^n  \mathds{1}_{\{q_i(X)\leq \alpha\}}}\right].
\end{align*}
Now observe that from \eqref{increasing2}, $q_i(X)\leq \alpha$ if and only if $|X_i|\geq \Psi(\alpha)$, for some function $\Psi$ such that $q(\Psi(\alpha);w,g)=\alpha$ (namely, $\Psi$ is the inverse of $u\in (0,\infty)\to q(u;w,g)$). Now, the result follows from
$$
\P(\theta_{i}=0\:|\: q_i(X)\leq \alpha) = \P(\theta_i=0 \:|\: |X_i|\geq \Psi(\alpha)) = q(\Psi(\alpha);w,g)=\alpha.
$$
Finally, the relation between \eqref{equBFDRlval} and \eqref{equBFDRqval} comes from Lemma~\ref{lem:qlvalue}.

\subsection{Proof of Proposition~\ref{prop:typeIerror}}\label{sec:prop:typeIerror}
For the $\ell$-value part, we use Lemma~\ref{bphi}:
 \begin{align*}
\P_{\theta_0=0}(\l_i(X)\leq t) &= 2 \ol{\Phi}\left(\xi(r(w,t))\right)\leq 2 \frac{\phi\left(\xi(r(w,t))\right)}{\xi(r(w,t))},
\end{align*}
which provides \eqref{ubtypeIerrorlvalue} because $\phi\left(\xi(r(w,t))\right)=r(w,t) g\left(\xi(r(w,t))\right)$ by definition of $\xi(\cdot)$. Next, if $\xi(r(w,t))\geq 1$, that is if $r(w,t)\le (\phi/g)(1)$ using \eqref{equxi},
\begin{align*}
\P_{\theta_0=0}(\l_i(X)\leq t) 
&\geq  \frac{\phi\left(\xi(r(w,t))\right)}{\xi(r(w,t))},
\end{align*}
which provides \eqref{lbtypeIerrorlvalue}. 
The $q$-values part follows from the definition of $\chi$. 

\subsection{Proof of Theorem~\ref{theorem:psharpFDRcontrol}}

We prove the result first for {\tt EBayesq}.
Recall that the exact number of nonzero coefficients  $\sigma_0$ of $\te_0$ is $s_n$ by definition of $\cL_0[s_n]$. 
Set $b=(a+1)/2>1$ and let $\cA$ be the event, for $K_n< s_n$ to be specified below,
\[ \cA=\left\{ \# \{i\in S_0,\ |X_i| > b\{2\log (n/s_n)\}^{1/2}\}\ge s_n-K_n\right\}. \]
If $\cA^c$ denotes the complement of $\cA$,
\begin{align*}
  \cA^c & =\left\{ \# \{i\in S_0,\ |X_i| > b\{2\log (n/s_n)\}^{1/2}\} < s_n-K_n\right\} \\
  & = \left\{ \# \{i\in S_0,\ |X_i| \le b\{2\log (n/s_n)\}^{1/2}\} >  K_n\right\} \\
  & \subset 
  \left\{ \# \{i\in S_0,\ |\veps_i| > (a-b)\{2\log (n/s_n)\}^{1/2}\} > K_n\right\}=:\cC, 
\end{align*}  
where we have used $X_i=\te_{0,i}+\veps_i$ to get $|\veps_i|\ge |\te_{0,i}|-|X_i|$ by the triangle inequality. Let $c=\sqrt{2}(a-b)>0$.  
By looking at the indicator variables $Z_i=\1_{|\veps_i|\ge x_n}$ with $x_n= c\{\log (n/s_n)\}^{1/2}$, one can translate the event $\cC$ in the last display into an event for a binomial trial, leading to 
\[ \sup_{\te_0\in\cL_0[s_n]} P_{\te_0}[\cA^c] \le P\left[ \text{Bin}(s_n, 2\ol{\Phi}(x_n))>K_n\right].\]
Let $p_n=2\ol{\Phi}(x_n)$, then using the expression of $x_n$ above,
\[ p_n \le 2\phi(x_n)/x_n \le C(s_n/n)^{c^2/2}/(c\sqrt{\log(n/s_n)}),\]
which goes to $0$ with $n$ as $s_n/n\to 0$.  

Let $K_n=\max(2s_np_n,s_n/\log{s_n})$.
By Bernstein's inequality, see Lemma \ref{th:bernstein}, as $K_n\ge 2s_np_n$ and $\sum_{i\in S_0}\text{Var}(Z_i)\le s_n p_n$,
\[ P\Bigg[\sum_{i\in S_0} Z_i >K_n\Bigg]\le P\Bigg[\sum_{i\in S_0} (Z_i-p_n) >K_n/2\Bigg]\le \exp\Bigg\{-\frac18 \frac{K_n^2}{K_n/6+s_np_n}\Bigg\},\] 
which is less, using $s_np_n\le K_n/2$ again, than $\exp(-CK_n)$, which goes to $0$ with $n$, since $K_n\ge s_n/\log{s_n}\to\infty$. So, we have obtained $P_{\te_0}[\cA^c]=o(1)$, uniformly over $\te_0\in\cL_0[s_n]$. 

Now one can follow the proof of Theorems~\ref{th1}~and~\ref{th2} and consider the fundamental equation \eqref{equw1}, for some fixed $\theta_0 \in \cL_0[s_n]$, and $n$ large enough. The lower bound on $w$ is given here by $w_0$ in \eqref{defw0}, for some $M=M_n$ that we choose as $M_n=\min(c_0 s_n,\log{n})$, so that $M_n\to\infty$ and $c_0$ a small enough constant to be chosen below.

Consider both sides of the equation \eqref{equw1} at the point $w=s_n/n$. On the one hand, by definition of $\cL_0[s_n]$, we have $|\te_{0,i}|\ge a\{2\log(n/s_n)\}^{1/2}$ for $i\in S_0$. Lemma \ref{lemzetagen} implies $\zeta(s_n/n)\sim \{2\log(n/s_n)\}^{1/2}$, so one can apply Lemma \ref{m1binflargesignal} (recall $\mu\to m_1(\mu,w)$ is even for all $w$) for a small $\veps>0$ to get, for large enough $n$,
\[ \sum_{i\in S_0} m_1(\te_{0,i},s_n/n) \ge (1-\veps) \frac{s_n}{(s_n/n)}=(1-\veps)n. \]
On the other hand, the right hand side of \eqref{equw1} equals $(1-\nu)(n-s_n)\tilde{m}(s_n/n)= o(n)$, since $\tilde m(w)$ goes to $0$ as $w\to 0$. 
Recall that $\sum_{i\in S_0}m_1(\te_{0,i},1)$ is bounded from above by a constant times $s_n$ (as $m_1(\te_{0,i},1)$ is bounded, see Section~\ref{sec:nosolution}) and that $(1-\nu)n\tilde{m}(1)$ is of the order $n$. Combining the previous inequalities, the intermediate values theorem shows that \eqref{equw1} has a solution, at least on $[s_n/n,1)$, for $n$ large enough. 

To show that $w_1$ exists, it is enough to check that the solution also belongs to $[w_0,1)$. We distinguish two cases. If $w_0\le s_n/n$ then this is obvious by definition. In case $w_0>s_n/n$, let us evaluate both sides of \eqref{equw1} this time at $w=w_0$. First, using the second display of Lemma \ref{lemw0} (compatible with the present choice on $M_n=\min(c_0 s_n,\log{n})$) combined with Lemma \ref{lemzetagen} on $\zeta$, one gets,  for arbitrary $\veps>0$ and using $w_0>s_n/n$, that 
\[ \zeta(w_0)\le (1+\veps)\sqrt{2\log(1/w_0)}\le (1+\veps)\sqrt{2\log(n/s_n)},\] 
for large enough $n$. Deduce that one can apply Lemma \ref{m1binflargesignal} as $(1+\rho)\zeta(w_0)\le |\te_{0,i}|$ for small enough $\rho$. In particular
\[ \sum_{i\in S_0} m_1(\te_{0,i},w_0) \ge (1-\veps) \frac{s_n}{w_0}. \]
On the other hand, the right hand side of \eqref{equw1} is $(1-\nu)(n-s_n)\tilde{m}(w_0) = (1-\nu)\{(n-s_n)/n\} M_n/w_0$ by definition of $w_0$. As $M_n\le c_0s_n$,  this quantity is thus smaller than the last display, provided $c_0$  is small enough. By the same reasoning as above, this shows that the solution to \eqref{equw1} indeed belongs to $[w_0,1)$, so $w_1$ exists.   

Now that we have the existence of $w_1$, the fact that $w=s_n/n$ cannot be a solution of \eqref{equw1} (for $n$ large enough) and the monotonicity of both sides of \eqref{equw1} show that $w_1\ge s_n/n$, for $n$ large enough. Using the same argument with equation \eqref{equw2} leads to $w_2\ge s_n/n$, for $n$ large enough.

As \eqref{equw1} has a solution, we can use the properties of the proof of Section~\ref{sec:proof:th1} in this case (referred to as Case 2 in that proof). In particular,  \eqref{equuseful} provides  for some constant $C>0$,
$$
\sup_{\theta_0 \in \cL_0[s_n]}P_{\theta_0}( \hat{w} \notin [w_2, w_1]) \leq 2e^{-CM_n}.
$$

Let us introduce the event $\Omega_0 = \cA \cap \{\hat w \in[w_2,w_1]\}$. By the previous bounds, we have $P_{\te_0}[\Omega_0^c]=o(1)$, uniformly over $\te_0\in\cL_0[s_n]$.  
Note that, on the event $\Omega_0$,
\[ \chi(r(\hat{w},t)) \leq  \zeta(\hat w)\leq \zeta(w_2) 
\]
 using  Lemma~\ref{propchizeta} and the monotonicity of $\zeta(\cdot)$. We have seen that here  $w_2\ge s_n/n$, so $\zeta(w_2)\le \zeta(s_n/n)$ and combining  with the equivalent of $\zeta(w)$ as $w\to 0$ from Lemma~\ref{lemzetagen}, one finally gets $\chi(r(\hat{w},t)) \le c(2\log(n/s_n))^{1/2}$ for any $c>1$ for $n$ large enough, so in particular for $c=b$ as defined above.  
One deduces that on $\Omega_0$, the $q$-value procedure  
  $\vphi^{\mbox{\tiny $q$-val}}$
rejects the null hypotheses corresponding to the (at least $s_n-K_n$) indexes $i$ in $S_0$ such that  $|X_i| > b\{2\log (n/s_n)\}^{1/2}$, because $b\{2\log (n/s_n)\}^{1/2}\geq \chi(r(\hat{w},t))$ by using the previous bounds and the definition of the event $\cA$. 

Combining the above facts, we obtain
\begin{align*}
&\sup_{\theta_0 \in \cL_0[s_n]} \FDR(\te_0,\vphi^{\mbox{\tiny $q$-val}}(t;\hat{w},g)) \\
&= \sup_{\theta_0 \in \cL_0[s_n]}  \E_{\te_0} \left[ \frac{\sum_{i=1}^n \ind{\theta_{0,i}=0} \vphi^{\mbox{\tiny $q$-val}}(t;\hat{w},g)}{1\vee \sum_{i=1}^n \vphi^{\mbox{\tiny $q$-val}}(t;\hat{w},g)} \right]\\
&\leq \sup_{\theta_0 \in \cL_0[s_n]}  \E_{\te_0} \left[ \frac{\sum_{i=1}^n \ind{\theta_{0,i}=0} \vphi^{\mbox{\tiny $q$-val}}(t;\hat{w},g)}{1\vee \sum_{i=1}^n \vphi^{\mbox{\tiny $q$-val}}(t;\hat{w},g)}   \ind{\Omega_0}\right] + o(1).
\end{align*} 
Therefore, since $\vphi^{\mbox{\tiny $q$-val}}(t;\hat{w},g)$ 
makes at least $s_n-K_n$ correct rejections, 
 that is, $\#\{i\in S_0:\ \vphi^{\mbox{\tiny $q$-val}}_i(t;\hat{w},g)=1\}\ge s_n-K_n$, we derive
\begin{align}
&\sup_{\theta_0 \in \cL_0[s_n]} \FDR(\theta_0,\vphi^{\mbox{\tiny $q$-val}}) \nonumber\\
&\leq \sup_{\theta_0 \in \cL_0[s_n]}  \E_{\theta_0}\left[\frac{\sum_{i=1}^n \mathds{1}{\{\theta_{0,i}=0\}}  \vphi_i^{\mbox{\tiny $q$-val}}(t;w_1)}{\sum_{i=1}^n \mathds{1}{\{\theta_{0,i}=0\}} \vphi_i^{\mbox{\tiny $q$-val}}(t;w_1) + s_n-K_n }\right] + o(1)\nonumber\\
&\leq \frac{\sup_{\theta_0 \in \cL_0[s_n]}\E_{\theta_0}\left[ \sum_{i=1}^n \mathds{1}{\{\theta_{0,i}=0\}}  \vphi_i^{\mbox{\tiny $q$-val}}(t;w_1) \right] }{\sup_{\theta_0 \in \cL_0[s_n]}\E_{ \theta_0}\left[\sum_{i=1}^n \mathds{1}{\{\theta_{0,i}=0\}} \vphi_i^{\mbox{\tiny $q$-val}}(t;w_1) \right]  + s_n-K_n}+ o(1),\label{equ:intemth4}
\end{align}
by concavity and monotonicity of the function $x\in [0,+\infty)\to x/(x+1)$. 

Now combine \eqref{eqtypeIerrorqvalue}, Lemma~\ref{lemxizeta} and Lemma~\ref{lemmtilde} to get for any $\veps\in(0,1)$, for any $\theta_0 \in \cL_0[s_n]$, 
\begin{align*}
\E_{\theta_0}\left[ \sum_{i=1}^n \mathds{1}{\{\theta_{0,i}=0\}}  \vphi_i^{\mbox{\tiny $q$-val}}(t;w_1) \right] 
&= (n-s_n) r(w_1,t)\: 2\ol{G}\left(\chi( r(w_1,t))\right) \\
&\leq (1+\veps) t (1-t)^{-1} w_1 (n-s_n) \: 2\ol{G}\left(\zeta(w_1)\right) \\
&\leq (1+\veps)^2 t (1-t)^{-1} (n-s_n)  w_1 \tilde{m}(w_1).
\end{align*}
Next, since  $w_1$ is a solution of \eqref{equw1}, 
the latter is bounded above by
$$
(1+\veps)^2  (1-\nu)^{-1}  t (1-t)^{-1} \sum_{i\in S_0} w_1 m_1(\te_{0,i},w_1) \leq   (1+\veps)^2  (1-\nu)^{-1}  t (1-t)^{-1} s_n,
$$ 
by using that $m_1(\cdot,w)$ is always upper-bounded by $1/w$ for small $w$, see Lemma~\ref{lemm12} (recall that $w_1$ goes to $0$ with $n$ by Lemma \ref{lemw1}). 
Putting this back into \eqref{equ:intemth4} gives for $n$ large enough,
\begin{align*}
\sup_{\theta_0 \in \cL_0[s_n]} \FDR(\theta_0,\vphi^{\mbox{\tiny $q$-val}}) 
&\leq \frac{(1+\veps)^2  (1-\nu)^{-1} t (1-t)^{-1} s_n }{(1+\veps)^2  (1-\nu)^{-1} t (1-t)^{-1} s_n + s_n-K_n}+ o(1).
\end{align*}
As $K_n=o(s_n)$ as shown above, taking the limsup as $n\to \infty$ and then letting $\veps, \nu$ go to $0$, we get,  observing that  $\frac{t (1-t)^{-1} }{t (1-t)^{-1} +1} = t$,
\begin{align*}
\varlimsup_n \sup_{\theta_0 \in \cL_0[s_n]} \FDR(\theta_0,\vphi^{\mbox{\tiny $q$-val}}) 
&\leq t.
\end{align*}

Let us now turn to prove
\begin{align}\label{equ:lowerboundth4}
\varliminf_n \inf_{\theta_0 \in \cL_0[s_n]} \FDR(\theta_0,\vphi^{\mbox{\tiny $q$-val}}) 
&\geq t,
\end{align}
which will lead to the conclusion.
 Fix some $\delta\in(0,1)$ and for any $\te_0\in\cL_0[s_n]$ 
   consider $w_1$ and $w_2$ the associated solution of \eqref{equw1} and \eqref{equw2}, respectively. The fact that both exist has been seen above. 
   Let $\Omega_1 = \{\hat w \in[w_2,w_1]\}$, then   
 \begin{align*}
&\inf_{\theta_0 \in \cL_0[s_n]} \FDR(\theta_0,\vphi^{\mbox{\tiny $q$-val}}) \\
&\geq  \inf_{\theta_0 \in \cL_0[s_n]} \E_{\theta_0}\left[\frac{V_q}{V_q + s_n } \mathds{1}{\{\Omega_1\}} \right]\\
&\geq  \inf_{\theta_0 \in \cL_0[s_n]} \E_{\theta_0}\left[\frac{\E_{\theta_0} V_q (1-\delta)}{\E_{\theta_0} V_q (1-\delta) + s_n } \mathds{1}{\{\Omega_1\}} \mathds{1}{\{ |V_q-\E_{\theta_0} V_q|\leq \delta \E_{\theta_0} V_q\}}  \right],
\end{align*}
where we have denoted $V_q=\sum_{i=1}^n \mathds{1}{\{\theta_{0,i}=0\}}  \vphi_i^{\mbox{\tiny $q$-val}}(t;w_2)$, which is a Binomial variable.
Similarly to the upper bound, combine \eqref{eqtypeIerrorqvalue}, Lemma~\ref{propchizeta} and Lemma~\ref{lemmtilde} to get for any $\veps\in(0,1)$ and $\theta_0 \in \cL_0[s_n]$, 
\begin{align*}
\E_{\theta_0} V_q 
&= (n-s_n) r(w_2,t)\: 2\ol{G}\left(\chi( r(w_2,t))\right) \\
&\geq  t (1-t)^{-1} w_2 (1-w_2)^{-1} (n-s_n) \: 2\ol{G}\left(\zeta(w_2)\right) \\
&\geq (1-\veps) t (1-t)^{-1} w_2  (n-s_n) \: 2\ol{G}\left(\zeta(w_2)\right) \\
&\geq (1-\veps)^2 t (1-t)^{-1} (n-s_n)  w_2 \tilde{m}(w_2).
\end{align*}
Now using  that $w_2$ is a solution of \eqref{equw2} and Lemma~\ref{m1binflargesignal}, we obtain
\begin{align*}
\E_{\theta_0} V_q
&\geq (1-\veps)^2 (1+\nu)^{-1} t (1-t)^{-1}  \sum_{i\in S_0} w_2 \:m_1(\te_{0,i},w_2)  \\
&\geq (1-\veps)^3 (1+\nu)^{-1} t (1-t)^{-1} s_n.
\end{align*}
 Next, observe that by Chebychev's inequality, the supremum over $\te_0\in\cL_0[s_n]$ of the following probability
 $$
\P_{\theta_0}( |V_q-\E_{\theta_0} V_q|> \delta \E_{\theta_0} V_q)\leq \frac{\var_{\theta_0}(V_q)}{\delta^2 (\E_{\theta_0} V_q)^2} \leq \frac{1}{\delta^2 \E_{\theta_0} V_q}
 $$
goes to $0$, because $s_n$ tends to infinity. 
Combining the above facts leads to
  \begin{align*}
\inf_{\theta \in \cL_0[s_n]} \FDR(\theta,\vphi^{\mbox{\tiny $q$-val}}) 
\geq  \frac{(1-\veps)^3 (1-\delta)(1+\nu)^{-1} t (1-t)^{-1}  }{(1-\veps)^3 (1-\delta)(1+\nu)^{-1} t (1-t)^{-1}  + 1}+ o(1),
\end{align*}
and the result is proved by taking the liminf in $n$ and then $\delta,\veps,\nu$ tending to zero.

Finally, to prove the result for \EBayesqseuillage
 one notes that by the previous arguments $\hat w$ belongs to $[w_1,w_2]$ with probability tending to $1$, and $w_2\ge s_n/n$, 
  which is larger than $2\omega_n$ by assumption. Deduce that the event $\{\hat w>\omega_n\}$ has probability going to $1$ so the procedures  {\tt EBayesq} and \EBayesqseuillage coincide with probability going to $1$, which proves that  \EBayesqseuillage also satisfies the desired property.

\subsection{Proof of Theorem \ref{theorem:power}}

Since the denominator in the definition \eqref{deffnr} of the FNR is a constant, and as $q$-values are more liberal than $\ell$-values, it is enough to prove the result for $\ell$-values, i.e. that
$\FNR(\theta_0,\vphi^{\mbox{\tiny $\ell$-val}})$ goes to $0$ uniformly over $\theta_0$ in the set $\cL_0[s_n]$. As we work with $\te_0$ in $\cL_0[s_n]$, we are in the setting of the proof of Theorem \ref{theorem:psharpFDRcontrol}. We now recall some elements from that proof that are helpful here as well. First recall the notation $c=\sqrt{2}(a-b)>0$ and  
\[x_n= c\{\log (n/s_n)\}^{1/2}, \qquad p_n=2\ol{\Phi}(x_n).\]
Setting $K_n=\max(2s_np_n,s_n/\log{s_n})$, it has been seen in the proof of Theorem \ref{theorem:psharpFDRcontrol} that for the event
\[ \cA=\left\{ \# \{i\in S_0,\ |X_i| > b\{2\log (n/s_n)\}^{1/2}\}\ge s_n-K_n\right\}, \]
one has, uniformly over $\cL_0[s_n]$, that
$P_{\te_0}[\cA^c]=o(1)$. It was also shown that if further $\Omega_0=\cA \cap \{\hat w \in[w_2,w_1]\}$, then $P_{\te_0}[\Omega_0^c]=o(1)$ uniformly over $\cL_0[s_n]$ as well as $w_2\ge s_n/n$.
  
Combining the previous facts implies $\zeta(\hat{w})\le \zeta(w_2)\le \zeta(s_n/n)$ as well as $\zeta(\hat{w})\ge \zeta(w_1)$.
 The definition of $w_1$ as solution of the fundamental equation \eqref{equw1} implies that $w_1$ is smaller than an arbitrary small constant for $n$ large enough, by Lemma \ref{lemw1}.
  
From \eqref{relzetaxi} in Lemma \ref{lemxizeta}, one deduces
\[ \xi(r(\hat{w},t)) \le \zeta(\hat{w}) + \frac{2|\log\left(\frac{t}{1-t}\right)|+C}{\zeta(\hat{w})}.\]
By combining with the previous upper and lower bounds on $\zeta(\hat{w})$, one obtains $\xi(r(\hat{w},t)) \le \zeta(s_n/n)+C'$ on $\Omega_0$, so that $\xi(r(\hat{w},t)) \le y(2\log(n/s_n))^{1/2}$ for any $y>1$ for $n$ large enough, by Lemma \ref{lemzetagen}. By definition of the event $\cA$ above, one deduces that the $\ell$-value procedure $\vphi^{\mbox{\tiny $\ell$-val}}$ rejects the null hypotheses for the (at least $s_n-K_n$ by definition of the set $\cA$ part of $\Omega_0$) indexes $i\in S_0$ such that $|X_i|>b(2\log(n/s_n))^{1/2}$ with $b=(a+1)/2>1$. 

One deduces that, uniformly for $\theta_0\in\cL_0[s_n]$,
\begin{align*}
\FNR(\theta_0,\vphi^{\mbox{\tiny $\ell$-val}}) 
 &\le   \E_{\theta_0}\left[\frac{\sum_{i=1}^n \ind{\theta_{0,i}\neq 0} (1-\vphi_i(X))}{s_n\vee 1}1_{\Omega_0}\right] + P_{\te_0}(\Omega_0). \\
 & \le \frac{K_n\wedge s_n}{s_n\vee 1} + o(1).
 \end{align*}
which is a $o(1)$ as by definition $K_n\le \max(2s_np_n,s_n/\log{s_n})=o(s_n)$, which concludes the proof of Theorem \ref{theorem:power}.

\section{Basic properties of $\ell$--, $q$-- and $p$--values}\label{sec:proplqval}

Let us assume that $g$ satisfies \eqref{assumpgdebase} throughout this section. Recall that this assumption holds in particular whenever $g$ is of the form $g=\phi\star \gamma$ as in the Bayesian setting.
 
\begin{lemma}
The $q$-value functional \eqref{qvaluesfunc} has the explicit expression 
\[ q(x;w,g) = \frac{(1-w)\ol{\Phi}(|x|)}{(1-w)\ol{\Phi}(|x|) + w \:\ol{G}(|x|)} , \:\: x\in\R, \:\:w\in [0,1].
\]
\end{lemma}
\begin{proof}
The latter comes from the fact that, for $s\ge 0$ and by symmetry of $\ga$ and $\phi$,  
\begin{align*}
\P( |X_i|\geq s \:|\: \theta_i = 0)&= \P( |\varepsilon_1|\geq s) = 2 \ol{\Phi}(s),\\
\P( |X_i|\geq s \:|\: \theta_i \neq 0)&= \int \P( |\varepsilon_1+u|\geq s) \gamma(u) du = \int (\ol{\Phi}(s-u)+ \ol{\Phi}(s+u)) \gamma(u) du\\
&=2 \int  \ol{\Phi}(s-u) \gamma(u) du = 2 \int \int \mathds{1}_{\{s-x\leq u\}} \gamma(u) du \phi(x)  dx \\
&= 2 \int \int \mathds{1}_{\{s\leq v\}} \gamma(v-x) dv \phi(x)  dx = 2 \int \mathds{1}_{\{s\leq v\}} g(v)dv.
\end{align*}
\end{proof}

\begin{lemma}\label{lem:noninclandq}
For any fixed $x\in\R$, the $\l$-value functional $\l(x;w,g)$ \eqref{lformula} and the $q$-value functional $q(x;w,g)$ \eqref{qvaluesfunc} are both nonincreasing in $w$.
\end{lemma}

\begin{proof}
This is immediate from their explicit expression.
\end{proof}

\begin{lemma}
Under \eqref{tails}, $\log \ol{G}$ is Lipschitz on $\RR^{+}$
\end{lemma}

\begin{proof}
We have $(\log \ol{G})'=-g/\ol{G}$. Now using \eqref{tails}, we have  $(g/\ol{G})(x)\asymp x^{1-\kappa}$ ($x\to\infty$). This provides that $(\log \ol{G})'$ is a bounded function. \end{proof}

\begin{lemma}\label{lemmetrivialisant}
Assumption \eqref{increasing} implies \eqref{increasing2}.
\end{lemma}

\begin{proof}
Let us consider the function 
$$
\Psi:u\in(0,1/2) \to \ol{G}(\ol{\Phi}^{-1}(u)) = \int_{\ol{\Phi}^{-1}(u)}^\infty g(x)dx.
$$
This defines a continuous function on $[0,1/2)$ by setting $\Psi(0)=0$.
For all $u\in(0,1/2)$, we have 
$
\Psi'(u) = \frac{g}{\phi}(\ol{\Phi}^{-1}(u))
$, which means by \eqref{increasing} 
 that $\Psi'$ is decreasing on $(0,1/2)$ and therefore $\Psi$ is strictly concave on $(0,1/2)$. 
This implies that $u\in(0,1/2) \to \Psi(u)/u$ is decreasing and thus that $x\in \R_+ \to\ol{G}(x)/\ol{\Phi}(x)$ is increasing by letting $u=\ol\Phi(x)$, $x>0$. 
Moreover, since $\infty= \lim_{u\to 0^+}\Psi'(u) =\lim_{u\to 0^+}\Psi(u)/u = \lim_{x\to \infty} \ol{G}(x)/\ol{\Phi}(x) $ and $\ol{G}(0)/\ol{\Phi}(0)=1$,  \eqref{increasing2} is proved.
\end{proof}

\begin{lemma}\label{lem:qlvalue}
Assume that $g$ comes from \eqref{log-lip}--\eqref{increasing}.
For $w\in [0,1]$, the functions $x\to \l(x;w,g)$ and $x\to q(x;w,g)$ are symmetric and decreasing on $\R_+$.
For all $x\in\R$, $w\in [0,1]$, we have $q(x;w,g)\leq \l(x;w,g)$. In particular, $q_i(X)\leq \ell_i(X)$ almost surely.
\end{lemma}

\begin{proof}
The first claim comes from the explicit expressions of $\l(x;w,g)$ and $q(x;w,g)$ together with \eqref{increasing} and \eqref{increasing2}, respectively.
Now, denoting $\P$ the probability operator in the Bayesian setting, a simple relation is that for all $x\in\R$, 
\begin{align*}
q(x;w,g)&=\P (\theta_i=0 \:|\: |X_i|\geq |x|) \\
&=\E( \mathds{1}\{\theta_i=0\} \:|\: |X_i|\geq |x|)\\
&=\E[ \,\P(\theta_i=0 \:|\: X_i)\:|\: |X_i|\geq |x|]\\
&=\E[ \ell_i(X)\:|\: |X_i|\geq |x|]\\
&\leq \l(x;w,g),
\end{align*} 
by using the monotonicity of $x\to \l(x;w,g)$.
\end{proof}

Figure~\ref{fig:gGbar} shows how the choice of the prior influences the quantities $g$ and $\ol{G}$. The Laplace calculations are done thanks to Remark~\ref{rem:explicit}. Strikingly, while the quantities $g$ stays of the same order (which guided the choice $a=1/2$), the difference for $\ol{G}$ is more substantial.

\begin{center}
\begin{figure}[h!]
\includegraphics[scale=0.3]{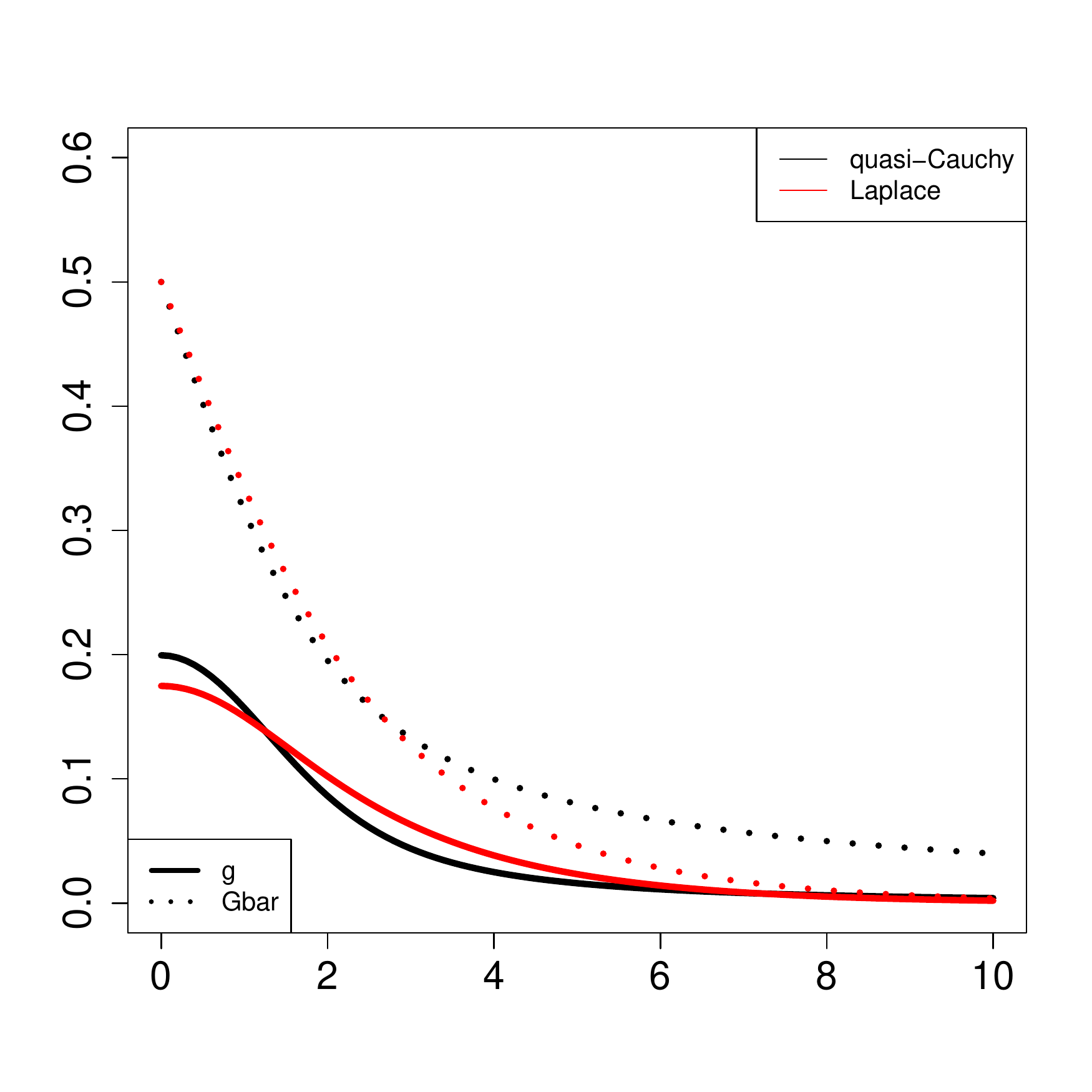}
\includegraphics[scale=0.3]{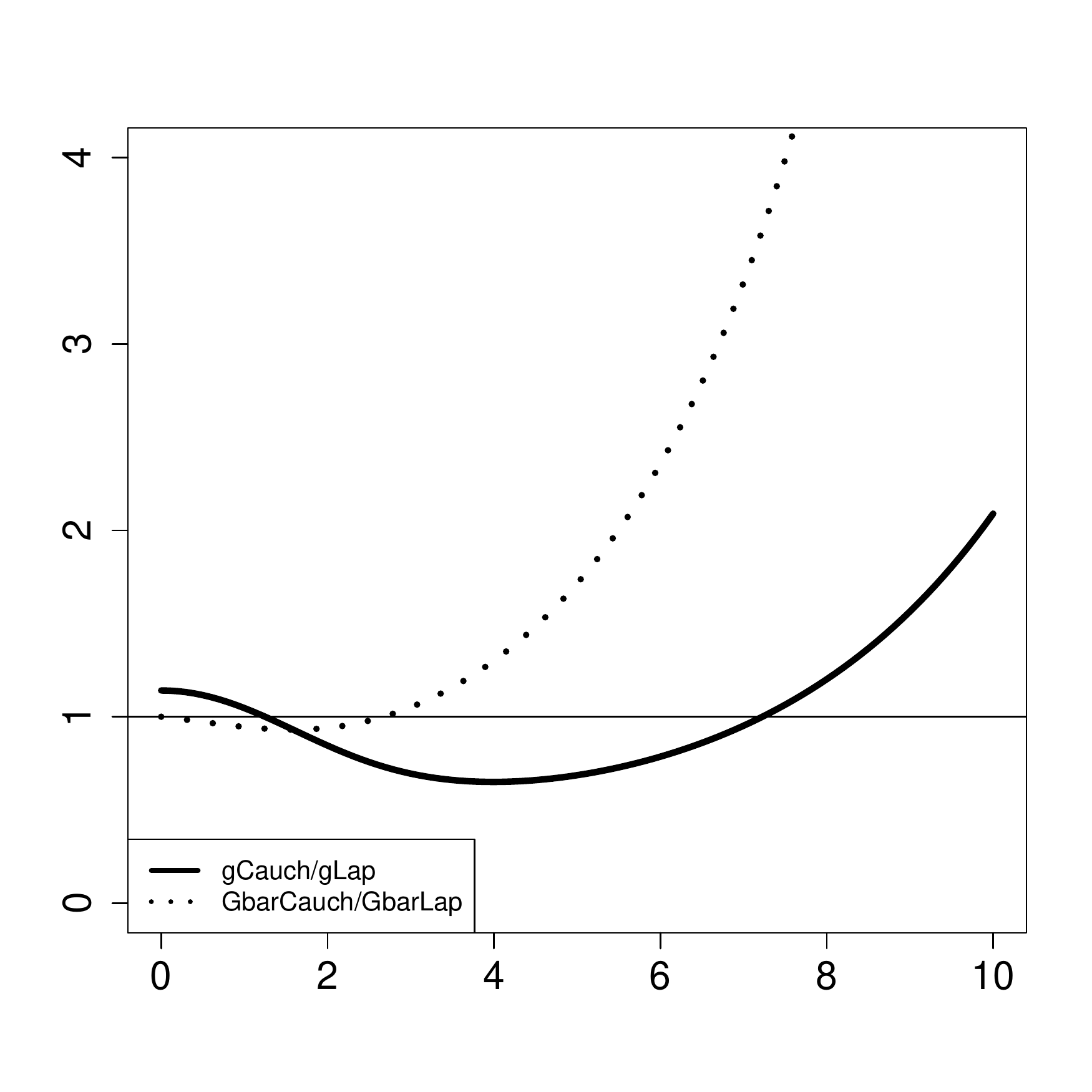}
\caption{\label{fig:gGbar}
Plots of the functions $g$ and $\ol{G}$ for the quasi-Cauchy and Laplace ($a=1/2$) priors respectively (left) and ratio (right).
}
\end{figure}
\end{center}

Figure~\ref{fig:lqvaluew} below shows how the parameters $w$ and $g$ interplay in the quantities $q(x;w,g)$ and $\l(x;w,g)$: for  large values of $|x|$ (which play a central role in the multiple testing phase), the quantity $\l(x;w,g)$ decreases as the prior puts its mass away from $0$, that is, making the tail distribution heavier or increasing $w$. 

\begin{figure}[h!]
\begin{tabular}{ccc}
\vspace{-0.5cm}
$w=0.01$  
& 
$w=0.2$\\
\includegraphics[scale=0.3]{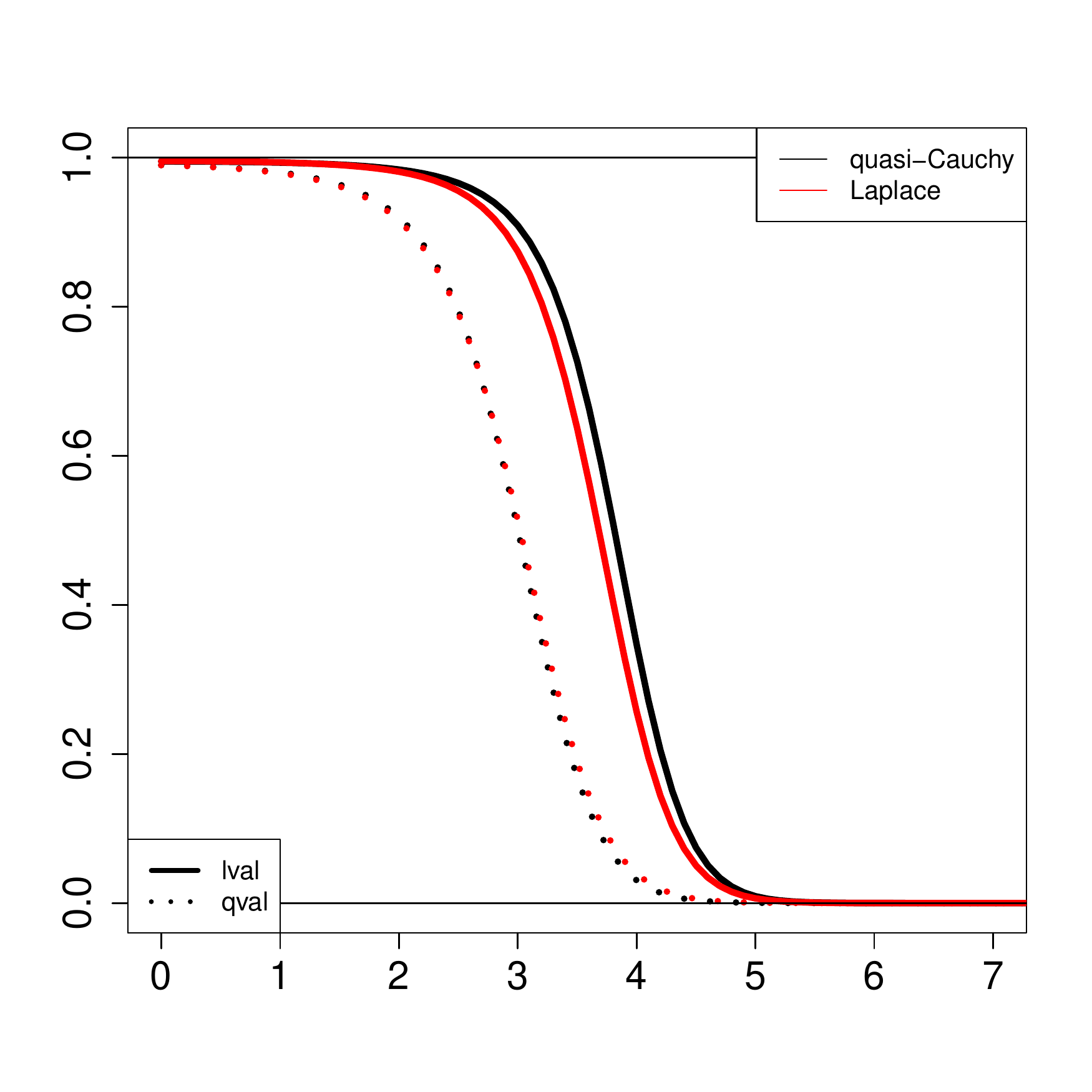}
&\includegraphics[scale=0.3]{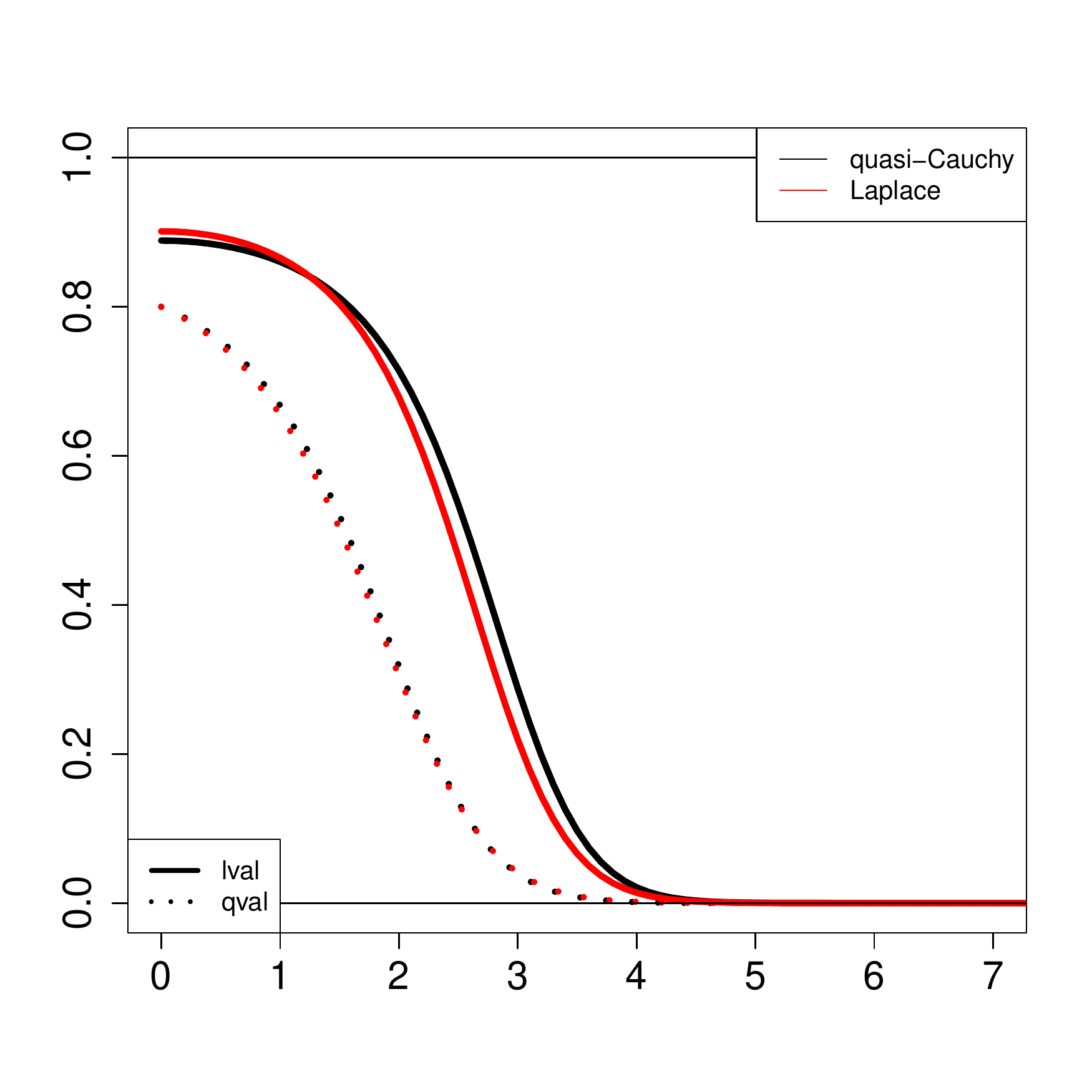}\\
\includegraphics[scale=0.3]{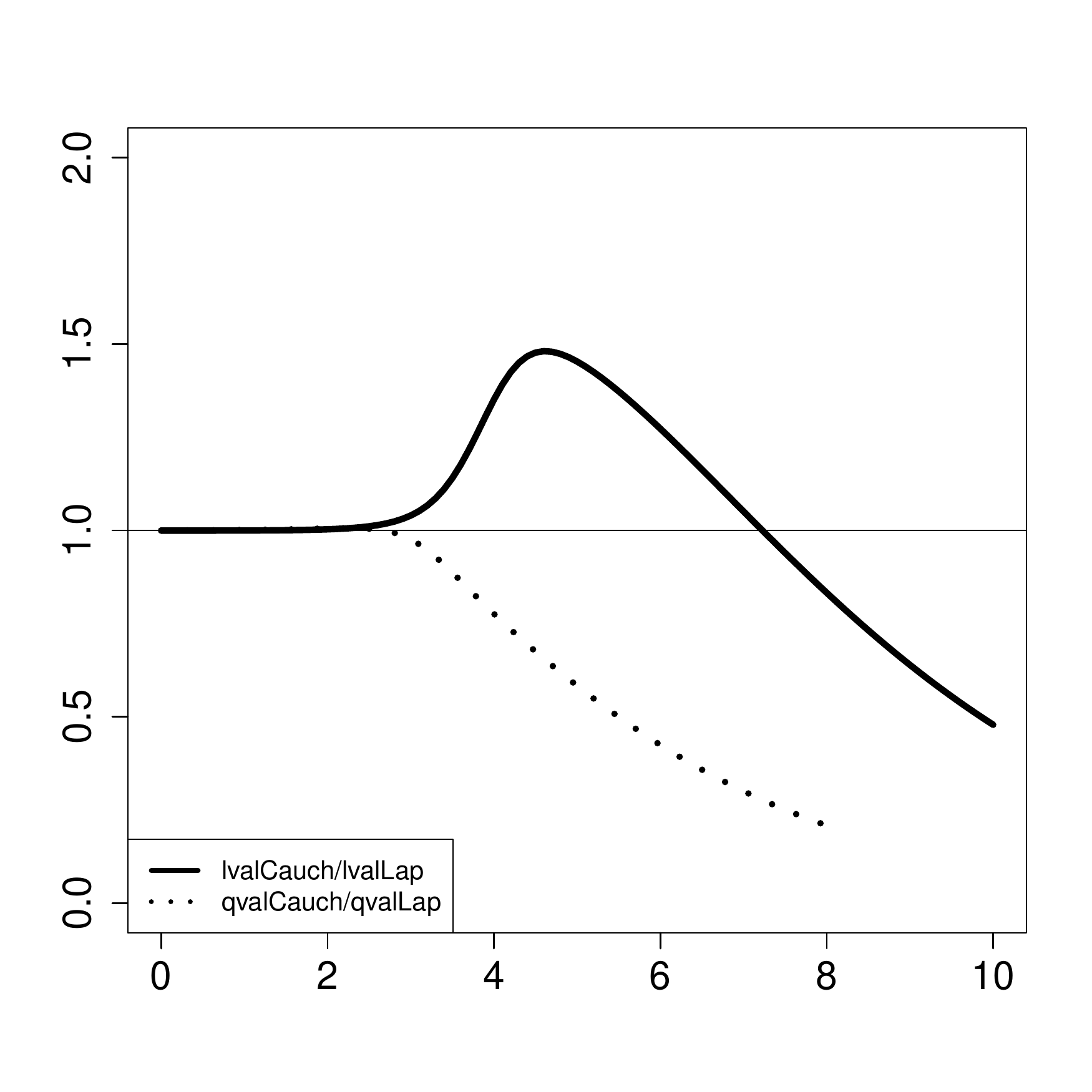}
&\includegraphics[scale=0.3]{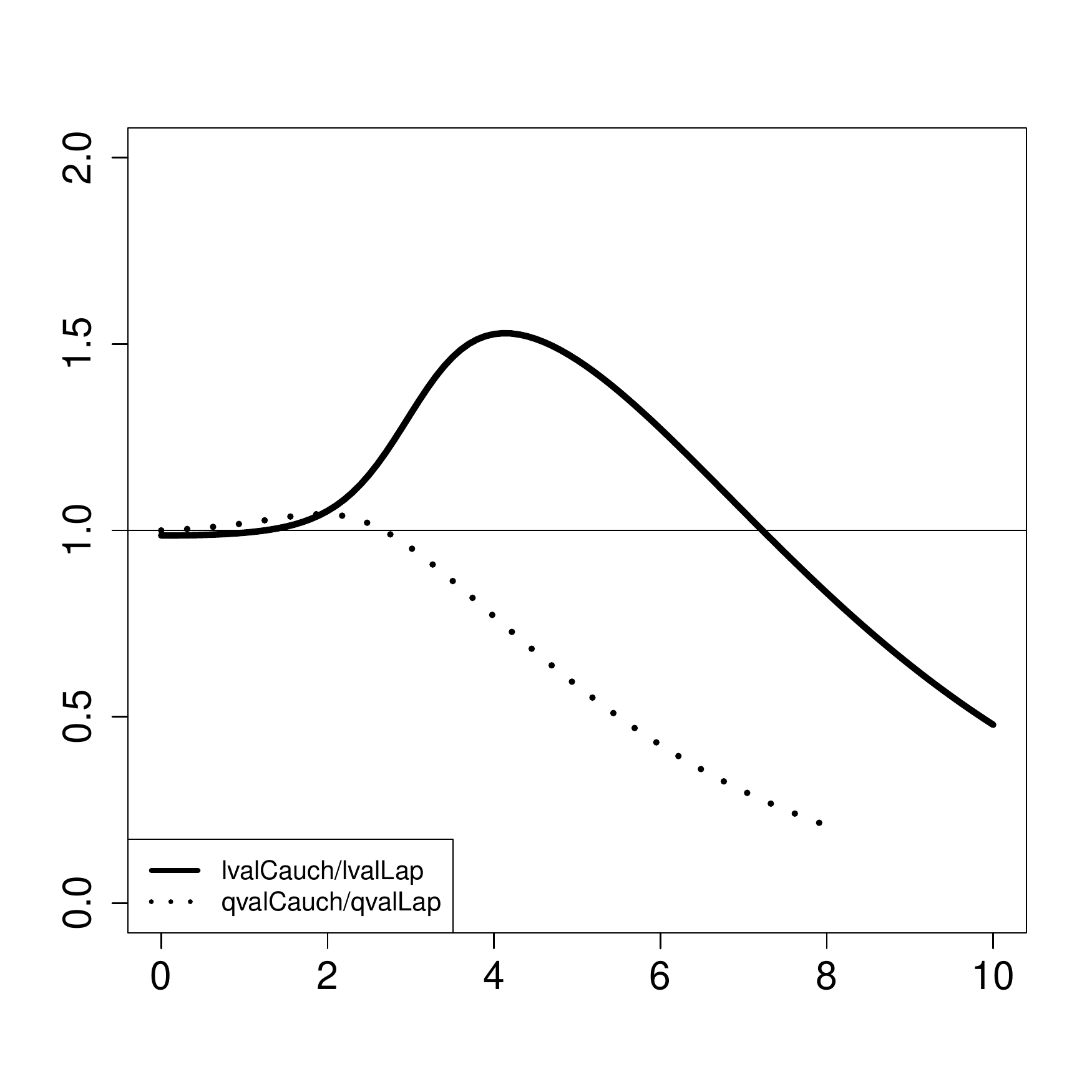}
\end{tabular}
\caption{\label{fig:lqvaluew}
Plot of the functions $x\to \ell(x,g,w)$ and $x\to q(x,g,w)$ for different values of $w$ and $g$ (see text, top) and ratio (bottom). 
}
\end{figure}

\begin{remark}[Explicit expressions for Laplace prior]\label{rem:explicit}
The Laplace prior of parameter $a>0$ is given by 
\begin{equation}
\gamma(x)= \gamma_a(x)=(a/2)\:e^{-a |x|}, \quad x\in\R.\label{equ:Laplaceprior}
\end{equation}
Straightforward calculations show, for $\ga$ as in \eqref{equ:Laplaceprior},
\begin{align*}
g(x)&= (a/2) e^{a^2/2}\left(e^{-ax} \ol{\Phi}(a-x) + e^{ax} \ol{\Phi}(a+x) \right) ;
\\
g(x)/\phi(x) &= (a/2) \left( \frac{\ol{\Phi}(a-x)}{\phi(a-x)} +  \frac{\ol{\Phi}(a+x)}{\phi(a+x)} \right) ;
\\
\ol{G}(x)&= (1/2)\:e^{a^2/2} \left(e^{-ax} \ol{\Phi}(a-x) - e^{ax} \ol{\Phi}(a+x) \right) + \ol{\Phi}(x).
\end{align*}
\end{remark}

\section{Threshold properties}\label{sec:thresholdprop}

We henceforth assume that $g$ satisfies \eqref{assumpgdebase}--\eqref{increasing}. In this section, all the non-universal constants appearing in the results depend on $g$.

\subsection{Link between $\xi$, $\chi$ and $\zeta$} Recall the definitions \eqref{equxi}-\eqref{zeta}-\eqref{equchi} of the thresholds $\xi,\zeta, \chi$.
We start by a simple connection between  $\zeta$ and $\xi$. Namely, 
$$
\frac{\phi(\zeta)}{g(\zeta)} = \frac{1}{\beta(\zeta)+1}= 1/(1/w+1)=w/(1+w), 
$$
so 
\begin{equation}\label{zetatoxi}
\zeta(w)= (\phi/g)^{-1}(w/(1+w))=\xi(w/(1+w)),
\end{equation}
 which implies in particular that $\zeta(w) \ge \xi(w)$.
The next lemma relates these quantities to $\chi(w)$.

\begin{lemma}\label{propchizetabase}
For any $w\in (0,1)$, we have $ \chi(w)\le \xi(w) \leq \zeta(w). $
 \end{lemma}
\begin{proof}
From the proof of Lemma \ref{lemmetrivialisant}, by concavity 
$\ol{G}(\ol{\Phi}^{-1}(u))/u\geq 
 \frac{g}{\phi}(\ol{\Phi}^{-1}(u))$ holds  for any $u\in(0,1/2)$.
 Any $x>0$ can be written $\ol\Phi^{-1}(u)$ for $u\in(0,1/2)$, so for such $x$ we have $(\ol{\Phi}/\ol{G})(x)\le (\phi/g)(x)$. As $\ol{\Phi}/\ol{G}$ is decreasing by \eqref{increasing2}, so is its reciprocal, which implies $x\ge  (\ol\Phi/\ol{G})^{-1}((\phi/g)(x))$. The inequality follows by setting $x=(\phi/g)^{-1}(w)=\xi(w)$.   
\end{proof}

\subsection{Bounds for $\xi$, $\chi$ and $\zeta$} 
\begin{lemma}\label{prop:BMTlval}
Consider $\xi$ as in \eqref{equxi}. Then for $C=(2\pi)^{1/2}\|g\|_\infty$ we have for $u\in(0,1]$ small enough,
\begin{align} 
 \xi(u)&\geq \left(-2\log u-2\log g\left(\sqrt{-2  \log ( C u)}\right)-\log(2\pi)\right)^{1/2} ;\label{rightlvalues}\\
 \xi(u)&\leq  \left(-2\log u-2\log g\left(\sqrt{-4 \log u}\right)-\log(2\pi)\right)^{1/2}. \label{leftlvalues}
\end{align}
We also have the following sharper bound: for $u\in(0,1]$ small enough,
\begin{align} 
 \xi(u)&\leq  \left(-2\log u-2\log g\left(\left(-2\log u + 5 \Lambda (-  \log u)^{1/2}\right)^{1/2}\right)-\log(2\pi)\right)^{1/2}. \label{leftlvaluessharp}
\end{align}

In particular, $ \xi(u)\sim \left(-2\log u\right)^{1/2}$ when $u$ tends to zero.
\end{lemma}

\begin{proof}
Now fix $u\in(0,1]$. Since $\phi(\xi(u))= g(\xi(u)) u$, we have $\phi(\xi(u)) \leq  \|g\|_\infty u$ which implies $\xi(u)\geq \sqrt{-2 \log ( C u)}$, so  $g(\xi(u)) \leq  g\left(\sqrt{-2 \log ( C u)}\right)$ for $u$ small enough. This in turn implies $\phi(\xi(u))\leq  u g\left(\sqrt{-2 \log ( C u)}\right)$ and thus \eqref{rightlvalues}.
Conversely, using \eqref{log-lip}, $g(|x|)\geq g(0) e^{-\Lambda |x|}$ 
for all $x\in\R$ 
and thus $\phi(|x|)/g(|x|)\leq (g(0)\sqrt{2\pi})^{-1} e^{-x^2/2} e^{\Lambda |x|} \leq e^{-x^2/4}$ for $|x|$ larger than a constant, 
which in turn provides $|x|\leq \sqrt{-4\log (\phi(|x|)/g(|x|))}$ and thus 
$\phi(|x|)/g(|x|)\leq (g(0)\sqrt{2\pi})^{-1} e^{-x^2/2} e^{\Lambda \sqrt{-4\log (\phi(|x|)/g(|x|))}}$.
On the one hand, this gives that if $u$ is small enough, 
$\phi(\xi(u)) \geq g(0) u e^{-\Lambda \sqrt{-4\log u}} $, so
\begin{align} 
\xi(u)& \leq \left(-2\log u+4 \Lambda (-  \log u)^{1/2}-2\log g(0)-\log(2\pi)\right)^{1/2}\nonumber\\ 
&\leq  \left(-2\log u+5 \Lambda (- \log u)^{1/2}\right)^{1/2}.
\label{inter-xi}
\end{align}
As $g$ decreases on a vicinity of $\infty$, we have 
$
g(\xi(u))\geq g\left(\sqrt{-4\log{u}}\right)
$ for $u$  small enough. Hence,
$$
\phi(\xi(u))
 \geq 
 (\phi(\xi(u))/g(\xi(u)))  \:g\left(\sqrt{-4\log{u}}\right) = u \:g\left(\sqrt{-4\log{u}}\right),
 $$
which leads to  \eqref{leftlvalues}. To get \eqref{leftlvaluessharp} we use the same reasoning as above with the bound \eqref{inter-xi} instead of $\sqrt{-4\log{u}}$.
\end{proof}

\begin{lemma}\label{prop:BMTqval}
Consider  $\chi$ as in \eqref{equchi}. Then we have for all $u\in(0,1]$,
\begin{align} 
 \chi(u)&\geq \ol{\Phi}^{-1} \left(u\: \ol{G}\left(\ol{\Phi}^{-1}(u)\right)\right);\label{rightqvalues}\\
  \chi(u)&\leq \ol{\Phi}^{-1}  \left(u \:\ol{G}\left(\left(-2\log u+4 \Lambda (-  \log u)^{1/2}+C\right)^{1/2}\right)\right) \mbox{ for $u$ small enough},\label{leftqvalues2}
\end{align}
and  $C=-2\log g(0)-\log(2\pi)$.
We also have the following sharper bound: for some constant $C'>0$, for $u\in(0,1]$ small enough,
\begin{align} 
 \chi(u)&\geq  \left(-2\log \left(u\ol{G}\left(\ol{\Phi}^{-1}(u)\right)\right) -\log \log (1/u)-C'\right)^{1/2}. \label{lowchisharp}
\end{align}
\end{lemma}

\begin{proof}
Let $u\in(0,1]$.  Since $\ol{\Phi}(\chi(u))= \ol{G}(\chi(u)) u$, we have $\ol{\Phi}(\chi(u))\leq u$ and thus $\chi(u)\geq \ol{\Phi}^{-1}(u)$, which in turn implies
$\ol{\Phi}(\chi(u))\leq \ol{G}(\ol{\Phi}^{-1}(u)) u$ and \eqref{rightqvalues}. 
Conversely, as $\chi\le \xi$ by Lemma~\ref{propchizetabase}, 
using the bound on $\xi(u)$ just above \eqref{inter-xi} in the proof of Lemma \ref{prop:BMTlval},
\begin{align*}
\chi(u)\leq \xi(u)&\le \left(-2\log u+4 \Lambda (-  \log u)^{1/2}-2\log g(0)-\log(2\pi)\right)^{1/2},
\end{align*}
so the relation $\chi(u)= \ol{\Phi}^{-1}(\ol{G}(\chi(u)) u)$ leads to \eqref{leftqvalues2}. 
Let us now prove \eqref{lowchisharp}. First observe, by using \eqref{tails3}, that 
$
\ol{G}(\chi(u))\geqa e^{-\Lambda \chi(u)}$.
 Next using the upper bound \eqref{inter-xi} on $\xi\geq \chi$ leads to  $u\ol{G}(\chi(u))\geq u^2$
for $u$ small enough. 
Now, by the second part of Lemma~\ref{bphi},  for $u$ small enough,
\begin{align*}
\chi(u)&= \ol{\Phi}^{-1}(\ol{G}(\chi(u)) u)\\
&\geq \left\{ 2\log (1/\{u\ol{G}(\chi(u))\}) - \log \log (1/\{u\ol{G}(\chi(u))\})-C\right\}^{1/2}\\
&\geq \left\{ 2\log (1/\{u\ol{G}(\chi(u))\}) - \log \log (1/u^2)-C\right\}^{1/2},
\end{align*}
for some constant $C>0$, which gives the result.

\end{proof}

\begin{lemma} \label{lemzetagen}
Consider $\zeta$ as in \eqref{zeta}. Then for a constant $C>0$, we have for $w$ small enough,
\begin{align} 
 \zeta(w)&\geq \left(-2\log w-2\log g\left(\sqrt{-2  \log ( C w)}\right)-\log(2\pi)\right)^{1/2} ;\label{zetamin2}\\
 \zeta(w)&\leq  \left(-2\log w-2\log g\left(\sqrt{-5 \log w}\right)+C\right)^{1/2}. \label{zetamaj2}
\end{align}
We also have the following sharper bound: for $w\in(0,1]$ small enough,
\begin{align} 
 \zeta(w)&\leq  \left(-2\log w-2\log g\left(\left(-2\log w + 6 \Lambda (-  \log w)^{1/2}\right)^{1/2}\right)+C\right)^{1/2}. \label{zetamaj3}
\end{align}
In particular,  $\zeta(w) \sim  \left(-2\log w\right)^{1/2}$ as $w$ tends to zero.
\end{lemma}
\begin{proof}
The result follows from Lemma~\ref{prop:BMTlval}, combined with the relations $\zeta(w) \ge \xi(w)$  and $\zeta(w)=\xi(w/(1+w))$ established above.
\end{proof}

\subsection{Relations between $\xi(r(w,t))$, $\chi(r(w,t))$ and $\zeta(w)$}

Let us recall the definition $r(w,t)=wt/\{(1-w)(1-t)\}$, see \eqref{equinterm}. 

\begin{lemma}\label{propchizeta}
For any $t\in (0,1)$, for $\omega_0=\omega_0(t)$ small enough,  for all $w\leq \omega_0$, we have
$ \chi(r(w,t))\leq \zeta(w). $
 \end{lemma}
\begin{proof}
 Denote by $T(u)=\left(-2\log u+4 \Lambda (-  \log u)^{1/2}+C\right)^{1/2}$ the term appearing in \eqref{leftqvalues2}.
 By \eqref{leftqvalues2} and Lemma~\ref{bphi},  for $u$ small enough,
\begin{align*}
\chi(u)&\leq \ol{\Phi}^{-1}  \left(u \:\ol{G}\left(T(u)\right)\right)\\
&\leq \left\{ \left(2\log (1/u) - 2\log \ol{G}\left(T(u)\right) - \log \log (1/u)\right) \right\}^{1/2}.
\end{align*}
Now using that $\ol{G}(y)\geq D\: g(y)$ for $y$ large enough (see \eqref{tails}), we have for $u$ small enough,
\begin{align*}
\chi(u)^2&\leq  2\log (1/u) -2\log D- 2\log g\left(T(u)\right) - \log \log (1/u).
\end{align*}
Hence, for $w$ small enough, denoting $R=(1-t)(1-w)/t$ and  recalling $r(w,t)=w/R$ via \eqref{equinterm}, and using \eqref{zetamin2} together with assumption \eqref{log-lip},
\begin{align*} 
&\chi(r(w,t))^2 -\zeta(w)^2 \\
&\leq  2\log (1/r(w,t)) -2\log D- 2\log g\left(T\left( r(w,t) \right)\right) - \log \log (1/r(w,t))\\
&+2\log w+2\log g\left(\left\{-2  \log ( C w)\right\}^{1/2}\right)+\log(2\pi)\\
\leq  \:&2\log R + 2\Lambda \left|\left\{-2  \log ( C w)\right\}^{1/2}-   T\left( r(w,t) \right)\right| - \log \log (1/r(w,t))+C',
\end{align*}
for some constant $C'>0$.
Now using $|\sqrt{a}-\sqrt{b}|=|a-b|/(\sqrt{a}+\sqrt{b})$ one gets, for $w$ small enough,
\begin{align*}
\lefteqn{\left|\left\{-2  \log ( C w)\right\}^{1/2} - T\left( r(w,t) \right)\right| }\\
& \le \frac{\left| 2\log(r(w,t)/(Cw)) - 4\La\left(-\log r(w,t)\right)^{1/2} -C\right|}
{\{2\log(1/(Cw))\}^{1/2}} \\
& \leq C'_1 
\left(\frac{|\log ((1-t)/t)|}{\left(\log 1/w\right)^{1/2}}+1\right)
\end{align*}
As a result, for $w$ small enough and smaller than a threshold 
$
\omega_0(t)
$ (depending on $t$ in a way such that  
$\log(1/w)\ge \log^2((1-t)/t)$ as well as $\log\log(1/w)\ge 2\log{R}+C''$ for a large enough constant $C''>0$) 
we have $\chi(r(w,t))^2 -\zeta(w)^2\leq 0$ and the result holds.
\end{proof}

\begin{lemma} \label{lemxizeta} 
There exists some constant $C=C(g)>0$ such that for all $t\in(0,1)$ there exists $\omega_0(t)$ such that for all $w\leq \omega_0(t)$,
\begin{align}
 |\zeta(w) - \xi(r(w,t))| & \le \frac{2\big|\log\big(\frac{t}{1-t} \big)\big|+C}{\zeta(w)+\xi(r(w,t))}. 
\label{relzetaxi}
\end{align}
Furthermore, for all $\eps>0$ and $t\in(0,1)$, there exists $\omega_0(t,\eps)$ such that for 
$w\leq \omega_0(t,\eps)$, 
\begin{align}
\frac{g(\xi(r(w,t)))}{g(\zeta(w))}&\le 1+\veps;\label{gxirsurg}\\
\frac{\ol{G}(\chi(r(w,t)))}{\ol{G}(\zeta(w))}  &\le 1+\veps\label{GchirsurGzeta}.
\end{align} 
\end{lemma}

\begin{proof}  
Let us set \[ S_1(w)= \left(-2\log w + 6 \Lambda (-  \log w)^{1/2}\right)^{1/2}\]
 and $S_2(w)= \sqrt{-2  \log ( C w)}$ the terms appearing  in the  bounds \eqref{zetamaj3} and \eqref{rightlvalues}, respectively. Using these bounds, one obtains
\begin{align*}
&\zeta(w)^2 - \xi(r(w,t))^2 \\
&\leq  \:2\log (r(w,t)/w)  + 2\log g(S_2(r(w,t)))- 2\log g(S_1(w)) +D \\
&\leq 2\big|\log\big(t/(1-t) \big)\big| + D' ,
\end{align*}
for $w$ smaller than a threshold depending on $t$, by using that $\log{g}$ is Lipschitz and proceeding as in the proof of Lemma~\ref{propchizeta} to bound the difference $|S_1(w)-S_2(r(w,t))|$ by a universal constant. Conversely, by using \eqref{leftlvaluessharp} and \eqref{zetamin2}, we have, with $S_3(w)$ as $S_1(w)$ except that $6\Lambda$ is replaced by $5\Lambda$ and $S_4(w)$ as $S_2(w)$ with $C$ as in \eqref{zetamin2}, 
\begin{align*}
& \xi(r(w,t))^2 - \zeta(w)^2  \\
&\leq  -\:2\log (r(w,t)/w)  - 2\log g\left(S_3(w)\right) + 2\log g\left(S_4(w)\right) +D'' \\
&\leq 2\big|\log\big(t/(1-t) \big)\big| + D''',
\end{align*}
as above, which leads to \eqref{relzetaxi} by using $a^2-b^2=(a-b)(a+b)$.
Next, \eqref{gxirsurg} is a direct consequence of \eqref{relzetaxi} by using that $\log g$ is Lipschitz. Finally, let us prove  \eqref{GchirsurGzeta}. 
By Lemma~\ref{propchizeta} and the  bounds \eqref{zetamaj3} and \eqref{lowchisharp}, we have
for $w\leq w_0(t)$ and $S_1(w)$ as above,
\begin{align*} 
0\leq \:&\zeta(w)^2-\chi(r(w,t))^2\\
\leq \:& -2\log w-2\log g\left(S_1(w)\right)+C\\ 
 &+2\log \left\{r(w,t)\ol{G}\circ\ol{\Phi}^{-1}(r(w,t))\right\} +\log \log\{1/r(w,t)\}+C'\\
\leq \:& |2\log (t/(1-t))|+D+\log \log\{1/r(w,t)\}+2\log\Big\{\frac{ \ol{G}\circ\ol{\Phi}^{-1}(r(w,t))}{g\left(S_1(w)\right)}\Big\}.
\end{align*} 
Next, we have 
\begin{align*}
\log\Big\{\frac{ \ol{G}\circ\ol{\Phi}^{-1}(r(w,t))}{g\left(S_1(w)\right)}\Big\} &= \log\Big\{\frac{ \ol{G}\circ\ol{\Phi}^{-1}(r(w,t))}{\ol{G}\left(S_1(w)\right)}\Big\}+\log \Big\{ \frac{ \ol{G}\left(S_1(w)\right)}{g\left(S_1(w)\right)}\Big\}.
\end{align*}
The first term is bounded by a constant, by an argument similar to the proof of Lemma \ref{propchizeta}, as $\log\ol{G}$ is Lipschitz. For the second term, by \eqref{tails}, 
\begin{align*}
\log \Big\{ \frac{ \ol{G}\left(S_1(w)\right)}{g\left(S_1(w)\right)}\Big\}\leq \log S_1(w).
\end{align*}
This gives, upon dividing by $\zeta(w)+\chi(r(w,t)$ the obtained inequality on $\zeta(w)^2-\chi(r(w,t)^2$, that $|\zeta(w)-\chi(r(w,t)|$ is arbitrary small when $w$ is small, which 
 leads to \eqref{GchirsurGzeta}
by using again that $\log\ol{G}$  is Lipschitz.
\end{proof}

\begin{lemma} \label{lembfi}
There exists a constant $C=C(g)>0$ such that for all $t\in(0,0.9)$ there exists $\omega_0(t)$ such that for  $w\leq \omega_0(t)$ and $\mu\in\R$, 
\begin{align}
 \bfi(\xi(r(w,t))-\mu) & \ge C\:t\: \bfi(\zeta(w)-\mu) .\label{xizetawithmu}
\end{align} 
\end{lemma}

\begin{proof}
By Lemma~\ref{lemxizeta}, for small $w$,  $|\zeta(w) - \xi(r(w,t))|\leq 1/4$. Hence, we can apply Lemma~\ref{lemphixphiy}, which gives
\begin{align*}
\frac{\ol{\Phi}(\xi(r(w,t))-\mu)}{\ol{\Phi}(\zeta(w)-\mu)} &\geq  \frac{1}{4} e^{-|\xi(r(w,t)^2)-\zeta(w)^2|/2}\\
&\geq C e^{-\big|\log\big(\frac{t}{1-t} \big)\big|},
\end{align*}
by using again \eqref{relzetaxi}. This shows the desired result.
\end{proof}

\subsection{Variations of certain useful functions}

For any $w\in(0,1)$ and $\mu\neq 0$, let us denote 
\begin{equation}\label{Tmu}
T_\mu(w) =
1+\frac{|\zeta(w)-|\mu||}{|\mu|}.
\end{equation}

\begin{lemma} \label{lemTmuw}
First, for all $\eps\in(0,1)$, for any 
   $z\ge 1$, there exists $\omega_0=\omega_0(z,\eps)\in (0,1)$, such that for all $w\leq \omega_0$, 
\begin{equation}\label{equGbarwM} 
\left\{\begin{array}{c}1-\eps\:  \leq g(\zeta(w/z))/g(\zeta(w)) \leq 1\\ 
1-\eps\:  \leq \ol{G}(\zeta(w/z))/\ol{G}(\zeta(w)) \leq 1.
\end{array}\right.
\end{equation}
Second, for any $K\ge 1$, one can find $d_1=d_1(K)$ and $d_2=d_2(K)>0$ such that for all $z\geq 1$, for $w\leq \omega_0=\omega_0(z,1/2)$ as before and $|\mu|>\zeta(w)/K$, 
\begin{equation}\label{equTmu} 
   d_1 \: \leq T_\mu(w/z)/T_\mu(w) \leq d_2.
   \end{equation}
\end{lemma}

\begin{proof}
Since $\log g$ and $\log \ol{G}$ are Lipschitz and by monotonicity, it is sufficient to bound $\zeta(w/z)-\zeta(w)$ from above. For this, we combine \eqref{zetamin2} and \eqref{zetamaj3} to obtain, with $S_1, S_4$ as in the proof of Lemma \ref{lemxizeta},
\begin{align*}
&\zeta(w/z)^2-\zeta(w)^2 \\
\leq \:& 
2\log w+2\log g\left(S_4(w)\right)+\log(2\pi) -2\log (w/z)-2\log g\left(S_1(w/z)\right)+C\\
\leq \:& 2 \log z + D + 2\Lambda \left|S_4(w)-S_1(w/z)\right|,
\end{align*}
by using that $\log g$ is $\Lambda$-Lipschitz by  \eqref{log-lip}.
Since the last bound is bounded by some constant for $w\leq w_0(z)$, we obtain \eqref{equGbarwM}.

To prove \eqref{equTmu},
one notes that since $|\mu|>\zeta(w)/K$, we have $1\le T_\mu(w/z)\le 2+K\zeta(w/z)/\zeta(w)$ which itself is less than $2+K+K(\zeta(w/z)-\zeta(w))/\zeta(w)$. Using the previous bound on $\zeta(w/z)-\zeta(w)$ 
and the fact that $\zeta(w)$ goes to $\infty$ as $w$ goes to $0$ the last bound is no more than a constant $C=C(K)$ whenever $w\le \omega(z,1/2)$. On the other hand, $1\le T_\mu(w)\le 2+K$ for $|\mu|>\zeta(w)/K$. The desired inequality follows.
\end{proof}

 Let us denote, for $w\in(0,1)$ and $\mu\in \R$,
\begin{equation} \label{gmu}
 G_\mu(w) = \frac{\bfi(\zeta(w)-|\mu|)}{w}.
\end{equation} 

\begin{lemma} \label{lemregv}
Consider $G_\mu$ defined by \eqref{gmu}.
For all $K_0>1$ and any $ z\geq 1$, there exists $\omega_0=\omega_0(K_0,z)$ such that for all $w\leq \omega_0$, any $\mu\in\R$ with $ |\mu| \geq \zeta(w)/K_0$, we have
\begin{equation}\label{equGmuwM}
 G_\mu(w/z) \ge z^{1/(2K_0)} G_\mu(w).
\end{equation}
\end{lemma}

\begin{proof}
Let us focus on $\mu\geq 0$ without loss of generality.
Let us rewrite the desired inequality as, with $\Gamma(u)=\log G_\mu(e^{-u})$,
\[ \Ga\left(\log\frac{z}{w}\right)-\Ga\left(\log\frac1w\right) \ge \frac{1}{2K_0}\left(\log\frac{z}{w}-\log\frac{1}{w}\right).\]
To prove this, it is enough to check that $\Ga'(u)\ge 1/(2K_0)$ for  
$u\in[\log\frac{1}{w},\log\frac{z}{w}]$, for appropriately small $w$. To do so, one computes the derivative of $\Ga$ explicitly using the chain rule. First one notes that 
\[ \zeta'(w)=-\frac{1}{w^2 \beta'(\zeta(w))},\]
and from this one deduces that
\[ \Ga'(u) = 1- \frac{e^u}{\beta'( \zeta(e^{-u}) )} \frac{\phi}{\bfi}(\zeta(e^{-u})-\mu).\]
One further computes
$$
\beta'(x)= (\beta(x)+1) x Q(x) , \:\: \mbox{ for } \:Q(x)=1+\frac{(\log g)'(x)}{x},
$$
which gives
 $\be'(\zeta(e^{-u}))=\zeta(e^{-u})Q(\zeta(e^{-u}))(\beta(\zeta(e^{-u}))+1)$. Using the identity $\beta(\zeta(e^{-u}))=e^u$ leads to
\begin{align*}
  \Ga'(u) & = 
  1 - \frac{e^u}{1+e^u} \frac{1}{Q(\zeta(e^{-u}))} \frac{1}{\zeta(e^{-u})} 
  \frac{\phi}{\bfi}(\zeta(e^{-u})-\mu).
\end{align*}
  Now, by using \eqref{log-lip} one sees that the map $u\to e^u(1+e^u)^{-1}Q(\zeta(e^{-u}))^{-1}$ has limit $1$ as $u$ goes to infinity. So, for $u$ large enough, $e^u(1+e^u)^{-1}Q(\zeta(e^{-u}))^{-1}\leq 1+ \eps$ for some $\eps>0$ to be chosen later on.
  Now using Lemma~\ref{bphi}, 
   whenever $\mu\leq \zeta(e^{-u})-1$,
  \begin{align*}
  \frac{1}{\zeta(e^{-u})} 
  \frac{\phi}{\bfi}(\zeta(e^{-u})-\mu) &\leq  \frac{1}{\zeta(e^{-u})}  \frac{1+(\zeta(e^{-u})-\mu)^2}{\zeta(e^{-u})-\mu} \\
   &=  \frac{\zeta(e^{-u})-\mu}{\zeta(e^{-u})} + \frac{1}{\zeta(e^{-u}) (\zeta(e^{-u})-\mu)}.  
   \end{align*}
By definition of $u$, we have $e^{-u}\in[w/z,w]$, so $\zeta(e^{-u})\le \zeta(w/z)$. 
Deduce that, using that by assumption $\mu\ge \zeta(w)/K_0$,
\[ \frac{\zeta(e^{-u})-\mu}{\zeta(e^{-u})}  \le 1-\frac{1}{K_0}\frac{\zeta(w)}{\zeta(w/z)}. \]
The behaviour of the difference $\zeta(w/z)-\zeta(w)$ was studied in the proof of Lemma \ref{lemTmuw} where it is seen that this quantity is smaller  a certain universal  constant if $w$ is small enough. By writing 
\[ \zeta(w/z)/\zeta(w) = \left(1+\frac{\zeta(w/z)-\zeta(w)}{\zeta(w)}\right)^{-1},\]
one gets that this ratio is at least $1-1/8$ for $w$ small enough, using $\zeta(w)\to \infty$ as $w\to 0$. 
This shows that for $w\le \omega(z)$ small enough,
\[\frac{1}{\zeta(e^{-u})} 
  \frac{\phi}{\bfi}(\zeta(e^{-u})-\mu) \leq 1-(1/K_0)(1-1/8) +  \frac{1}{\zeta(e^{-u})},
   \]
where we have used $\zeta(e^{-u})-\mu\ge 1$. 
On the other hand, if $\mu\ge \zeta(e^{-u})-1$, 
   \begin{align*}
  \frac{1}{\zeta(e^{-u})} 
  \frac{\phi}{\bfi}(\zeta(e^{-u})-\mu) &\leq   \frac{\phi(0)}{\bfi(1) \zeta(e^{-u})},
   \end{align*}
 which can be made arbitrarily small for $w$ small enough.   
  As a result, in both cases, for $w\le \omega(K_0,z)$ small enough, for all $\mu \geq \zeta(w)/K_0$,
  $$
  1-\Ga'(u)\leq (1+ \eps) (1-7/8K_0 + 1/(4K_0)) \leq  (1+ \eps) (1-5/(8K_0)) = 1-1/(2K_0)
  $$
  by choosing $\eps^{-1}=8K_0-5$. 
This proves the desired inequality.  
\end{proof}

\section{Moment properties}\label{sec:moment}

The main results in this section concern the moments of the score function,  $\tilde m(w) = - \E_0\beta(X,w)= -\int_{-\infty}^{\infty} \be(t,w)\phi(t)dt$ and 
$m_1(\ta,w)  = \E_\ta[\be(X,w)]$, 
$m_2(\ta,w)  = \E_\ta[\be(X,w)^2]$. 
Remember that $g$ is assumed to enjoy \eqref{assumpgdebase}--\eqref{increasing}.
Also, since these functions only depends on $g$,  all the constants appearing in the results of this section only depend on $g$ (except in Section~\ref{m1tilderatio} where the sparsity comes in). In this section, we freely use $\zeta=\zeta(w)$ as a shorthand notation.

\subsection{Basic lemmas on moments}

The following two lemmas are (mostly) small parts of Lemmas 7--9 in \cite{js04}. We include the proofs for completeness.
\begin{lemma} \label{lembeta}
For $c_1=(-\beta(0))^{-1}-1>0$, for any $x \in\R$ and $w\in(0,1]$,
\begin{equation}\label{equmajbetaxw}
 |\beta(x,w)| \le \frac{1}{w\wedge c_1}.
 \end{equation}
 \end{lemma}
\begin{proof}
It suffices to distinguish the cases $\beta(x)<0$ and $\beta(x)\ge 0$ and to bound 
$|\beta(x,w)|$ by $|\beta(0)|/(1+\beta(0))$ and $1/w$, respectively.
\end{proof}
\begin{lemma} \label{lemm12}
The function $w\in(0,1]\to \tilde{m}(w)$ is continuous, nonnegative, increasing and $\tilde{m}(0)=0$. The map $w\in(0,1]\to m_1(\mu,w)$ is continuous and decreasing. 
In addition, $m_1(\mu,0)>0$ if $\mu\neq 0$ and $\mu\in\R_+ \to m_1(\mu,w)$ is nondecreasing for any $w\in[0,1]$.
Also, there exists a constant  $\omega=\omega(g)$ such that, for  any $w\le \omega$ and any $\mu\in\RR$, 
\[ m_1(\mu,w)\le \frac{1}{w},\quad m_2(\mu,w)\le \frac{1}{w}. \]
\end{lemma}

\begin{proof}
Since $w\to \be(u,w)$ is decreasing (for any $u$ with $\beta(u)\neq 0$), so are $w\to -\tilde{m}(w)$ and $w\to m_1(\mu,w)$ for any real $\mu$. The continuity of $\tilde m$ follows by continuity of $\be(u,w)$ and domination of $\be(u,w)\phi(u)$ by $g(u)+\phi(u)$ (up to a constant). 
In addition, since, as $g$ is a density, $\int \beta(u)\phi(u)du=0$, and we have
\begin{align}\label{mtildepos}
\tilde m(w) &= - \int \frac{\beta(u)}{1+w \beta(u)}\phi(u)du= \int \frac{w\beta(u)^2}{1+w \beta(u)}\phi(u)du.
\end{align}
From this one deduces that $\tilde m$ is nonnegative. 
For $m_1$, the continuity follows by local domination using Lemma \ref{lembeta}. 
Next, if $\mu\neq 0$, say $\mu>0$, we have 
\begin{align*}
m_1(\mu,0) = \int_{-\infty}^{\infty} \beta(u+\mu)\phi(u)du = \int_{-\infty}^{\infty} (\beta(u+\mu)-\beta(u)) \phi(u)du.
\end{align*}
Moreover, by \eqref{increasing}, $u\to \beta(u+\mu)-\beta(u)$ is a positive function.
Since it is also continuous, the integral is positive, which means that $m_1(\mu,0)>0$. 
To see that $\mu\in\R_+ \to m_1(\mu,w)$ is nondecreasing, we compute its derivative 
$$
\frac{\partial m_1(\mu,w)}{\partial \mu} = \int_0^\infty \frac{\partial \{\beta(x)/(1+w\beta(x))\}}{\partial x} (\phi(x-\mu)-\phi(x+\mu)) dx \geq 0.
$$
Finally, the bounds on $m_1,  m_2$ follow from Lemma \ref{lembeta}, with $\omega=c_1$.
\end{proof}

The following is a reformulation of Corollary 1
 in \cite{js04} (see (58) therein). We provide a proof below for completeness.
 
\begin{lemma}\label{equ-beta2maj}
Consider $\Lambda$ as in \eqref{log-lip}. Then for all $z\geq 4\Lambda$ 
and all $\mu\geq 0$,
\begin{equation}
  \int_0^z \left(\frac{g(u)}{\phi(u)}\right)^2\phi(u-\mu)du \le \frac {8}{z} \left(\frac{g(z)}{\phi(z)}\right)^2\phi(z-\mu). \end{equation}
\end{lemma}

\begin{proof}
We have for all $u\in[0,z]$,
\begin{align*}
 \left(\frac{g(u)}{\phi(u)}\right)^2\phi(u-\mu) &= 
 \left(\frac{g(z)}{\phi(z)}\right)^2 \phi(z-\mu)\exp\left\{ -\int_u^z \left[\log \left\{g^2/\phi^2(\cdot)\phi(\cdot-\mu)\right\}\right]'(v)dv\right\}. \end{align*}
Now, by \eqref{log-lip}, for all $v\in[0,z]$ and $\mu\ge 0$,
$$
\left(2\log g - 2 \log \phi +\log \phi(\cdot-\mu)\right)'(v) \geq -2 \Lambda + 2v -(v-\mu) \geq   v-2 \Lambda.
$$
Therefore, inserting the latter in the above display, we obtain
\begin{align*}
 \left(\frac{g(u)}{\phi(u)}\right)^2\phi(u-\mu) &\leq 
 \left(\frac{g(z)}{\phi(z)}\right)^2 \phi(z-\mu)e^{-(z-2\Lambda)^2/2} e^{(u-2\Lambda)^2/2}. \end{align*}
One concludes because letting $s=z-2\Lambda\geq z/2$ and noting that
 \begin{align*}
 e^{-s^2/2} \int_0^z e^{(u-2\Lambda)^2/2} du&\leq e^{-s^2/2} \int_{-s}^{s} e^{t^2/2} dt =2\int_{0}^s e^{-(s-t)(s+t)/2}dt \\
 & \le 2\int_0^s e^{-(s-t)s/2}\le \int_0^\infty e^{-xs/2}dx = 4/s\leq 8/z.
 \end{align*}
\end{proof}

\subsection{Behaviour of $\tilde{m}$}

The next lemma refines  Lemma~7 in \cite{js04}.

\begin{lemma} \label{lemmtilde}
For $\tilde{m}(w)$ defined by \eqref{mtilde},   we have, for $\zeta=\zeta(w)$ and asymptotically as $w\to 0$,
\begin{align}
\frac{\tilde{m}(w)}{2\ol{G}(\zeta)} \sim 1.
  \label{equivmtildeGbar}
\end{align}
In particular, for $\kappa$ as in  \eqref{tails}, as $w\to 0$, $ \tilde m (w) \asymp \zeta^{\kappa-1}g(\zeta)$ and $\tilde{m}(w)\geqa w^c$ for arbitrary $c\in(0,1)$.
\end{lemma}
\begin{proof}
Using \eqref{mtildepos}, symmetry of $\be$ and $\be\phi=g-\phi$ on $[\zeta,\infty)$,
\begin{align}\label{equintermmtilde}
\tilde m(w) &= 2 \int_0^\zeta \frac{w\beta(u)^2}{1+w \beta(u)}\phi(u)du - \int_\zeta^\infty \frac{2w \beta(u)}{1+w \beta(u)} \phi(u)du+ \int_\zeta^\infty \frac{2w \beta(u)}{1+w \beta(u)} g(u)du.
\end{align}
For the first term of \eqref{equintermmtilde}, since for $u\in[0,\zeta]$, $1+w \beta(u) \geq 1+\beta(0)$,
\begin{align*}
2 \int_0^\zeta \frac{w\beta(u)^2}{1+w \beta(u)}\phi(u)du &\leq 2w (1+\beta(0))^{-1} \int_0^\zeta \beta(u)^2\phi(u)du\\
&\leq \frac {C}{\zeta} w \beta(\zeta) (g/\phi)(\zeta) = \frac {C g(\zeta)}{\zeta} ,
\end{align*}
for $C=20/(1+\be(0))$, by Lemma~\ref{equ-beta2maj} ($\mu=0$),  where we use that $\beta(\zeta)\leq (g/\phi)(\zeta) \leq (5/4) \beta(\zeta)$ which holds for $\zeta$ large enough, or equivalently for $w\le \omega_1$ with $\omega_1=\omega_1(g)$ a universal constant.
The second term of \eqref{equintermmtilde} is negative whenever $\zeta>\beta^{-1}(0)$ and of smaller order than the third term. For the third term we use that for $u\geq \zeta$, $w\beta(u)\geq 1$ and thus $1\leq 2 w \beta(u)/(1+w \beta(u)) \leq 2$, hence
\begin{align*}
\ol{G}(\zeta)&\leq \int_\zeta^\infty \frac{2w \beta(u)}{1+w \beta(u)} g(u)du\leq 2\ol{G}(\zeta).
\end{align*}
Now,  by assumption   $\ol{G}(\zeta) \asymp g(\zeta)\zeta^{\kappa-1}$, see \eqref{tails}. Hence, when $w$ is small, the dominating term in \eqref{equintermmtilde} is the third one, which gives 
\begin{equation}
\tilde{m}(w) \sim  \int_\zeta^\infty \frac{2 w \beta(u)}{1+w \beta(u)} g(u) du\label{equivmtilde}
\end{equation}
Now, let us prove 
\begin{equation}
  \int_\zeta^\infty \frac{ w \beta(u)}{1+w \beta(u)} g(u) du \sim \ol{G}(\zeta)\label{equivmtildeinter}
\end{equation}
from which \eqref{equivmtildeGbar} follows.
To prove \eqref{equivmtildeinter}, let us write 
\begin{align*}
 \int_\zeta^\infty \frac{w \beta(u)}{1+w \beta(u)} g(u)du &= \ol{G}(\zeta) - \int_\zeta^\infty \frac{ g(u)}{1+w \beta(u)} du.
\end{align*}
Hence,  we obtain
\begin{align*}
\left|\ol{G}(\zeta)- \int_\zeta^\infty \frac{w \beta(u)}{1+w \beta(u)} g(u)du \right| 
&\leq  \int_\zeta^\infty \frac{ g(u)}{1+w \beta(u)} du\\
&\leq  w^{-1} \int_\zeta^\infty \phi(u) du = \frac{\ol{\Phi}(\zeta)}{w},
\end{align*}
because $1+w \beta(u) = 1-w + w g(u)/\phi(u) \geq w g(u)/\phi(u)$.
Now using that $\ol{\Phi}(\zeta)\sim \phi(\zeta)/\zeta \sim w g(\zeta)/\zeta$ and since  $\ol{G}(\zeta) \asymp g(\zeta)\zeta^{\kappa-1}$ (see \eqref{tails}), the difference in the last display is a $o(\ol{G}(\zeta))$ and \eqref{equivmtildeGbar} is proved.
Then, $\tilde m (w) \asymp \zeta^{\kappa-1}g(\zeta)$ follows from \eqref{tails} and this in turn implies
by \eqref{tailsg} and Lemma \ref{lemzetagen}, 
$\tilde m (w)  \geqa e^{-\Lambda\zeta(w)}\geqa w^{c}$ for any $c>0$. 
\end{proof}

\subsection{Upper bound on $m_1$}

The next lemma refines the bounds on $m_1$ of Lemma 9 in \cite{js04}. The refinement is important in that we obtain a precise upper-bound for any $\mu$ larger than a constant. Moreover, the bound is sharp in this regime of $\mu$'s, as we shall see below. 
\begin{lemma} \label{lemun} 
There exist constants $C>0$ and $\omega_0\in(0,1)$ such that for any $w\le \omega_0$, for any $\mu$ such that $\mu\ge \mu_0:=2\La$, with $T_\mu(w)$ as in  \eqref{Tmu},
\begin{align*} 
m_1(\mu,w) & \le
 C \frac{\ol{\Phi}(\zeta-|\mu|)}{w} T_\mu(w).
\end{align*}
 In particular, $m_1(\mu,w) \le C\zeta^2 \ol{\Phi}(\zeta-\mu)/w$ holds for any $\mu\ge \mu_0$ and $w\le \omega_0$. 
For any $w\le \omega_0$, one also has
\begin{align*}  
   m_1(\mu,w) & \le \frac{C}{|\mu|}e^{-\mu^2/2 + |\mu|\zeta}, & \text{for any  } \ \zeta^{-1} \le |\mu| \le \mu_0, \\
   |m_1(\mu,w)| & \le C(1+\zeta\mu^2),& \text{for any  }\  |\mu| \le\zeta^{-1}.
\end{align*}
 \end{lemma}

Since $T_\mu(w) =1+|\zeta-|\mu||/|\mu|$ can be written  $1 + (\zeta-|\mu|)_+/|\mu| + (|\mu|-\zeta)_+/|\mu| \leq 2 + (\zeta/|\mu|-1)_+ $, we deduce the following corollary.

\begin{corollary} \label{cor:lemun}
There exists $\omega_0 \in(0,1)$ such that 
for any $K>1$, there exist constants $C(K) >0$ such that for any $w\le \omega_0$, for any $\mu$ such that $\mu\ge \zeta/K$, we have 
\begin{align*} 
m_1(\mu,w) & \le
 C(K) \frac{\ol{\Phi}(\zeta-|\mu|)}{w}.
\end{align*}
 \end{corollary}

We now prove Lemma~\ref{lemun}.

\begin{proof}
As $\mu\to m_1(\mu,w)$ is even by symmetry of $\be$ and $\phi$, it suffices to consider the case $\mu\ge 0$. 
For $\mu>\zeta-1$, the result directly follows from the global bound $|m_1(\mu,w)|\le Cw^{-1}$, a consequence of Lemma \ref{lembeta}.
By definition 
\begin{align*}
 m_1(\mu,w) & = \int_{-\infty}^{\infty} \frac{\beta(x)}{1+w\be(x)} \phi(x-\mu)dx \\ 
 & = \int_{-\zeta}^{\zeta} \frac{\beta(x)}{1+w\be(x)} \phi(x-\mu)dx\ +\ \int_{|x|>\zeta} \frac{\beta(x)}{1+w\be(x)} \phi(x-\mu)dx\\
 & = \qquad \qquad (I) \qquad \qquad \qquad\ +\ \qquad \qquad \qquad  (II).
 \end{align*}
We first deal with the term (II), for which $\beta(x)\ge \be(\zeta) \ge 0$ (for small enough universal $\omega_0$), so $(II)\ge 0$, and using $1+w\be(x) \ge w\be(x)$ one obtains
 \[ (II) \quad \le \quad \frac1w \int_{|x|>\zeta} \phi(x-\mu)dx
\le \frac{2}{w} \ol{\Phi}(\zeta-\mu). \]
Now one rewrites (I) as
\begin{align*}
 (I) & = \int_{-\zeta}^{\zeta} \beta(x) \phi(x-\mu)dx 
- w\int_{-\zeta}^{\zeta} \frac{\beta(x)^2}{1+w\beta(x)} \phi(x-\mu)dx  \\
 & \le \int_{-\zeta}^{\zeta} \beta(x) \phi(x-\mu)dx.
\end{align*}
Let us split 
\begin{align*}
 \int_{-\zeta}^{\zeta} \beta(x) \phi(x-\mu)dx & = \int_{|x|\le 1/\mu} \beta(x) \phi(x-\mu)dx +
\int_{1/\mu\le |x| \le \zeta}\beta(x) \phi(x-\mu)dx \\
& = \qquad \qquad(a)\qquad\qquad+\qquad\qquad (b).
\end{align*}
First, the integral (a) can be written, by definition of $\beta$,
\begin{align*}
 \int_{|x|\le 1/\mu} \beta(x) \phi(x-\mu)dx & = 
\int_{-1/\mu}^{1/\mu} (g-\phi)(x)e^{\mu x- \frac{\mu^2}{2}} dx
\end{align*}
Using $|g-\phi|\le \|g-\phi\|_{\infty}\le C$, one gets  $(a)\leqa e^{-\mu^2/2}/\mu$.
For the integral $(b)$, with $\beta(x)\le (g/\phi)(x)$ (note that $\be(x)$ is possibly negative here), 
\begin{align*}
 (b) & \leq \int_{-\zeta}^{-1/\mu} g(x) e^{\mu x-\frac{\mu^2}{2}} dx + 
 \int_{1/\mu}^{\zeta} g(x) e^{\mu x-\frac{\mu^2}{2}} dx\\
 & \leq \int_{1/\mu}^{\zeta} g(x) e^{-\mu x-\frac{\mu^2}{2}} dx 
 + \int_{1/\mu}^{\zeta} g(x) e^{\mu x-\frac{\mu^2}{2}} dx \\
 & \leq 2e^{-\frac{\mu^2}{2}}  \int_{1}^{\mu\zeta} g(t/\mu) e^{t} dt/\mu.
 \end{align*}
From this one deduces the global bound, for $\mu>1/\zeta$,
 \begin{align*} 
  m_1(\mu,w) & \le \frac{2}{w}\ol{\Phi}(\zeta-\mu) + \frac{C}{\mu}\|g\|_\infty e^{-\mu^2/2 + \mu\zeta}  \\
  & \leqa \frac{g(\zeta)}{\phi(\zeta)} \phi(\zeta-\mu) + \frac{1}{\mu}e^{-\mu^2/2 + \mu\zeta} \leqa (\|g\|_\infty + \mu^{-1})e^{-\mu^2/2 + \mu\zeta},
 \end{align*} 
 which leads to the second inequality of the lemma. Now turning to the first inequality, an integration by parts  gives, with $0\le -g'/g\le \La$ from  
\eqref{log-lip},
\begin{align*}
  \int_{1}^{\mu\zeta} g(t/\mu) e^{t} dt 
  & = [g(t/\mu) e^{t}]_1^{\mu\zeta} - 
  \int_1^{\mu\zeta} \frac{1}{\mu} g'(t/\mu) e^t dt \\
  &  \le g(\zeta)e^{\mu \zeta} + \frac{\La}{\mu}\int_1^{\mu\zeta} g(t/\mu)e^tdt.
\end{align*}
One obtains 
\[  (b) \leq  2\Big(1-\frac{\La}{\mu}\Big)^{-1}g(\zeta)e^{\mu \zeta}\frac{e^{-\frac{\mu^2}{2}}}{\mu}. \]
Noting that $g(\zeta)e^{\mu \zeta}\ge g(0)e^{(\mu-\La)\zeta}$ using \eqref{log-lip} again, and that this quantity is bounded away from $0$ for $\mu\ge \mu_0=2\Lambda$, one concludes that for such $\mu$'s the upper-bound for $(b)$ dominates the one for $(a)$, so that
\[ (a)+(b) \le C  g(\zeta)\frac{e^{\mu\zeta-\frac{\mu^2}{2}}}{\mu}.\]
Now one can note, using $\mu_0 \le \mu\le \zeta-1$ and $(g/\phi)(\zeta)\asymp w^{-1}$,
\begin{align*}
 g(\zeta)\frac{e^{\mu\zeta-\frac{\mu^2}{2}}}{\mu} & = g(\zeta)\frac{\phi(\zeta-\mu) }{\phi(\zeta)}\frac{1}{\mu} \\
 & \le C\frac{\ol\Phi(\zeta-\mu)}{w}\frac{|\zeta-\mu|}{\mu}.
\end{align*} 
This gives the result in the case $\mu_0\le \mu\le \zeta-1$, which concludes the proof of the first inequality. The last part of the lemma follows by noting that $T_\mu(w)\le C\zeta^2$. 

 For $|\mu|\le 1/\zeta$, we can invoke Lemma 9, eq. (89) from \cite{js04}, that is
 \[ m_1(\mu,w)\le -\tilde{m}(w)+C\zeta\mu^2\]
 which is at most $C+C\zeta\mu^2$.
\end{proof}

\subsection{Upper bound on $m_2$}

\begin{lemma} \label{lemdeux}
There exist constants $C>0$ and $\omega_0\in(0,1)$ such that for any $w\le \omega_0$, 
for any $\mu\in\R$,
\begin{align*} m_2(\mu,w) & \le
 C \frac{\ol{\Phi}(\zeta-|\mu|)}{w^2}.
 \end{align*}
\end{lemma}
\begin{proof}
Since $m_2(\mu,w)=E[\beta(Z+\mu,w)^2]=\int_{-\infty}^\infty 
\beta(u,w)^2 \phi(u-\mu) du$ by definition, we first bound
\[ \beta(u,w)^2 = \left(\frac{\beta(u)}{1+w\beta(u)}\right)^2
\le C\beta(u)^2\1_{|u|\le \zeta} + w^{-2} \1_{|u|>\zeta}. \]
Indeed, for $\be(u)\ge 0$ this follows from bounding the denominator from below 
by $1$ or $w\be(u)$ respectively, and for $\be(u)<0$ (in which case $|u|<\zeta$, as soon as $w_0<\be^{-1}(0)$) one uses
the fact that $1+w\be(u)\ge 1+w\beta_{min}\ge c_0>0$. Deduce that
\begin{align*} 
m_2(\mu,w) & \le C\int_{-\zeta}^\zeta \beta(z)^2\phi(z-\mu)dz +
\int_{|z|>\zeta} w^{-2}\phi(z-\mu)dz\\
&\le  \qquad\qquad (A)\qquad\qquad+\qquad\qquad (B).
\end{align*}
By definition of (B),
\[ (B) =w^{-2}(\bfi(\zeta-\mu) + \bfi(\zeta+\mu))
\le 2w^{-2}\bfi(\zeta-|\mu|).\]
To bound (A), we note
\[ (A)=C\left(\int_0^\zeta \be(z)^2\phi(z+\mu)dz +\int_0^\zeta \be(z)^2\phi(z-\mu)dz\right) 
\le 2C\int_0^\zeta \be(z)^2\phi(z-|\mu|)dz.  \]
As the last bound is symmetric in $\mu$, it is enough to obtain the desired bound for $\mu\ge 0$, which we thus assume for the remaining of the proof. 
For large enough C, it holds $(\frac{g}{\phi}-1)^2 \le C(\frac{g}{\phi})^2$ (e.g. expanding the square and using that $g/\phi$ is bounded away from $0$) which with Lemma \ref{equ-beta2maj} leads to
\[  \int_0^\zeta \be(z)^2\phi(z-\mu)dz \le C\int_0^\zeta (g/\phi)(z)^2\phi(z-\mu)dz\le 
C\frac8{\zeta} \big(\frac{g}{\phi}\big)^2(\zeta)\phi(\zeta-\mu). \]
Also, $(g/\phi)(\zeta)=\be(\zeta)+1=w^{-1}+1\le 2w^{-1}$. To conclude one writes
\[ \frac{\phi(\zeta-\mu)}{\zeta} = \frac{\phi(\zeta-\mu)}{\zeta-\mu+\mu}. \]
If $\zeta-\mu\ge1$, one can use Lemma \ref{bphi} to obtain that the previous quantity is less than $2\bfi(\zeta-\mu)$ (bound the denominator from below by $\zeta-\mu$). If $\zeta-\mu\le 1$, there exist $C_1,C_2>0$ with
\[ \sup_{\mu:\, \mu\ge \zeta-1} \frac{\phi(\zeta-\mu)}{\zeta} \le C_1\le C_2\bfi(1)\le C_2\bfi(\zeta-\mu).\]
The lemma follows by combining the previous bounds.
\end{proof}

\subsection{Lower bound on $m_1$}

\begin{lemma} \label{m1binf}
There exist constants $M_0, C_1>0$ and $\omega_0\in(0,1)$ such that for any $w\le \omega_0$, and any $\mu \ge M_0$, with $T_\mu(w)$ defined by \eqref{Tmu},
\begin{align*} 
m_1(\mu,w) & \ge
 C_1 \frac{\ol{\Phi}(\zeta-\mu)}{w}T_\mu(w). 
 \end{align*}
\end{lemma}

\begin{proof}
By definition, using $\zeta=\zeta(w)$ as shorthand,
\begin{align*}
 m_1(\mu,w) 
 & = \int_{-\zeta}^{\zeta} \frac{\beta(x)}{1+w\be(x)} \phi(x-\mu)dx\ +\ \int_{|x|>\zeta} \frac{\beta(x)}{1+w\be(x)} \phi(x-\mu)dx\\
 & = \qquad \qquad (I) \qquad \qquad \qquad\ +\ \qquad \qquad \qquad  (II).
 \end{align*}
To bound (II) from below, one notes that $1+w\be(x)\le 2w\be(x)$ for $|x|\ge \zeta$, so 
\[ (II) \ge \frac1{2w} \int_{|x|>\zeta}  \phi(x-\mu)dx= 
 \frac1{2w}( \ol{\Phi}(\zeta-\mu)+ \ol{\Phi}(\zeta+\mu))
 \ge \frac1{2w} \ol{\Phi}(\zeta-\mu). \]
To bound (I) from below, 
let us introduce $d=\max(d_1,d_2)$, where $d_1$ verifies $\beta(d_1)=1$ and $d_2$ is such that for $x\ge d_2$, the map $x\to g(x)$ is decreasing (such $d_2$ exists by \eqref{assumpgdebase}).
We isolate first the possibly negative part of the integral defining $(I)$ and write
\begin{align*} 
 \int_{|x|\le d} \frac{\beta(x)}{1+w\be(x)} \phi(x-\mu)dx 
& \ge - \int_{|x|\le d} \frac{|\beta(x)|}{1+w\be(0)} \phi(x-\mu)dx \\
& \ge - \int_{|x|\le d} \frac{|\beta(x)|}{1+w\be(0)} \frac{dx}{\sqrt{2\pi}}=:-D_1.
\end{align*}
Let $I_1$ be the part of the integral (I) corresponding to $x$ in $\Gamma:=\{x:\, d\le |x|\le \zeta\}$. If $\zeta>d$,
\begin{align*}
I_1 & \ge \int_\Gamma 
\beta(x) \phi(x-\mu)dx 
- w\int_\Gamma \frac{\beta(x)^2}{1+w\beta(x)} \phi(x-\mu)dx \\
& \ge \frac12 \int_\Gamma \beta(x) \phi(x-\mu)dx\\
& \ge \frac14  \int_\Gamma  g(x) \frac{\phi(x-\mu)}{\phi(x)}dx,
\end{align*}
where we have used that $w\be(\cdot)/(1+w\be(\cdot)) \le 1/2$ on $\Gamma$ 
and that $g/\phi-1\ge g/(2\phi)$ on $\Gamma$ by definition of this set. 
 An integration by parts now shows that
 \begin{align*}
 \int_d^\zeta g(x)e^{\mu x}dx & = \frac1\mu \int_{\mu d}^{\mu \zeta}
 g(t/\mu) e^t dt \\
& = \mu^{-1} [g(t/\mu)e^t]_{\mu d}^{\mu \zeta} -  \mu^{-2}\int_{\mu d}^{\mu \zeta} g'(t/\mu) e^t dt\\
& \ge \mu^{-1} \left[ g(\zeta)e^{\mu \zeta} - g(d)e^{\mu d} \right],
\end{align*} 
as $g'(u)<0$ for $u>d\ge d_2$.  
We now claim that $g(\zeta)e^{\mu \zeta}\ge 2 g(d)e^{\mu d}$ for any $\mu\ge 2\La$ and $\zeta\ge d+\log(2)/\La$. Indeed, for such $\mu, \zeta$, 
\[e^{\mu (\zeta-d) } \ge e^{2\La(\zeta-d)} \ge 2 e^{\La(\zeta-d)},
\]
while, using that $-\La\le (\log g)' <0$ on $(d,\infty)$ by \eqref{log-lip} and the definition of $d$, one obtains
\[ 2\frac{g(d)}{g(\zeta)} = 2 e^{-\{\log g(\zeta) - \log g(d)\}} \le 2e^{\La(\zeta-d)}\le  e^{\mu(\zeta-d)}.\]
Putting the two previous bounds together leads to, for such $\mu, \zeta$,
\[I_1 \ge \frac1{8\mu} g(\zeta) e^{\mu\zeta-\mu^2/2}. \]
Let us now distinguish two cases. 
Suppose first that $M_0\le \mu \le \zeta-1$ for $M_0:=2\Lambda$. The map $\mu\to \mu\zeta-\mu^2/2$ is increasing on this interval, so its minimum is attained for $\mu=M_0$. Combining this with $g(\zeta)\ge Ce^{-\Lambda\zeta}$ and using the rough bound $\mu^{-1}\ge \zeta^{-1}$ leads to, uniformly for $\mu\in[M_0,\zeta-1]$,
\[ I_1\geq \frac{e^{-\Lambda\zeta+M_0\zeta-M_0^2/2}}{8\zeta} \geqa \frac{e^{\Lambda\zeta}}{\zeta}. \] 
Since $e^{\Lambda u}/u\to \infty$ as $u\to\infty$ and $\zeta=\zeta(w)\to \infty$ as $w\to 0$, we have $I_1\ge 2D_1$ for any $ \mu\ge [M_0,\zeta-1]$ and any $w\ge \omega_0$ for $\omega_0$ small enough. One deduces that for such $w$ and $\mu$,
\[ I_1-D_1 \ge \frac{g(\zeta)}{16}\frac{e^{\zeta\mu-\mu^2/2}}{\mu}\geqa \frac{1}{\mu}\frac{\phi(\zeta-\mu)}{\phi(\zeta)}g(\zeta). \]
Noting that $\phi(\zeta)/g(\zeta)\sim w$ and combining with the bound on (II) above, one deduces, for $w\le \omega_0$ and $\mu\in[M_0,\zeta-1]$, 
\begin{align*}        
 m_1(\mu,w) & \ge \frac{\bfi(\zeta-\mu)}{2w} + C\frac{\phi(\zeta-\mu)}{\mu w}.
\end{align*}    
Using that $\mu\le \zeta-1$, one deduces that
\[ \frac{\phi(\zeta-\mu)}{\mu w} \ge \frac{\zeta-\mu}{\mu}\frac{\bfi(\zeta-\mu)}{w}.\]
This gives the desired inequality if $\mu\in[M_0,\zeta-1]$. The second case is now $\mu>\zeta-1$. In this case, we simply use $I_1\ge 0$ to get
\[ m_1(\mu,w)\ge -D_1+(II)\ge -D_1+ \frac1{2w} \ol{\Phi}(\zeta-\mu).\]
As $\ol{\Phi}(\zeta-\mu)/(2w)\ge \ol{\Phi}(1)/(2w)$ for small enough $w$, the last display is bounded from below by $\ol{\Phi}(\zeta-\mu)/(4w)$. Noting that the bound
\[ m_1(\mu,w)  \ge C\frac{\bfi(\zeta-\mu)}{w}\left[1+\frac{|\zeta-\mu|}{\mu}\right]\] holds in the two cases, for $C$ a small enough constant, leads to the result, recalling  the definition of $T_w(\mu)$ in \eqref{Tmu}.
\end{proof}

Combining Lemmas \ref{lemdeux} and \ref{m1binf} (and using $T_\mu(w)\ge 1$) one obtains the following bound.
\begin{corollary} \label{m2m1}
There exist constants $M_0, C_2>0$ and $\omega_0\in(0,1)$ such that for any $w\le \omega_0$, and any $\mu \ge M_0$,  
\begin{align*} 
m_2(\mu,w) & \le C_2 \frac{m_1(\mu,w)}{w}. 
 \end{align*}
\end{corollary}

Here is another lower bound for $m_1$ when the signal is large
\begin{lemma} \label{m1binflargesignal}
For any $\veps\in(0,1)$ and $\rho>0$, there exist $\omega_0=\omega_0(\eps,\rho)\in(0,1)$ such that for any $w\le \omega_0$, and any $\mu \ge (1+\rho)\:\zeta(w)$, 
\begin{align*} 
m_1(\mu,w) & \ge (1-\veps)/w.
 \end{align*}
\end{lemma}

\begin{proof}
Let $a=1+(\rho/2)$ and let us write, for $w$ small enough,
\begin{align*}
 w m_1(\mu,w) 
 & = \int_{-a\zeta}^{a\zeta} \frac{w \beta(x)}{1+w\be(x)} \phi(x-\mu)dx +\ \int_{|x|>a\zeta} \frac{w \beta(x)}{1+w\be(x)} \phi(x-\mu)dx\\
 &\geq \int_{x>a\zeta} \frac{w\beta(x)}{1+w\be(x)} \phi(x-\mu)dx -   \int_{-a\zeta}^{a\zeta} \phi(x-\mu)dx\\
 &\geq  \frac{w\beta(a\zeta)}{1+w\be(a\zeta)} \ol{\Phi}(a\zeta-\mu) - (1-\ol{\Phi}(a\zeta-\mu)).
 \end{align*}
Since for $\mu \ge (1+\rho)\:\zeta$, we have that $\ol{\Phi}(a\zeta-\mu) \geq \ol{\Phi}(-(\rho/2)\zeta)$ tends to $1$ when $w$ tends to zero, we only have to prove that $w\beta(a\zeta)=\beta(a\zeta)/\beta(\zeta)$ tends to infinity. The latter comes from 
$$
\beta(a\zeta)/\beta(\zeta) \gtrsim e^{-a\Lambda \zeta}  \:\frac{\phi(\zeta)}{\phi(a\zeta)} = e^{ (a^2-1) \zeta^2 - a\Lambda \zeta},
$$
by using the definition of $\beta$ and  \eqref{tailsg}.
\end{proof}

\subsection{Results for $m_1$ and $\tilde{m}$ ratio}\label{m1tilderatio}

In the next lemmas, we study the behaviour of the functionals, for given $\theta_0\in\R^n$,
\begin{align}
H_{\te_0}(w) &= \frac{\sum_{i\in S_0} m_1(\te_{0,i},w)}{\tilde m(w)},\:\: w\in(0,1), 
\label{defH} \\
H^\circ_{\te_0}(w, K) &= \frac{\sum_{i\in \cC_0(\theta_0,w,K)} m_1(\te_{0,i},w)}{\tilde m(w)}, \:\: w\in(0,1), 
\:\: K\ge 1,\label{defHbis} 
\end{align}
 where we denoted  $S_0=\{1\leq i\leq n\::\: \theta_{0,i}\neq 0\}$ and 
 \[  \cC_0(\theta_0,w,K)=\{1\leq i\leq n\::\: |\theta_{0,i}|\geq \zeta(w)/K\}  \subset S_0. \]
The set $\cC_0(\theta_0,w,K)$ is sometimes denoted by $\cC_0(w,K)$ or $\cC_0$ for short. 
 
\begin{lemma} \label{lemsmallsignal} 
Consider a sparsity $s_n\leq n^{\upsilon}$ for $\upsilon\in(0,1)$. 
Consider $H_{\te_0}$ and $H^\circ_{\te_0}$ as in \eqref{defH} and \eqref{defHbis}, respectively.  There exist constants $C=C(\upsilon,g)>0$ and $D=D(\upsilon,g)\in (0,1)$ such that
\begin{align}\label{removesmallsignal}
\sup_{\te_{0}\in \ell_0[s_n]}\ 
\sup_{  w\in \left[\frac1n,\frac1{\log{n}}\right],\ K\in\left[\frac2{1-\upsilon}, \frac4{1-\upsilon}\right] }\ \left|H_{\te_0}(w)- H^\circ_{\te_0}(w,K) \right| \le Cn^{1-D},
\end{align}
for any $n$ larger than an integer $N=N(\upsilon,g)$.
\end{lemma}
\begin{proof}
For $\te_{0}\in \ell_0[s_n]$ and $w\in [n^{-1},1/\log n]$, denote
\begin{align*}
 \cC_1 &= S_0\setminus\cC_0= \{1\leq i\leq n\::\: 0<|\te_{0,i}| < \zeta(w)/K \}.
    \end{align*} 
By using the upper bounds on $m_1$ obtained in Lemma~\ref{lemun} (and $\mu_0$ defined therein), with $\zeta=\zeta(w)$, and for now taking $K \ge 2$  arbitrary,
\begin{align*}
 \sum_{i\in \cC_1} m_1(\te_{0,i},w) 
&= \Big\{ \sum_{0<|\te_{0,i}|\le \zeta^{-1}} + \sum_{\zeta^{-1}< |\te_{0,i}|\le \mu_0} + \sum_{\mu_0 < |\te_{0,i}|< \zeta/K}\Big\}\, m_1(\te_{0,i},w)  \\
 & \leqa s_n\left\{(1+\zeta^{-1})+ \zeta e^{\mu_0\zeta} +
\zeta w^{-1}\ol\Phi\left(\zeta-\zeta/K\right) \right\},
\end{align*}
where to bound the third sum we use $\ol\Phi\left(\zeta-|\theta_{0,i}|\right)\le \ol\Phi\left(\zeta-\zeta/K\right)$ and   $T_{\mu}(w)\leqa \zeta(w)$. 
Now, by Lemma~\ref{bphi}, 
\begin{align*}
 \bfi\left(\zeta-\frac{\zeta}{K}\right) &\le \frac{K}{K-1}\zeta^{-1} \exp\left(-\frac{\zeta^2}{2} \frac{(K-1)^2}{K^2}\right)\leqa \frac{1}{\zeta} w^{(1-1/K)^2}
 \end{align*}
for $n$ large enough, where we used $\zeta(w)^2\geq -2\log w$ via \eqref{zetamin2} in the last step. Now using that for $w\ge n^{-1}$, we have $\zeta\le 2\sqrt{\log{n}}$ for large $n$ by Lemma \ref{lemzetagen}, so that $e^{\mu_0 \zeta}$ is negligible compared to any positive power of $n$. One deduces that, for $n$ large enough, using $w\ge n^{-1}$ and $s_n\leqa n^{\upsilon}$ by assumption, and any $K\ge 2$,
\begin{align*}
 \sum_{i\in \cC_1} m_1(\te_{0,i},w) 
 &\le C s_n\left\{ 1 + e^{C\zeta} +  w^{-2/K+1/K^2} \right\} \\
 &\le Cn^{\upsilon}e^{C\zeta} + C n   \: n^{\upsilon-1+2/K-1/K^2}.
 \end{align*}
Now if $\upsilon-1+2/K\le 0$, which holds for $K$ as in the statement, one gets
\begin{align*}
\sup_{\te_{0}\in \ell_0[s_n]}\sup_{  w\in [n^{-1},1/\log{n}]}  \frac{ \sum_{i\in \cC_1} m_1(\te_{0,i},w) }{\tilde{m}(w)}
 &\leq  \frac{C}{\tilde{m}(n^{-1})}  \{ n^{\upsilon}e^{2C\sqrt{\log{n}}} + n^{1-1/K^2}\}.
 \end{align*}
 For $K$ as in the statement, we further have $1-K^{-2}\le 1-(1-\upsilon)^2/16$. 
Since $\tilde{m}(n^{-1})$ decreases to $0$ slower than any power of $n$ (see Lemma~\ref{lemmtilde}, combined with \eqref{tailsg} and the bound \eqref{zetamaj2} on $\zeta$), the last display can be bounded by $Cn^{1-D}$,
 for $D$ small enough, which shows \eqref{removesmallsignal}.
\end{proof}

\begin{lemma} \label{lemhcirc}
Consider $H^\circ_{\te_0}$ as in \eqref{defHbis} for some choice of $K>1$.
 Then there exists a constant $C=C(K,g)>0$ such that, for all $z\geq 1$, there exists $\omega_0=\omega_0(z,K,g)\in (0,1)$ such that for all $w\in (0,\omega_0)$ and for all $\theta_0\in \R^n$, we have
\begin{equation} \label{proph2}
H^\circ_{\te_0}(w/z,K) \ge C z^{1/(2K)} H^\circ_{\te_0}(w,K/1.1). 
\end{equation}
\end{lemma}
 
\begin{proof}
According to Lemma~\ref{lemun} and Lemma~\ref{m1binf}, there exists constants $C_1,C_2>0$ and $\omega_0\in (0,1)$ such that for $w\in (0,\omega_0)$ and any $\theta_0$,
 $$
C_1 \sum_{i\in \cC_0(w,K)} G_{\te_{0,i}}(w)   \frac{ T_{\te_{0,i}}(w)}{\tilde m(w)} \leq H^\circ_{\te_0}(w,K) \leq  C_2 \sum_{i\in \cC_0(w,K)} G_{\te_{0,i}}(w)   \frac{ T_{\te_{0,i}}(w)}{\tilde m(w)},
$$
where 
$T_\mu$, $G_\mu$ are defined by \eqref{Tmu}, \eqref{gmu} respectively.
Now, by Lemmas~\ref{lemTmuw}~and~\ref{lemregv}, for all $z\geq 1$, there exists $\omega_0(z,K)\in (0,1)$ such that for $w\le \omega_0(z,K)$ and any $\mu\ge \zeta(w)/K$,
\begin{align*}
&G_\mu(w/z) \ge z^{1/(2K)} G_\mu(w) \\
&d_1 \:T_\mu(w) \leq T_\mu(w/z) \leq d_2 \:T_\mu(w),
\end{align*}
for some constants $d_1=d_1(K)$, $d_2=d_2(K)$.
Combining Lemma~\ref{lemmtilde} on $\tilde m$ with Lemma~\ref{lemTmuw} on $\ol{G}$, one can find $D_1, D_2>0$ with, for $w\le \omega(z)$,
\[ D_1\: \tilde{m}(w) \leq \tilde{m}(w/z) \leq D_2\: \tilde{m}(w). \]
Hence, by combining these results one gets, for $w\le \omega_0(z,K)$  (and then $w/z\le \omega_0(z,K)$ also holds),
\begin{align*}
H^\circ_{\te_0}(w/z,K) &\geq C_1 \sum_{i\in \cC_0(w/z,K)} G_{\te_{0,i}}(w/z)   \frac{ T_{\te_{0,i}}(w/z)}{\tilde m(w/z)}\\
&\geq (C_1 d_1/D_2) z^{1/(2K)} \sum_{i\in \cC_0(w/z,K)}G_{\te_{0,i}}(w) \frac{ T_{\te_{0,i}}(w)}{\tilde m(w)}.
\end{align*}
Now we claim that $\cC_0(w,K /1.1)\subset \cC_0(w/z,K)$  for $w$ small enough depending on $z$. Indeed, $\zeta(w/z) /\zeta(w) \leq 1+ (\zeta(w/z)- \zeta(w)) /\zeta(w) \leq 
1.1$ for $w$ small enough depending on $z$, as in the proof of Lemma \ref{lemTmuw}. So,
\begin{align*}
\cC_0(w/z,K)&=\{1\leq i\leq n\::\: |\theta_{0,i}|\geq \zeta(w/z)/K\}\\
&\supset\{1\leq i\leq n\::\: |\theta_{0,i}|\geq 1.1 \zeta(w)/K\}=\cC_0(w,K/1.1).
\end{align*}
One deduces that $H^\circ_{\te_0}(w/z,K)\ge C z^{1/(2K)} H^\circ_{\te_0}(w,K/1.1)$ for $w\le \omega_0(z,K)$ as announced.
\end{proof}

\section{Lower bound for the $\FDR$+$\FNR$ risk}\label{sec:lb}

For any $a_n\geq 0$, define the class of signals
\begin{align*}
\mathcal{L}^-_0[s_n;a_n]&= \{\theta_0\in \ell_0[s_n]\::\: |\theta_{0,i}|\leq a_n, |S_{\theta_0}|=s_n\}.
\end{align*}

\begin{theorem}
Let $s_n\geq 1$, $\epsilon\in(0,1)$ and  
\begin{equation}\label{equan}
 a_{n,\epsilon}= \ol{\Phi}^{-1}\left((1/\epsilon+1)\frac{s_n}{n-s_n}\right) -\ol{\Phi}^{-1}\left(\epsilon/4\right).
\end{equation}
Then we have
$$
\sup_{\varphi\in \mathcal{C}} \sup_{\theta_0\in \mathcal{L}^{-}_0[s_n;a_{n,\epsilon}]} \left(\P_{\theta_0}(\FDP(\theta_0,\varphi)+\FNP(\theta_0,\varphi)\leq 1-\epsilon)\right)\leq 3 e^{- s_n \epsilon/6}.
$$
\end{theorem}

By integration with respect to $\epsilon\geq 1/t_n$ for some sequence $t_n$, we get 
\begin{corollary}\label{cor:lb}
Let $s_n\geq 1$, $t_n\geq 1$, and
\begin{equation}\label{equbn}
b_n=\ol{\Phi}^{-1}\left((t_n+1)\frac{s_n}{n-s_n}\right) -\ol{\Phi}^{-1}\left(1/(4t_n)\right).
\end{equation}
Then we have
\begin{align*}
\inf_{\varphi\in \mathcal{C}} \inf_{\theta_0\in \mathcal{L}^{-}_0[s_n;b_n]} \left(\FDR(\theta_0,\varphi)+\FNR(\theta_0,\varphi) \right)\geq 1-(1/t_n+18/s_n).
\end{align*}

\end{corollary}

Taking $s_n \rightarrow \infty$ and $s_n\le n^{\upsilon}$ for some $\upsilon\in(0,1)$, and $t_n=e^{\sqrt{\log (n/s_n)}}$, we get $b_n\sim \sqrt{2 \log (n/s_n)}$ and thus for $a<1$, 
\begin{align*}
&\liminf_n \inf_{\varphi\in \mathcal{C}} \sup_{\theta_0\in \mathcal{L}_0[s_n;a]} \left(\FDR(\theta_0,\varphi)+\FNR(\theta_0,\varphi) \right)\\
&\geq \liminf_n \inf_{\varphi\in \mathcal{C}} \inf_{\theta_0\in \mathcal{L}^{-}_0[s_n;b_n]} \left(\FDR(\theta_0,\varphi)+\FNR(\theta_0,\varphi) \right)\geq 1.
\end{align*}
This proves Proposition~\ref{prop:lb}.

\begin{proof}
Let $\delta>0$ and $a_n$ arbitrary with $|\theta_{0,i}|\leq a_n$ (to be chosen below).
On the one hand, we have 
\begin{align*}
\FDP(\theta_0,\varphi)&\geq \frac{s_n^{-1}\sum_{i=1}^n \ind{\theta_{0,i} =0,X_i\geq \tau_1(X) \mbox{ or } -X_i\geq \tau_2(X)} }{ 1+s_n^{-1}\sum_{i=1}^n \ind{\theta_{0,i} =0, X_i\geq \tau_1(X) \mbox{ or } -X_i\geq \tau_2(X)} }\\
&\geq 1 - \left(s_n^{-1}\sum_{i=1}^n \ind{\theta_{0,i} =0,X_i\geq \tau_1(X) \mbox{ or } -X_i\geq \tau_2(X)}\right)^{-1} .
\end{align*}
Furthermore, 
\begin{align*}
&\:s_n^{-1}\sum_{i=1}^n \ind{\theta_{0,i} =0,X_i\geq \tau_1(X) \mbox{ or } -X_i\geq \tau_2(X)}
\\=\:&s_n^{-1}\sum_{i=1}^n \ind{\theta_{0,i} =0,\eps_i\geq \tau_1(X) \mbox{ or }  -\eps_i\geq \tau_2(X)}\\
\geq& \left(s_n^{-1}\sum_{i=1}^n \ind{\theta_{0,i} =0,\eps_i\geq \tau_1(X)\wedge \tau_2(X)}\right)\\
&\wedge\left(s_n^{-1}\sum_{i=1}^n \ind{\theta_{0,i} =0,-\eps_i\geq \tau_1(X)\wedge \tau_2(X)}\right).
\end{align*}
The latter is true, because it holds whether $\tau_1(X)\wedge \tau_2(X)$ is $\tau_1(X)$ or $\tau_2(X)$.
Thus on the event $\{(\tau_1(X)\wedge \tau_2(X))-a_n \leq  \delta\}$, we have $\tau_1(X)\wedge \tau_2(X) \leq a_n+ \delta$, and we get 
\begin{align}
\FDP(\theta_0,\varphi)
\geq 1- &\:\left(s_n^{-1}\sum_{i=1}^n \ind{\theta_{0,i} =0,\eps_i\geq a_n+ \delta}\right)^{-1}\nonumber\\
&\vee\left(s_n^{-1}\sum_{i=1}^n \ind{\theta_{0,i} =0,-\eps_i\geq a_n+ \delta}\right)^{-1} \label{Erylb1}.
\end{align}

On the other hand
\begin{align*}
\FNP(\theta_0,\varphi) &= s_n^{-1}\sum_{i=1}^n \ind{\theta_{0,i} \neq  0, -\tau_2(X)<X_i< \tau_1(X)}\\
&=  s_n^{-1}\sum_{i=1}^n \ind{\theta_{0,i} \neq  0, -\tau_2(X)-\theta_{0,i}<\eps_{i} < \tau_1(X)-\theta_{0,i}}\\
&\geq  s_n^{-1}\sum_{i=1}^n \ind{\theta_{0,i} \neq  0, -(\tau_2(X)-|\theta_{0,i}|)<\eps_{i} < \tau_1(X)-|\theta_{0,i}|}.
\end{align*}
Hence, noting that $\tau_1(X)\wedge \tau_2(X)-a_n $, is smaller than $\tau_1(X)-|\theta_{0,i}|$ and $\tau_2(X)-|\theta_{0,i}|$, we obtain
\begin{align*}
\FNP(\theta_0,\varphi) &\geq  s_n^{-1}\sum_{i=1}^n \ind{\theta_{0,i} \neq  0, |\eps_{i}| < \tau_1(X)\wedge \tau_2(X)-a_n}.\end{align*}
Hence, on the event $\{(\tau_1(X)\wedge \tau_2(X))-a_n\geq \delta\}$, we get 
\begin{align}\label{Erylb2}
\FNP(\theta_0,\varphi) &\geq  s_n^{-1}\sum_{i=1}^n \ind{\theta_{0,i} \neq  0, |\eps_{i}| < \delta}.
\end{align}
Combining \eqref{Erylb1} and \eqref{Erylb2}, we obtain for all $\delta>0$,
\begin{align*}
&\FDP(\theta_0,\varphi)+\FNP(\theta_0,\varphi)\\
\geq &\left(s_n^{-1}\sum_{i=1}^n \ind{\theta_{0,i} \neq  0, |\epsilon_{i}| < \delta}\right)\\
&\wedge \left(1- \left(s_n^{-1}\sum_{i=1}^n \ind{\theta_{0,i} =0,\eps_i\geq a_n+ \delta}\right)^{-1}\right)\\
&\wedge \left(1- \left(s_n^{-1}\sum_{i=1}^n \ind{\theta_{0,i} =0,-\eps_i\geq a_n+ \delta}\right)^{-1}\right).
\end{align*}
This induces that for all $\eps\in(0,1)$,
\begin{align*}
&\P_{\theta_0}(\FDP(\theta_0,\varphi)+\FNP(\theta_0,\varphi)\leq 1-\epsilon)\\
\leq&\: \P_{\theta_0}\left(s_n^{-1}\sum_{i : \theta_{0,i} \neq  0} \ind{ |\eps_{i}| < \delta} \leq 1-\epsilon\right)\\
&+2\P_{\theta_0}\left(s_n^{-1}\sum_{i :\theta_{0,i} =0} \ind{\eps_i\geq a_n+ \delta}\leq 1/\epsilon\right).
\end{align*}
Now choose $\delta$ such that $\epsilon=4\ol{\Phi}(\delta)$, so that 
\begin{align*}
\P_{\theta_0}\left(s_n^{-1}\sum_{i : \theta_{0,i} \neq  0} \ind{ |\eps_{i}| < \delta} \leq 1-\epsilon\right)&=
\P_{\theta_0}\left(\sum_{i : \theta_{0,i} \neq  0} (\ind{ |\eps_{i}| \geq \delta}-2 \ol{\Phi}(\delta))\geq s_n\epsilon/2
\right)\\
&\leq e^{-s_n\epsilon/6}
\end{align*}
by applying Bernstein inequality (see Lemma~\ref{th:bernstein}) with $A=s_n\epsilon/2$, $V=2s_n \ol{\Phi}(\delta)=A$ and $\mathcal{M}=1$. Similarly, by choosing $a_n$ as in \eqref{equan}
so that $(n-s_n)\ol{\Phi}(a_n+\delta)= s_n(1/\epsilon+1)$, we have
\begin{align*}
&\P_{\theta_0}\left(s_n^{-1}\sum_{i :\theta_{0,i} =0} \ind{\eps_i\geq a_n+ \delta}\leq 1/\epsilon\right)\\
=&\P_{\theta_0}\left(\sum_{i :\theta_{0,i} =0} (\ind{\eps_i\geq a_n+ \delta}-\ol{\Phi}(a_n+\delta))\leq 
-s_n 
\right)\leq e^{- s_n \epsilon/6}
\end{align*}
by applying Bernstein inequality (see Lemma~\ref{th:bernstein}) with $A=s_n$, $V=(n-s_n) \ol{\Phi}(a_n+\delta)\leq 2s_n/\epsilon$ and $\mathcal{M}=1$. The proof is finished.
\end{proof}

\section{Details on MCI procedures}\label{sec:MCIanalysis}

Let us consider the procedure $\vphi^m$ at cut-off level $t\in(0,1/2)$ defined by \eqref{mval} in   Section \ref{sec:other}. We henceforth refer to it as procedure \MCI. 
 We  show below that $\vphi^m$ can be rewritten in terms of  $\phi$ as well as $g_-, g_+$ defined as, for any $x\in\RR$,
\begin{align*} 
g_-(x) & := \int_{-\infty}^0 \phi(x-u)\gamma(u)du,\\
g_+(x) & := \int_0^{\infty} \phi(x-u)\gamma(u)du=(g-g_-)(x). 
\end{align*}

\begin{lemma} \label{lemgpm}
For any real $x$, it holds $g_+(-x)=g_-(x)$. Also, $g_+(x)>g_-(x)$ if and only if $x>0$.
\end{lemma}
\begin{proof}
The first assertion follows from the symmetry of $\phi$ and $\gamma$. To check the second assertion,  by symmetry of $\ga$,
\[ g_+(x) = \int_0^\infty \phi(x-u)\gamma(u)du=\int_{-\infty}^0 \phi(x+v)\ga(v) dv. \]
For $x>0$ and $v<0$, we have $|x+v|< x-v$ so that $\phi(x+v)> \phi(x-v)$ which gives $g_+(x)>g_-(x)$ and the 'if' part. For the 'only if' part, by symmetry, as before $x<0$ implies $g_-(x)>g_+(x)$ and for $x=0$ we have $g_+(0)=g_-(0)$. So $g_+>g_-$ can only occur if $x>0$.
\end{proof}

\subsection{The $m$--value} 
By analogy to $\ell$--values, for a given weight $w\in(0,1)$, define an {\em m--value} as, for $i=1,\ldots,n$,
\begin{align}
m_i(X) & = m(X_i;w);\\
m(x;w) & = \Pi(\te_1\ge 0\given X_1=x) \wedge \Pi(\te_1\le 0\given X_1=x). \label{mvaldef}
\end{align}
A BMT of the form $\vphi=\ind{m_i(X)\le t}$ is called a {\em m-value} procedure (where `$m$' stands for (posterior) `mass', as opposed to `$\ell$' for `local' standing for the local `density' at $0$). This definition is motivated by the following lemma. 
\begin{lemma} \label{lemmvalue}
The procedure \MCI$\,$ defined by $\vphi^m$ in \eqref{mval} at level $t\in(0,1/2)$ can be written as, denoting $\hat m_i(X)  := m(X_i;\hat w)$, for $i=1,\ldots,n$,
\[\vphi^m_i = \ind{\hat m_i(X)< t}.\]
\end{lemma}
\begin{proof}
Let us denote by $z^t(x)$ the quantile at level $t\in(0,1/2)$ of the marginal posterior distribution of $\te_1$ given $X_1=x$. 
By definition of the quantile, $0< z^t(x)$ if and only if $\Pi[\te_1\le 0\given X_1=x]< t$. 
Further,  $z^{1-t}(x)<0$ if and only if $\Pi[\te_1\ge 0\given X_1=x]<t$: this uses the definition of the quantile and the fact that $(-\infty,0)\ni u\to \Pi[\te_1<u\given X_1=x]$ is  strictly increasing and continuous, as follows from the explicit expression of the posterior distribution. By definition of $\vphi^m$, the procedure rejects  $H_{0,i}$ if and only if either $z_i^t(X)>0$ or $z_i^{1-t}(X)<0$, which concludes the proof.
\end{proof}
\begin{lemma} \label{mvalexpr}
For any $w\in(0,1)$, the $m$-value $m(x;w)$ at point $x\in\RR$ can be written as
 \begin{equation} \label{mexpr}
  m(x;w) = \frac{(1-w)\phi(x)+wg_-(|x|)}{(1-w)\phi(x)+wg(x)}.
 \end{equation}
 Additionally, for any real $x$, the map $w\to m(x;w)$ is decreasing.
\end{lemma} 
\begin{proof}
By definition, recalling \eqref{posteriordistribution}--\eqref{lvalues}--\eqref{lformula},
\begin{align*}
\lefteqn{ \Pi[\te_1>0\given X_1=x]  = \Pi[\te_1>0\given X_1=x, \te_1\neq 0] \cdot \Pi[\te_1\neq 0\given X_1=x] }\\
\quad & = \int_{0}^{\infty} \ga_x(u)du \cdot (1-\ell(x;w,g))= \frac{g_+(x)}{g(x)}\cdot \frac{wg(x)}{(1-w)\phi(x)+wg(x)} \\
\quad &  = \frac{wg_+(x)}{(1-w)\phi(x)+wg(x)}.
\end{align*}
Using now the definition of $m(x;w)$, one obtains
\begin{align*}
 m(x;w) &= (1-\Pi[\te_1>0\given X_1=x]) \wedge (\ell(x;w,g)+\Pi[\te_1>0\given X_1=x]) \\
& = \frac{(1-w)\phi(x)+wg_-(x)}{(1-w)\phi(x)+wg(x)} \wedge \frac{(1-w)\phi(x)+wg_+(x)}{(1-w)\phi(x)+wg(x)}. 
\end{align*}
The announced expression follows by noting that $g_-(x)\wedge g_+(x)=g_{-}(|x|)$ which itself is a consequence of Lemma \ref{lemgpm}. The monotonicity in $w$ is obtained by computing, for any real $x$,
\[ \frac{\partial m}{\partial w}(x;w) = - \frac{g_+(|x|)\phi(x)}{\left[(1-w)\phi(x)+wg(x)\right]^2}<0. \qedhere \]
\end{proof}

\subsection{Link to $\ell$--values} 
\begin{lemma} \label{linkmlval}
The m--value \eqref{mvaldef} satisfies, for any $w\in[0,1)$, $x\in\RR$,
\begin{equation} \label{linkml}
\ell(x;w)\le m(x;w)\le \left(1+\frac{w}{1-w}\frac{\gamma(0)}{2}\right)\ell(x;w), 
\end{equation}
where for short we denote $\ell(x;w):=\ell(x;w,g)$ (itself defined in \eqref{lformula}).
\end{lemma}
\begin{proof}
The first inequality immediately follows from the expression of $m(x;w)$ in \eqref{mexpr} and the $\ell$-value expression \eqref{lformula}. The second inequality follows using Lemma \ref{gmphi}. 
\end{proof}
\begin{lemma} \label{gmphi}
For any $t\ge 0$, we have 
\[ g_-(t) \le \frac12 \ga(0)\phi(t). \]
\end{lemma}
\noindent {\em Remark.} The following more precise bounds also hold, for any $t\ge 0$,
\[ \gamma(-1)\left(\overline\Phi(t)-\overline\Phi(t+1)\right)\le g_-(t) \le \ga(0) \overline\Phi(t). \]
showing that $g_-(t)\asymp \phi(t)/t$ for large $t$.
\begin{proof}
As $\ga$ is unimodal, continuous and symmetric, its maximum is attained at $0$, so $\|\ga\|_\infty=\ga(0)$, and
\begin{align*}
(g_-/\phi)(t) & = \int_{-\infty}^0 e^{u t-u^2/2}\ga(u)du \\
& \le \ga(0) e^{t^2/2} \int_{-\infty}^0 e^{(t-u)^2/2} du\le \ga(0)\phi(t)^{-1}\overline\Phi(t).
\end{align*}
The lemma follows using the standard bound $\overline\Phi(t)/\phi(t)\le 1/2$, as well as the upper bound in the remark above. The lower bound in the remark is obtained by restricting the integral defining $g_-$ to $[-1,0]$.
\end{proof}

\subsection{Proof of Theorem \ref{thm-mci}}
The idea of the proof for the procedure \MCI$\,$  is as follows. To control the FDR for $m$--values, one combines the inequalities \eqref{linkml} with the bounds for $\ell$--values already derived in the proof of Theorem \ref{th1}. Using these inequalities will only  modify by a constant multiplicative factor (close to $1$, e.g. $1+\epsilon$, $\epsilon>0$) the level `$t$' of the original argument for $\ell$--values. This only modifies the constants $N_0$ and $C$ in the statement of Theorem \ref{th1}, leaving everything else unchanged and leading to the result. We now give the detailed argument for the procedure \MCI$\,$ for completeness.

Proceeding as in the proof of Theorem \ref{th1}, one distinguishes two cases depending on whether \eqref{equw1} has a solution or not.
 If \eqref{equw1} has no solution, then one bounds the FDR of the $m$-values procedure at level $t$ as follows, using the first inequality in \eqref{linkml},
 \begin{align*}
 \FDR( \theta_0,\vphi^{\mbox{\tiny $m$}}(t;\hat{w}))&
 \leq \P_{\theta_0}(\exists i \::\: \theta_{0,i}=0 ,\: \vphi_i^{\mbox{\tiny $m$}}(t;\hat w) =1) \\
&\leq \P_{\theta_0}(\exists i\::\: \theta_{0,i}=0 ,\:\vphi_i^{\mbox{\tiny $\ell$-val}}(t;w_0) =1) + \P_{\theta_0}(\hat w>w_0)
 \end{align*} 
and this quantity is that of the $\ell$--value case, which is thus bounded as in the proof of Theorem \ref{th1}.

If  \eqref{equw1} has a solution, similar to the $\ell$--value case, let us denote by $V_m^{[t]}(w)$ the number of false discoveries of the $m$--values procedure  $\vphi^m(t;w)$ at level $t$ and $S_m^{[t]}(w)$ the number of its true discoveries. Here we denote $V_\ell^{[t]}(w)$ and $ S_m^{[t]}(w)$ the corresponding quantities for $\ell$--values (as in the proof of Theorem \ref{th1}, except here we also keep the level $t$ explicit in the notation, which is important below). We start by writing the FDR as
\begin{align*}
\lefteqn{\FDR(\theta_0,\vphi^{\mbox{\tiny $m$}}(t;\hat{w})) 
  = \E_{\te_0}\left[ \frac{V_m^{[t]}(\hat w)}{(V_m^{[t]}(\hat w)+S_m^{[t]}(\hat w))\vee 1}\right]}\\
& \le  \E_{\te_0}\left[ \frac{V_m^{[t]}(\hat w)}{(V_m^{[t]}(\hat w)
+S_m^{[t]}(\hat w))\vee 1}
\ind{w_2\le \hat w\le w_1}\right] + \P_{\te_0}\left[\hat w\notin [w_2,w_1]\right].
\end{align*}
Thanks to the first inequality in \eqref{linkml}, 
\[ \ind{\hat m_i(X)<t} \le  \ind{\ell(\hat w,X_i)\le  t}, \]
which yields $V_m^{[t]}(\hat w)\le V_\ell^{[t]}(\hat w)$. Now  using the second inequality in \eqref{linkml}, and working on the event that $w_2\le \hat w\le w_1$,
 \begin{align*}
 \ind{\hat m_i(X)<t} &\ge \ind{\left(1+\frac{\hat w}{1-\hat w}\frac{\ga(0)}{2}\right)\ell(\hat w,X_i)< t} \\
&\ge  \ind{\left(1+\frac{\hat w}{1-\hat w}\ga(0)\right)\ell(\hat w,X_i)\le t} \\
&\ge \ind{\left(1+\frac{w_1}{1-w_1}\ga(0)\right)\ell(\hat w,X_i)\le t}\ge \ind{\frac{5}{4}\ell(\hat w,X_i)\le t},
 \end{align*}
where for the second  inequality we have used that $\ell$--values are strictly positive almost surely, and for the fourth inequality that $w_1$ goes to $0$ with $n$, using Lemma \ref{lemw1}. This leads to, on the event that $w_2\le \hat w\le w_1$,
\[ \ind{\ell(\hat w,X_i)\le  t'} \le \ind{\hat m_i(X)<t},\quad t':=\frac{4}{5}t, \]
which implies $S_m^{[t]}(\hat w)\ge S_\ell^{[t']}(\hat w)$. So, denoting $\cD=\{w_2\le \hat w\le w_1\}$ for short,
\[  \E_{\te_0}\left[ \frac{V_m^{[t]}(\hat w)}{(V_m^{[t]}(\hat w)
+S_m^{[t]}(\hat w))\vee 1}
\ind{\cD}\right]
\le \E_{\te_0}\left[ \frac{V_\ell^{[t]}(\hat w)}{(V_\ell^{[t]}(\hat w)
+S_\ell^{[t']}(\hat w))\vee 1}
\ind{\cD}\right].
\]
From this point on, one can use the bounds derived for $\ell$--values, replacing $t$ by $t'$ in the bound for $S_\ell^{[t']}$. This only induces changes in the constants appearing in the bound \eqref{lvalcase2} for $\ell$--values, everything else being unchanged. 

By combining the bounds in both cases, bounds which coincide with the $\ell$--values bounds up to the choice of the constants, this concludes the proof of Theorem \ref{thm-mci}. 
 
\section{Details on $\SC$ procedure}\label{sec:SCanalysis}

We explore here in more details the behavior of the Sun and Cai procedure $\SC$, as defined in Section~\ref{rem:SC}, with a heuristic, a lemma 
and numerical support. To fix the idea, we focus on the quasi-Cauchy prior (similar results could be obtained with Laplace prior).\\

\subsection{Numerical study}\label{tab:final}

Let us first consider the same simulation setting as in Section~\ref{sec:simu} for Figure~\ref{fig1}. The FDR of $\SC$ is computed on  Figure~\ref{fig:SC} for different values of thresholds $t\in\{0.05,0.1,0.2\}$. Clearly, compared to {\tt EBayesq}, we observe a more severe FDR inflation, especially when $s_n/n$ is not small and when the signal is large. This suggests the following question:

\begin{figure}[h!]
\begin{tabular}{ccc}
\vspace{-0.5cm}
&quasi-Cauchy & Laplace\\
\vspace{-1cm}
\rotatebox{+90}{\hspace{3cm}$s_n/n=0.1$} &\includegraphics[scale=0.35]{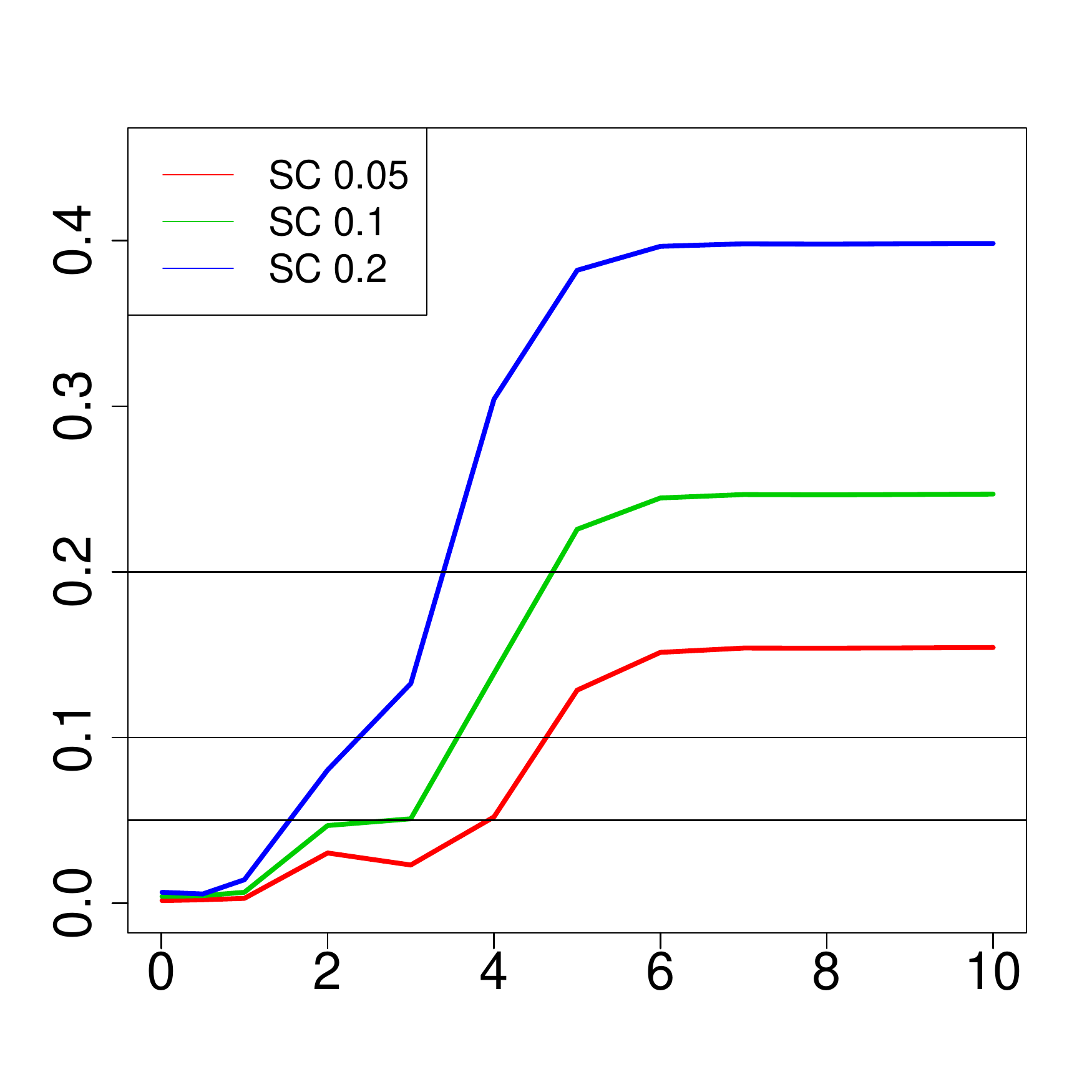}
&\includegraphics[scale=0.35]{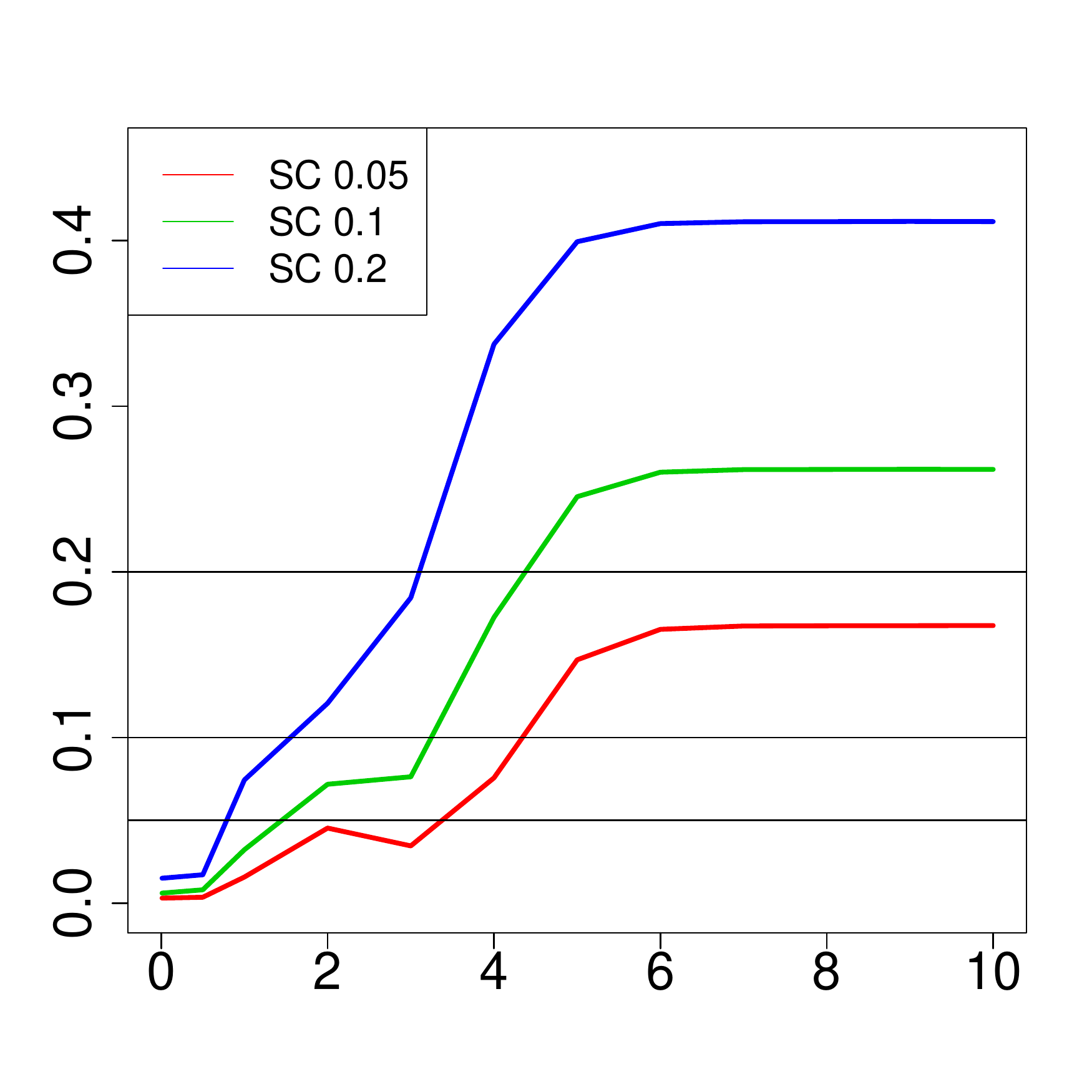}
\\
\vspace{-1cm}
\rotatebox{+90}{\hspace{3cm}$s_n/n=0.01$} &\includegraphics[scale=0.35]{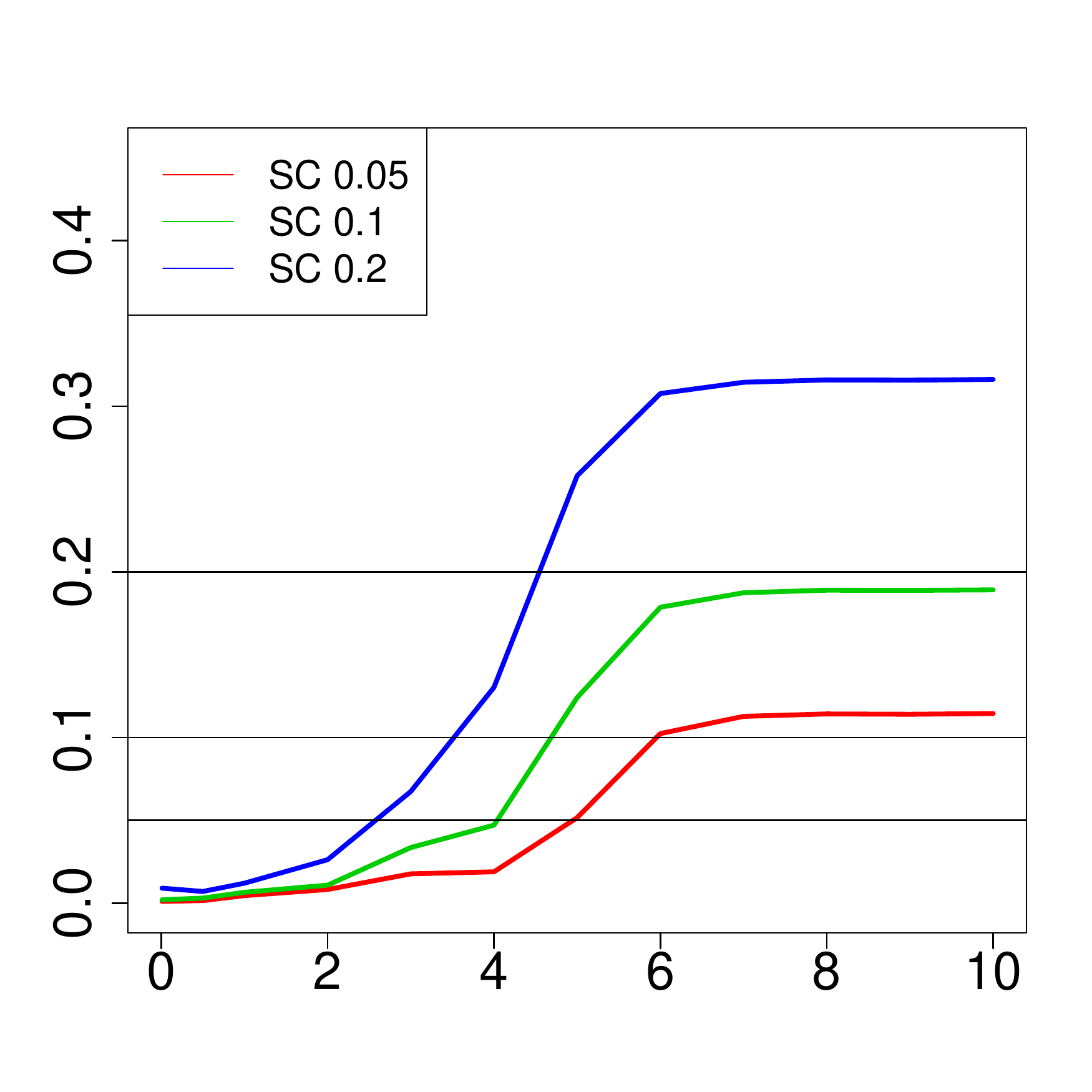}
&\includegraphics[scale=0.35]{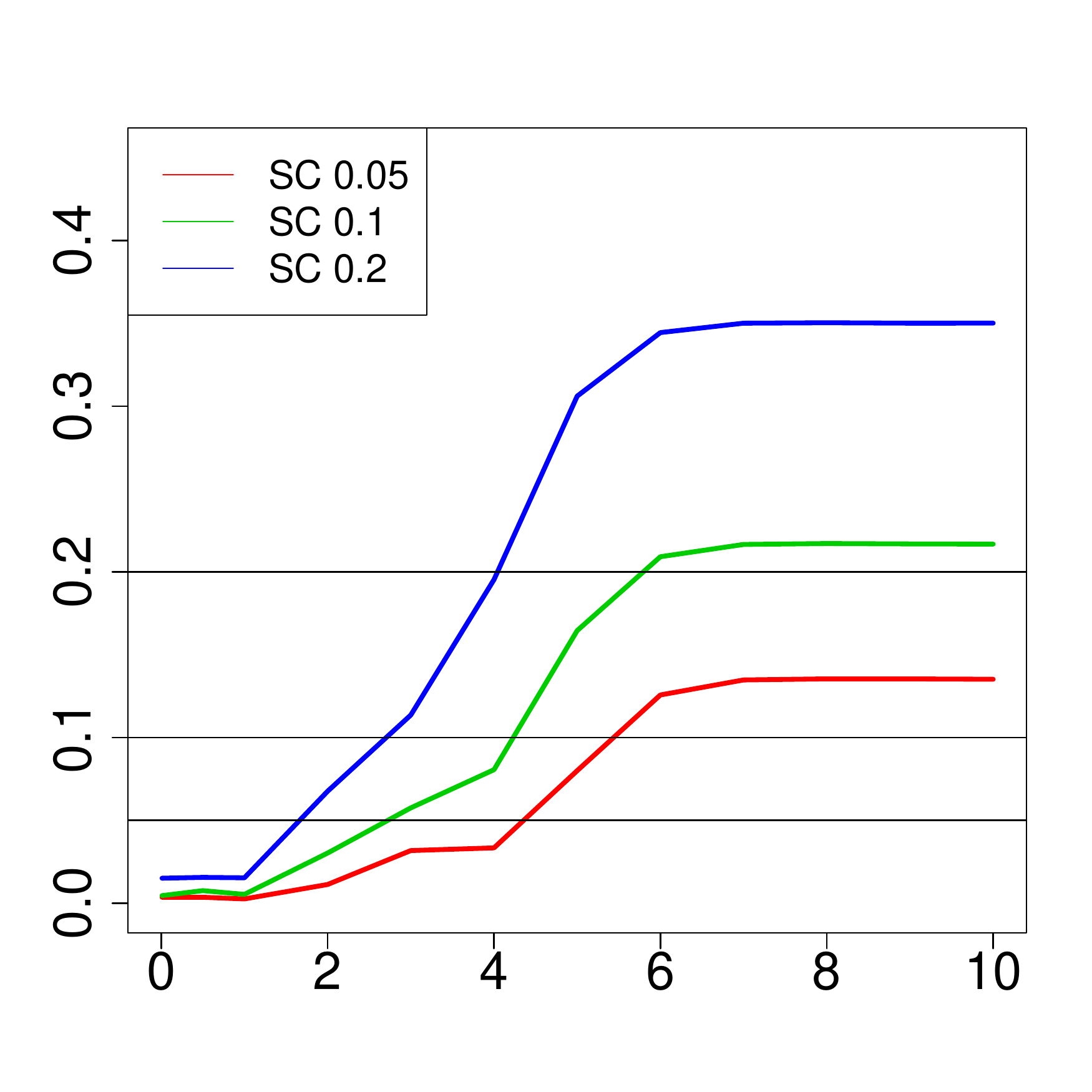}
\\
\vspace{-0.5cm}
\rotatebox{+90}{\hspace{3cm}$s_n/n=0.001$} &\includegraphics[scale=0.35]{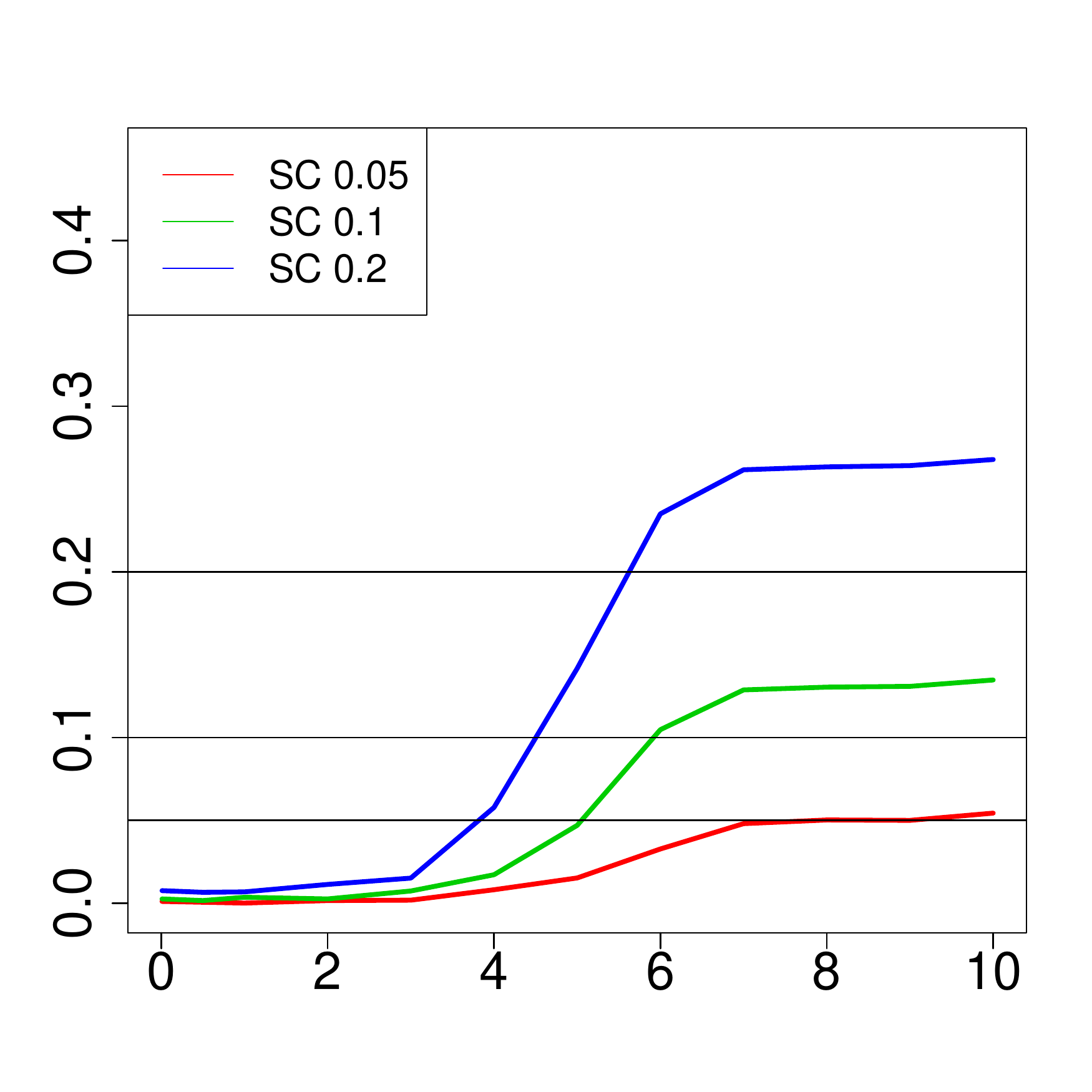}
&\includegraphics[scale=0.35]{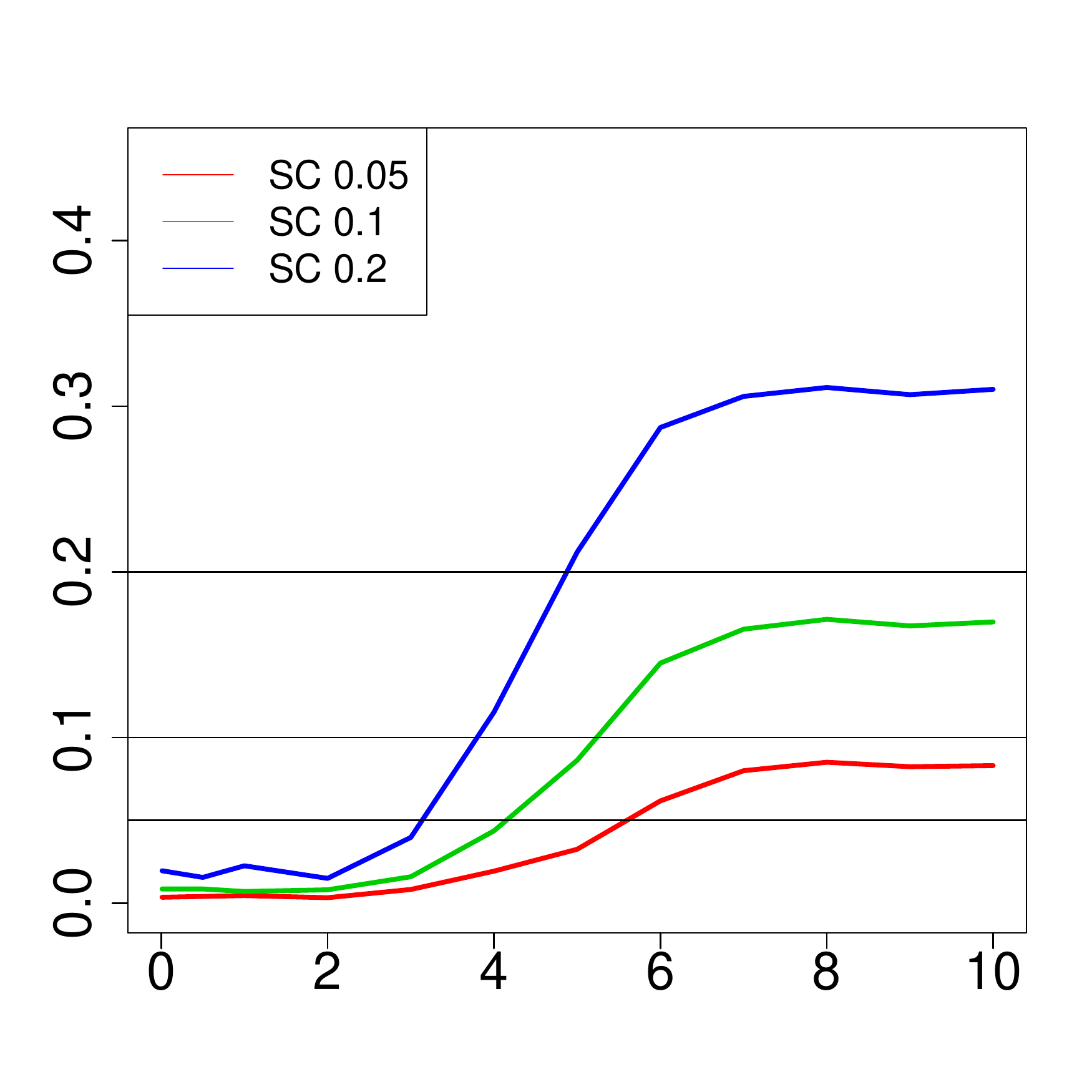}
\end{tabular}
\caption{\label{fig:SC}
FDR  for $\SC$ procedure with threshold $t\in\{0.05,0.1,0.2\}$.  $n=10,\,000$; $2000$ replications; alternative all equal to $\mu$ (on the $X$-axis).}
\end{figure}

\begin{center}
For a very large signal, does the FDR of $\SC$ procedure converges to $t$ when $n$ tends to infinity (and $s_n/n$ tends to $0$)?
\end{center}

To elucidate this question, we first perform a numerical experiment in the special case $t=0.2$, $\mu=15$ and $n=10^7$ for different values of $s_n$. Due to the large amplitude of $n$, it is too computationally demanding to use the empirical Bayes $\hat{w}$ into the $\ell$-values expression. Rather, we can safely replace it by $w=w^\star$ solving $s_n = (n-s_n) w \tilde{m}(w)$ because the signal is very strong. Also, it is enough to make $10$ replications to approximate the FDR, because the concentration of the FDP to the FDR is very fast for $n=10^7$. 
The result is given in the following table.
$$
\begin{array}{|c|c|c|c|c|c|}\hline
s_n & 10^4 & 10^3 & 10^2 & 10 & 5\\\hline
\FDR & 0.30 & 0.29 & 0.28& 0.24& 0.21\\\hline
\end{array}
$$
This experiment suggests that the FDR of $\SC$ does converge to the targeted level $t=0.2$, but very slowly with respect to $s_n/n$. \\

In the next sections, we provide an analysis to support the fact that the FDR of $\SC$ converges to $t$ at a logarithmic rate in $s_n/n$ when the signal is very large. This will thus corroborate the above numerical findings.

\subsection{Heuristic}

Let  $w^\star$ solving $s_n = (n-s_n) w^\star \tilde{m}(w^\star)$. We write $w$ for $w^\star$ for short. Let us introduce the quantity, for $u\in(0,1)$,
\begin{align*}
f_n(u)&=\frac{\E_{\theta_0=0}(\l(X_i;w)\ind{q(X_i;w)\leq u}) }{\P_{\theta_0=0}(q(X_i;w)\leq u)}\\
&= \frac{ \int_{\chi(r(w,u))}^{+\infty}  \frac{(1-w)\phi(x)}{(1-w)\phi(x) + w g(x)}  \phi(x) dx}{\overline{\Phi}(\chi(r(w,u)))}\leq 1.
\end{align*}
It is not difficult to check that $f_n$ is continuous increasing from $(0,1)$ to $(0,f_n(1))$, with $f_n(1)\to 1$ when $n\to \infty$ and $s_n/n\to 0$. 
We propose the following heuristic.

\begin{heu}\label{heuristic}
For $\theta_0\in\ell_0[s_n]$ with ``strong signal", the following holds: 
$$
\FDR(\theta_0,\SC)/t-1 \asymp  1-f_n(u^\star) \mbox{ as $n\to \infty$ and $s_n/n\to 0$,}
$$
where $u=u^\star\in [t,1)$ is the solution of  $f_n(u) u = t$.
\end{heu}

We will also assume in the sequel that for $n$ large, the solution  $u^\star$ above is below some universal constant $v_0\in(t,1)$.

\paragraph{Justifying Heuristic~\ref{heuristic}}
For short, we write $\l_i(X)$ (resp. ${q}_i(X)$) for $\l(X_i;w)$ (resp. $q(X_i;w)$).
First, let us observe that the $\SC$ procedure can be expressed as a thresholding rule rejecting the null hypotheses corresponding to $|X_i|$ larger than some threshold (contrary to {\tt EBayesL} and {\tt EBayesq}, this threshold in general depends on $X$).  Hence, even if the $\SC$ procedure is {\it a priori} not related to {\tt EBayesq} procedure, we can express this procedure as rejecting the null hypotheses corresponding to ${q}_i(X)\leq u(X,t)$ for some function $u(\cdot,\cdot)$. Clearly, the procedure ${\tt EBayesq}$ at threshold $u(X,t)$ is the $\SC$ procedure at threshold $t$.
Now assume that $u(X,t)$ is well concentrated around a value $u^\star$ (away from $0$ and $1$), so that $\SC(t)\approx {\tt EBayesq}(u^\star)$.
 The proof of Theorem~\ref{theorem:psharpFDRcontrol} hence suggests that, when the signal is large, 
$$
 \FDR(\theta_0, \SC(t)) \approx \FDR(\theta_0, {\tt EBayesq}(u^\star)) \approx  \frac{(n-s_n)\P_{\theta_0=0}(q_i(X)\leq u^\star) }{s_n+(n-s_n)\P_{\theta_0=0}(q_i(X)\leq u^\star)}  \approx u^\star.
$$
Next, the definition of $\SC(t)$ implies that 
$$
\frac{\sum_{i=1}^n \l_i(X)\ind{q_i(X)\leq u^\star}}{\sum_{i=1}^n \ind{q_i(X)\leq u^\star}}\approx \frac{\sum_{i=1}^n \l_i(X)\ind{q_i(X)\leq u(X,t)}}{\sum_{i=1}^n \ind{q_i(X)\leq u(X,t)}} \approx t .
$$
In addition, by standard concentration arguments and since the signal is strong, we also have
\begin{align*}
\frac{\sum_{i=1}^n \l_i(X)\ind{q_i(X)\leq u^\star}}{\sum_{i=1}^n \ind{q_i(X)\leq u^\star}} &\approx \frac{\sum_{i\in\cH_0} \l_i(X)\ind{q_i(X)\leq u^\star}}{s_n+\sum_{i\in\cH_0} \ind{q_i(X)\leq u^\star}}\\
&\approx \frac{(n-s_n)\E_{\theta_0=0}(\l_i(X)\ind{q_i(X)\leq u^\star}) }{s_n+(n-s_n)\P_{\theta_0=0}(q_i(X)\leq u^\star)}\\
&= f_n(u^\star) \times \frac{(n-s_n)\P_{\theta_0=0}(q_i(X)\leq u^\star) }{s_n+(n-s_n)\P_{\theta_0=0}(q_i(X)\leq u^\star)}\\
&\approx f_n(u^\star) \times  u^\star.
\end{align*}
Combining the above fact leads to $ \FDR(\theta_0, \SC(t)) \approx u^\star$ and $f_n(u^\star) u^\star\approx t$, which justifies, provided the remainder terms in the previous approximations are of smaller order, that   
$\FDR(\theta_0,\SC)/t-1 \asymp u^\star/t-1 \asymp 1/f_n(u^\star)-1$ and leads to Heuristic~\ref{heuristic}.

\subsection{Convergence of $f_n$}

Heuristic~\ref{heuristic} suggests that the inflation of the FDR of $\SC$ is determined by how much $f_n(u)$ is below $1$ for some fixed $u\in(0,1)$. The following result provides the order of $1-f_n(u)$ when $n$ is large. 

\begin{lemma}\label{lem:final}
Consider the quasi-Cauchy case.
There exist universal constants $c>0$, $C>0$ such that the following holds. Let $u_0,v_0\in(0,1)$, with $u_0<v_0$. For all $u\in (u_0,v_0)$, for all $w\in(0,1)$ smaller than some $\omega(u_0,v_0)>0$, we have
 \begin{align*}
c\frac{\log(\log (1/w))  }{ \zeta(w)^2} \leq 1-\frac{ \int_{\chi(r(w,u))}^{\infty}  \frac{(1-w)\phi(x)}{(1-w)\phi(x) + w g(x)}  \phi(x) dx}{\overline{\Phi}(\chi(r(w,u)))}\leq C \frac{\log(\log (1/w))}{ \zeta(w)^2}.
\end{align*}
In particular, for $w=w^\star$ solving $s_n = (n-s_n) w^\star \tilde{m}(w^\star)$ with $s_n \le n^\upsilon$ for some $\upsilon\in(0,1)$, and $n$ any integer larger than some $ N(\upsilon,u_0,v_0)>0$,
  \begin{align*}
c\frac{  \log(\log (n/s_n))}{\log (n/s_n)} \leq 1-f_n(u) \leq C\frac{ \log(\log (n/s_n))}{\log (n/s_n)} .
\end{align*}
\end{lemma}

\begin{proof}
Denoting $h(x)=w\be(x)\phi(x)/(1+w\be(x))$ and $\chi_w=\chi(r(w,u))$, 
\begin{align*}
\lefteqn{\int_{\chi_w}^{\infty}  \frac{\phi(x)}{(1-w)\phi(x) + w g(x)}  \phi(x) dx
 = \int_{\chi_w}^{\infty} \frac{1}{1+w\be(x)}  \phi(x) dx}&& \\
& = \int_{\chi_w}^{\infty} \phi(x) dx 
- \int_{\chi_w}^{\infty} h(x) dx\\
& = \overline{\Phi}(\chi_w) - \int_{\chi_w}^{\zeta(w)} h(x) dx
- \int_{\zeta(w)}^\infty h(x) dx.
\end{align*}
The following bounds on $h(x)$ follow from the definition of $\be=g/\phi-1$ and the fact that $x$ is large enough (as $w$ is small),
 \begin{align*}
  \phi(x)/2 & \le h(x)  \le \phi(x),\qquad & x\in[\zeta(w),\infty);\\
  wg(x)/4 & \le h(x)  \le wg(x),\qquad & x\in[\chi_w,\zeta(w)].
 \end{align*}
As $g$ is decreasing for $x$ large, the last line also implies 
$wg(\zeta(w))/4 \le h(x) \le wg(\chi_w)$ for  $x\in[\chi_w,\zeta(w)]$. 
Putting this together with the previous identity leads to (remember also that $ \chi(r(w,u))\leq \zeta(w)$ from Lemma~\ref{propchizeta})
\begin{align*}
\lefteqn{\overline{\Phi}(\chi_w) - \overline{\Phi}(\zeta(w))
-w g(\chi_w)[\zeta(w)-\chi_w]
}&& \\
& \le
\int_{\chi_w}^{\infty}  \frac{\phi(x)}{(1-w)\phi(x) + w g(x)}  \phi(x) dx\\
& \le  
\overline{\Phi}(\chi_w) - \overline{\Phi}(\zeta(w))/2
-w g(\zeta(w))[\zeta(w)-\chi_w]/4.
\end{align*} 
Further note that in the quasi-Cauchy case,
\begin{align*}
\overline{\Phi}(\chi_w) & \asymp \frac{w u}{(1-u)}\overline{G}(\chi_w)
\asymp w\frac{u}{1-u}\chi_w^{-1}\\
\overline{\Phi}(\zeta(w)) & \asymp \frac{\phi(\zeta(w))}{\zeta(w)}
\asymp w\frac{g(\zeta(w))}{\zeta(w)}\asymp w\zeta(w)^{-3}.
\end{align*}
As $u$ is bounded away from $0$ and $1$, we have
$\chi_w\sim \zeta(w)\sim (2\log(1/w))^{1/2}$. 
Also, it follows from the proofs of Lemmas \ref{propchizeta} and \ref{lemxizeta} respectively, using again that $u$ is bounded away from $0$ and $1$, that, for universal constants $c,C>0$,
$$
 c \frac{\log\log(1/w)}{ \zeta(w)}\leq  \zeta(w)-\chi_w\le C \frac{\log\log(1/w)}{ \zeta(w)}.
$$
Combining the previous estimates leads to the desired bound.
\end{proof}

Combining Lemma~\ref{lem:final}, the fact that $u=u^\star$ is the solution of  $f_n(u) u = t$, and 
that $u^\star\in [t,v_0]$ for $n$ large, we obtain 
$$
1-f_n(u^\star)\asymp (\log (n/s_n))^{-1} \log(\log(n/s_n)).
$$
Finally, the latter combined with Heuristic~\ref{heuristic} suggests that the FDR of $\SC(t)$ procedure is of order $t$ plus a positive term decreasing slowly with 
 $n/s_n$. This supports the fact that the FDR of $\SC(t)$ seems larger than $t$ on Figure~\ref{fig:SC}, but still converging to the targeted level $t$ for $n/s_n$ very large, as in the table of Section~\ref{tab:final}. Making the Heuristic precise is a very interesting direction for future work.

\section{Auxiliary lemmas} 

\begin{lemma} \label{bphi} 
For any $x>0$,
\[ \frac{x^2}{1+x^2}\frac{\phi(x)}{x}\le \bfi(x)\le \frac{\phi(x)}{x}.\]
In particular, for any $x\ge 1$,
$ \bfi(x) \ge \frac12\frac{\phi(x)}{x}$ and $\bfi(x)\sim \frac{\phi(x)}{x}$ when $x\to \infty$.
Furthermore, for any $y\in(0,1/2)$,
$$
\left\{ \left(2\log (1/y) - \log \log (1/y)-\log(16\pi)\right)_+\right\}^{1/2}\leq \ol{\Phi}^{-1}(y) \leq \left\{ 2\log (1/y)\right\}^{1/2}.
$$
and also for $y$ small enough,
$$
\ol{\Phi}^{-1}(y) \leq \left\{ 2\log (1/y) - \log \log (1/y)\right\}^{1/2} .
$$
In particular, $\ol{\Phi}^{-1}(y) \sim \left\{ 2\log (1/y)\right\}^{1/2}$ when $y\to 0$. 
\end{lemma}

\begin{proof}
The first display of the lemma are classical bounds on $\overline\Phi$. The second display follows using the first one and similar inequalities as those used to derive bounds on $\xi, \zeta, \chi$.
Let us prove the last relation: for all $y\in (0,1/2)$,  
$$y\left\{ \left(2\log (1/y) - \log \log (1/y)-\log(16\pi)\right)_+\right\}^{1/2}\leq y \ol{\Phi}^{-1}(y) \le \phi(\ol{\Phi}^{-1}(y)) $$
Hence, 
\begin{align*}
\ol{\Phi}^{-1}(y) &\leq \left\{ -2\log \left(y\left\{ \left(2\log (1/y) - \log \log (1/y)-\log(16\pi)\right)_+\right\}^{1/2}\right)\right\}^{1/2}\\
&\leq \left\{ -2\log y - \log \left(\left(2\log (1/y) - \log \log (1/y)-\log(16\pi)\right)_+\right)\right\}^{1/2}
\end{align*}
which provides the result. 
\end{proof}

\begin{lemma} \label{lemphixphiy} 
For any $x,y\in \R$, with $|x-y|\leq 1/4$, we have
\begin{equation}
\ol{\Phi}(x) \geq \ol{\Phi}(y)\: \frac{1}{4} e^{-(x^2-y^2)_+/2}.
\end{equation}
\end{lemma}

\begin{proof}
Let us assume $x>y$ (otherwise the result is trivial).
If $y\leq 0$, we have $\ol{\Phi}(x) \geq \ol{\Phi}(1/4) \geq 1/4 \geq 1/4 \ol{\Phi}(y)$ so the inequality is true.
Assume now $y>0$.
By Lemma~\ref{bphi}, 
\begin{align*}
\frac{\ol{\Phi}(x)}{\ol{\Phi}(y)} &\geq \frac{\ol{\Phi}(y+1/4)}{\ol{\Phi}(y)} \ind{y\leq 1} + \frac{xy}{1+x^2} e^{-(x^2-y^2)/2} \ind{y\geq 1}\\
&\geq  \frac{\ol{\Phi}(5/4)}{\ol{\Phi}(1)} \ind{y\leq 1} + \frac{x^2}{2(1+x^2)} e^{-(x^2-y^2)/2} \ind{y\geq 1}
\end{align*}
because $y\in(0,\infty)\to \frac{\ol{\Phi}(y+1/4)}{\ol{\Phi}(y)}$ is decreasing and $y\geq x/2$ when $y\geq 1$. This concludes the proof.
\end{proof}

\begin{lemma}\label{th:bernstein}[Bernstein's inequality]
Let $W_i$, $1\leq i \leq n$ centered independent variables with $|W_i|\le \mtc{M}$ and $\sum_{i=1}^n\var(W_i)\le V$, then for any $A>0$,
\[ P\left[ \sum_{i=1}^n W_i >A \right] \le \exp\left\{-\frac12 A^2/(V+\mtc{M}A/3) \right\}.\]
\end{lemma}

\begin{lemma} \label{integ}
There exists a constant $C>1$ such that, for any $M\ge 1$,
\[ \frac{e^M}{M^2}-1\le \int_1^M \frac{e^v}{v^2}dv \le C\frac{e^M}{M^2}. \]
\end{lemma}
\begin{proof}
For $M\le 3$ the result is immediate for $C$ chosen large enough beforehand. For $M>3$, one writes
\[  \int_3^M \frac{e^v}{v^2}dv = \left[ \frac{e^v}{v^2}\right]_3^M + 2\int_3^M  \frac{e^v}{v^3}dv
\le \frac{e^M}{M^2} + \frac{2}{3} \int_3^M \frac{e^v}{v^2}dv, \]
so that $\int_3^M \frac{e^v}{v^2}dv\le 3e^M/M^2$, from which the upper bound follows. The lower bound follows from integrating by parts between $1$ and $M$ and noting that the second term is nonnegative.
\end{proof}

\begin{lemma}\label{lembinom}
For $m \geq 1$, $p_1,\dots,p_m \in(0,1)$, consider $U=\sum_{i=1}^m B_i$, 
where $B_i\sim \mathcal{B}(p_i)$,  $1\leq i \leq m$, are independent.
For any nonnegative variable $T$ independent of $U$, we have 
\begin{align}
 \E \left(\frac{T }{T + U}\ind{T>0}\right) &\leq e^{-\E U}+ \frac{12\: \E T}{\E U} .  \label{equ-binom3}
\end{align}
\end{lemma}

\begin{proof} 
Let us prove the two following inequalities:  for all $u>0$,
\begin{align*}
\P(U=0) &\leq   e^{-\sum_{i=1}^m p_i}.\\
 \E \left(\frac{u \:  \sum_{i=1}^m p_i}{u \: \sum_{i=1}^m p_i + U \vee 1}\right) &\leq 12u. 
\end{align*} 
For the first inequality, using $\log (1-x) \leq -x$ for all $x\in(0,1)$, 
$$
\P(U=0) = \prod_{i=1}^m (1-p_i) = e^{ \sum_{i=1}^m \log (1-p_i)} \leq   e^{- \sum_{i=1}^m p_i}=e^{-\E U}.
$$
For the second assertion, we have 
\begin{align*}
& \E \left(\frac{u\:  \sum_{i=1}^m p_i}{u \: \sum_{i=1}^m p_i + U \vee 1}\right) \leq  \E  \left(\frac{\sum_{i=1}^m p_i}{U \vee 1}\right) u.
\end{align*}
Now applying Bernstein's inequality, we have 
\begin{align*}
\P  \left( U \leq \sum_{i=1}^m p_i/2\right) &= \P \left( U - \sum_{i=1}^m p_i \leq -\sum_{i=1}^m p_i/2\right) \\
&\leq \exp\left\{-\frac12 \sum_{i=1}^m p_i ( 1/2)^2/(1+1/6) \right\} \leq  e^{- 0.1 \sum_{i=1}^m p_i}. 
\end{align*}
As a result, one obtains, using $x e^{-x}\le 1$ for $x\ge 0$,
\begin{align*}
&\E \left(\frac{\sum_{i=1}^m p_i}{U \vee 1}\right)\\
&\leq \E \left(\frac{\sum_{i=1}^m p_i}{U \vee 1}\ind{U> \sum_{i=1}^m p_i/2}\right)+\E \left(\frac{\sum_{i=1}^m p_i}{U \vee 1}\ind{U\leq \sum_{i=1}^m p_i/2}\right)\\
&\leq 2+ 10 \left(0.1 \:\sum_{i=1}^m p_i\right) \: e^{- 0.1 \sum_{i=1}^m p_i} \leq 12,
\end{align*}
as announced. To show \eqref{equ-binom3}, we now use the independence assumption and the concavity of $x\to \frac{x}{x+u}$ (for $u>0$), to obtain
\begin{align*}
 \E \left[  \frac{T}{T+U} \mathds{1}{\{T>0\}} \right] 
 &= \P(U=0,T>0) +
  \E \left[  \frac{T}{T+U} \mathds{1}{\{U>0\}} \right]\\
&\leq \P(U=0) + \E \left[  \frac{\E T}{\E T+U} \mathds{1}{\{U>0\}} \right]\\
&\leq \P(U=0) + \E \left[  \frac{\E T}{\E T+U\vee 1} \right].
\end{align*}
The two previous inequalities for $u=\E T/\E U$ thus give the result.
\end{proof}

\section{$\ \,$Additional numerical experiments}\label{sec:addnum}

Figures~\ref{fig2},~\ref{fig3} and~\ref{fig4}  present further numerical experiments along the lines of the comments of Section~\ref{sec:simu}.

\begin{figure}[h!]
\begin{tabular}{ccc}
\vspace{-0.5cm}
&quasi-Cauchy & Laplace\\
\vspace{-1cm}
\rotatebox{+90}{\hspace{3cm}$s_n/n=0.1$} &\includegraphics[scale=0.35]{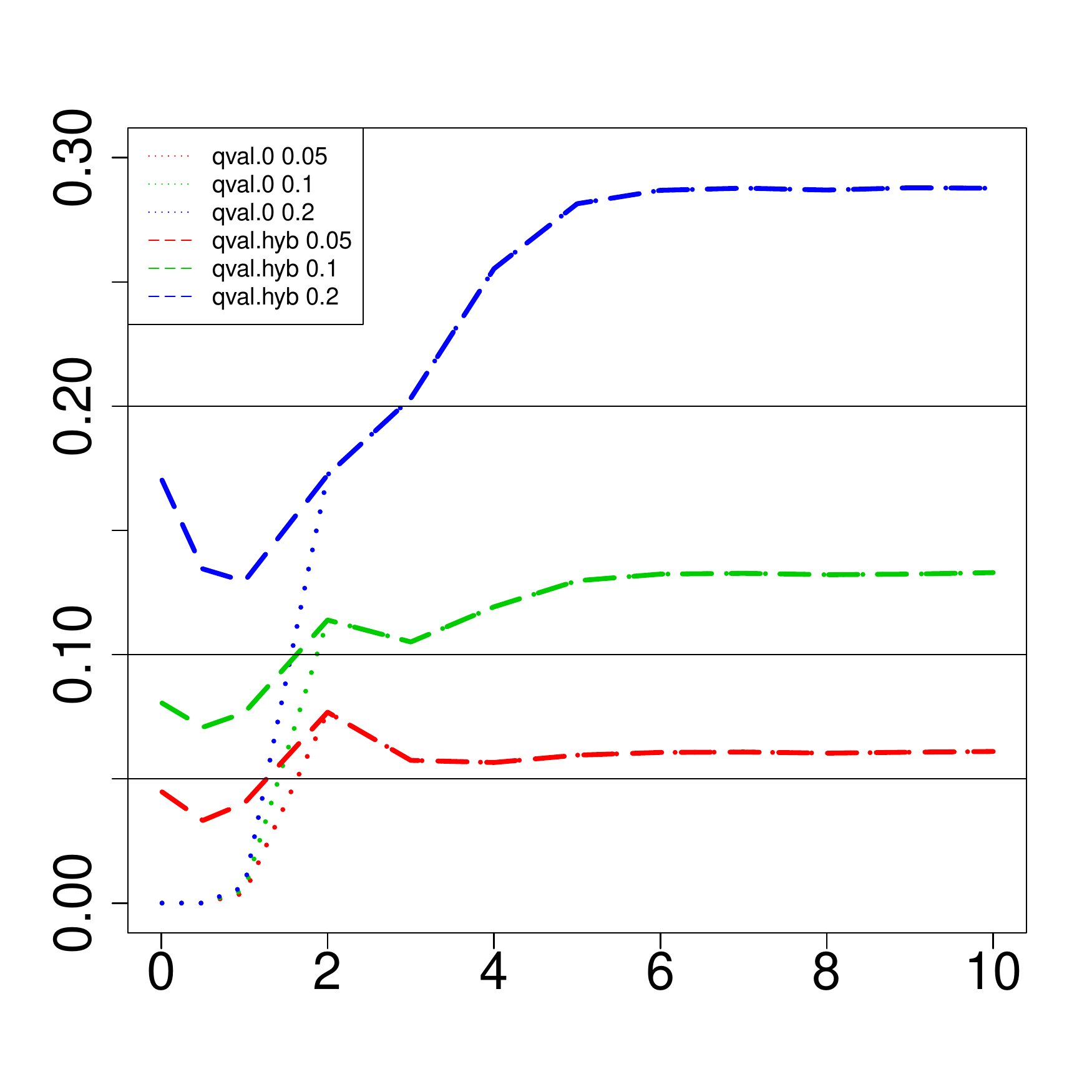}
&\includegraphics[scale=0.35]{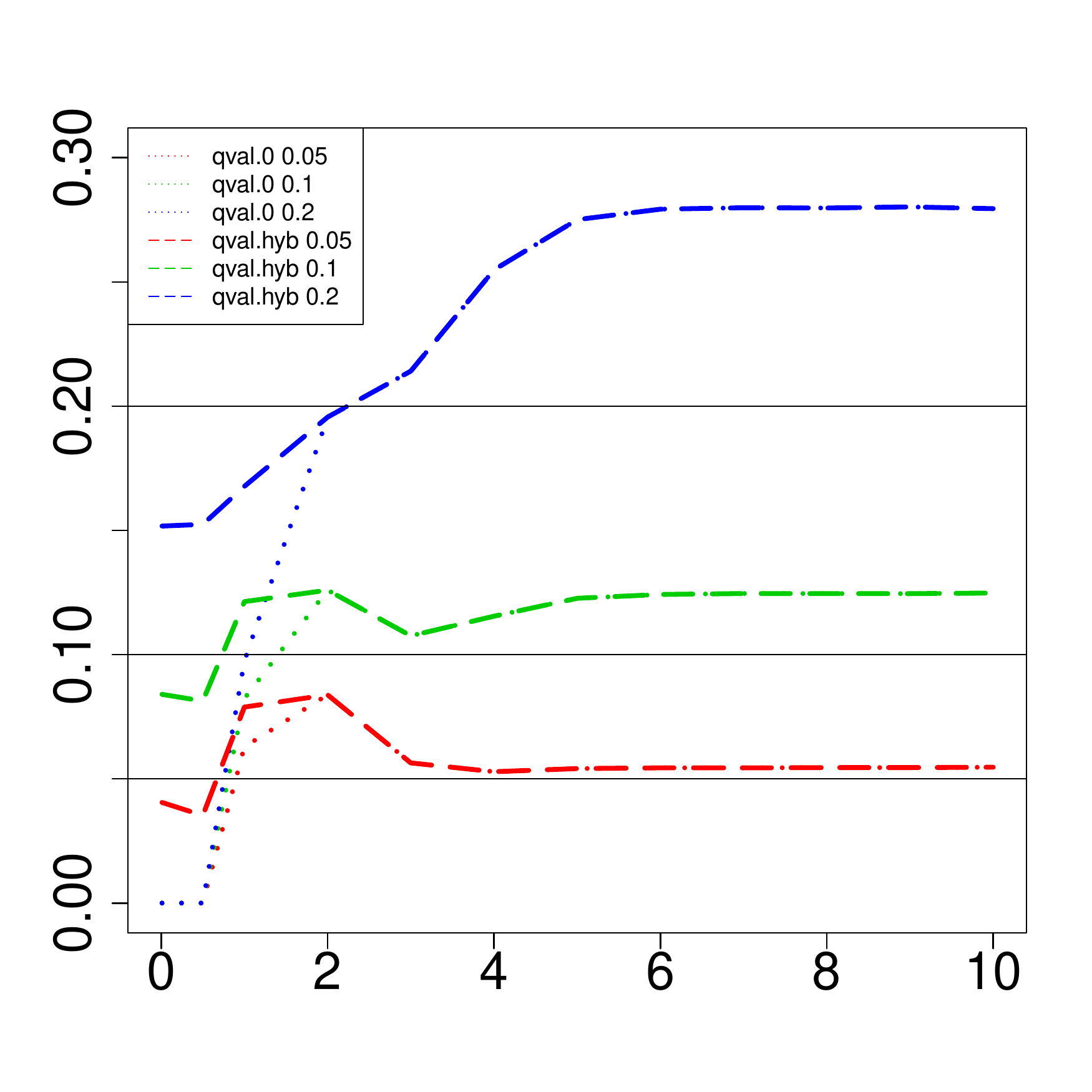}
\\
\vspace{-1cm}
\rotatebox{+90}{\hspace{3cm}$s_n/n=0.01$} &\includegraphics[scale=0.35]{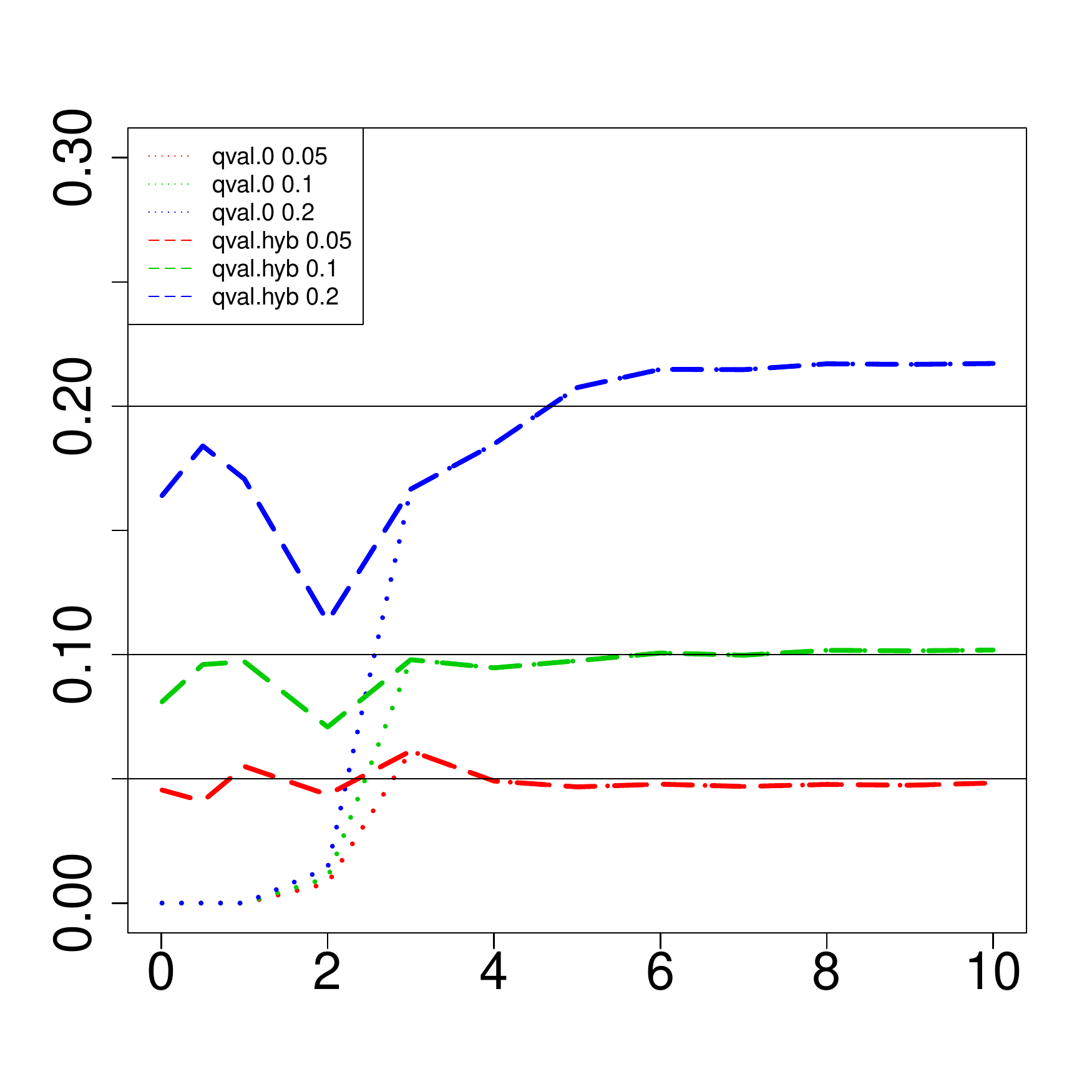}
&\includegraphics[scale=0.35]{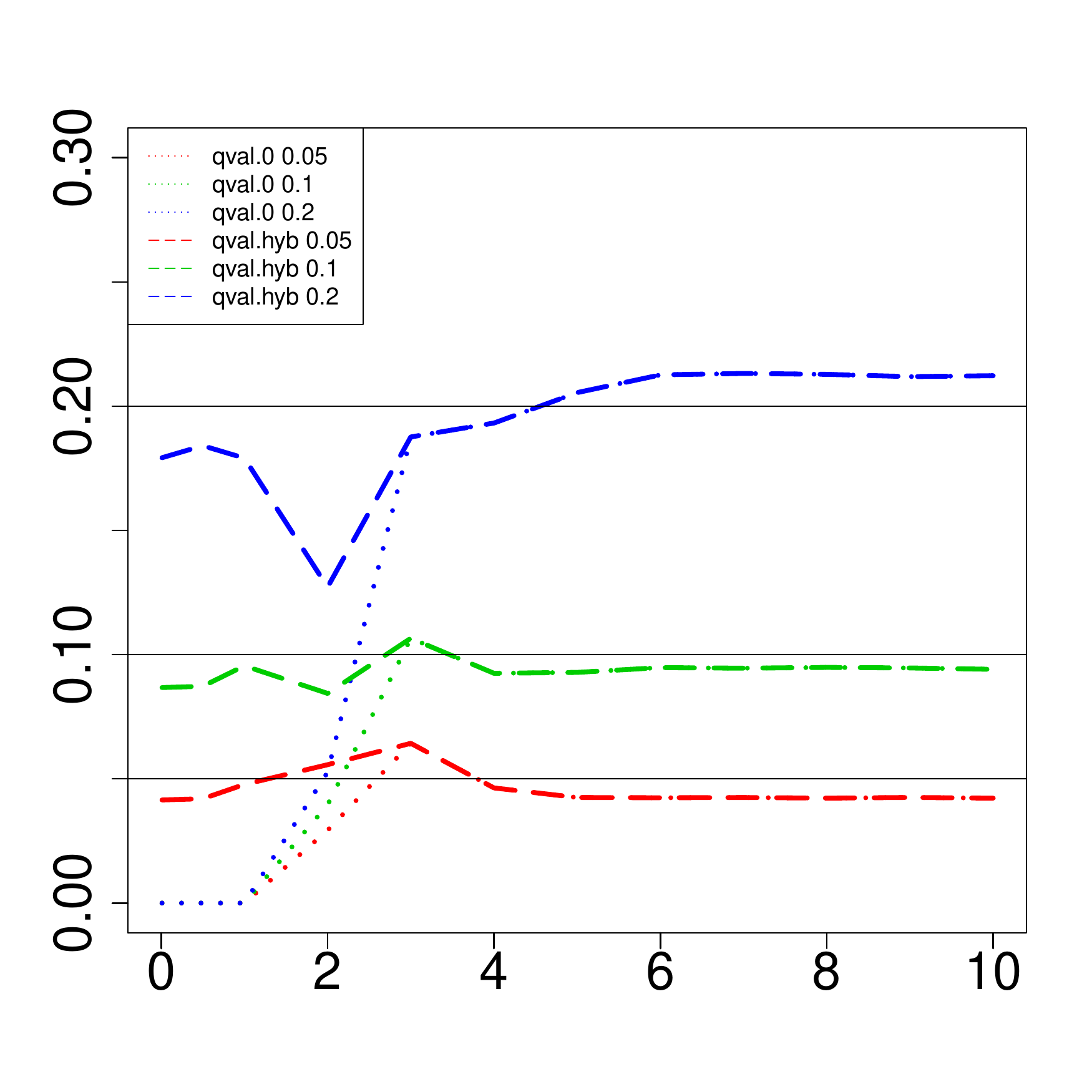}
\\
\vspace{-0.5cm}
\rotatebox{+90}{\hspace{3cm}$s_n/n=0.001$} &\includegraphics[scale=0.35]{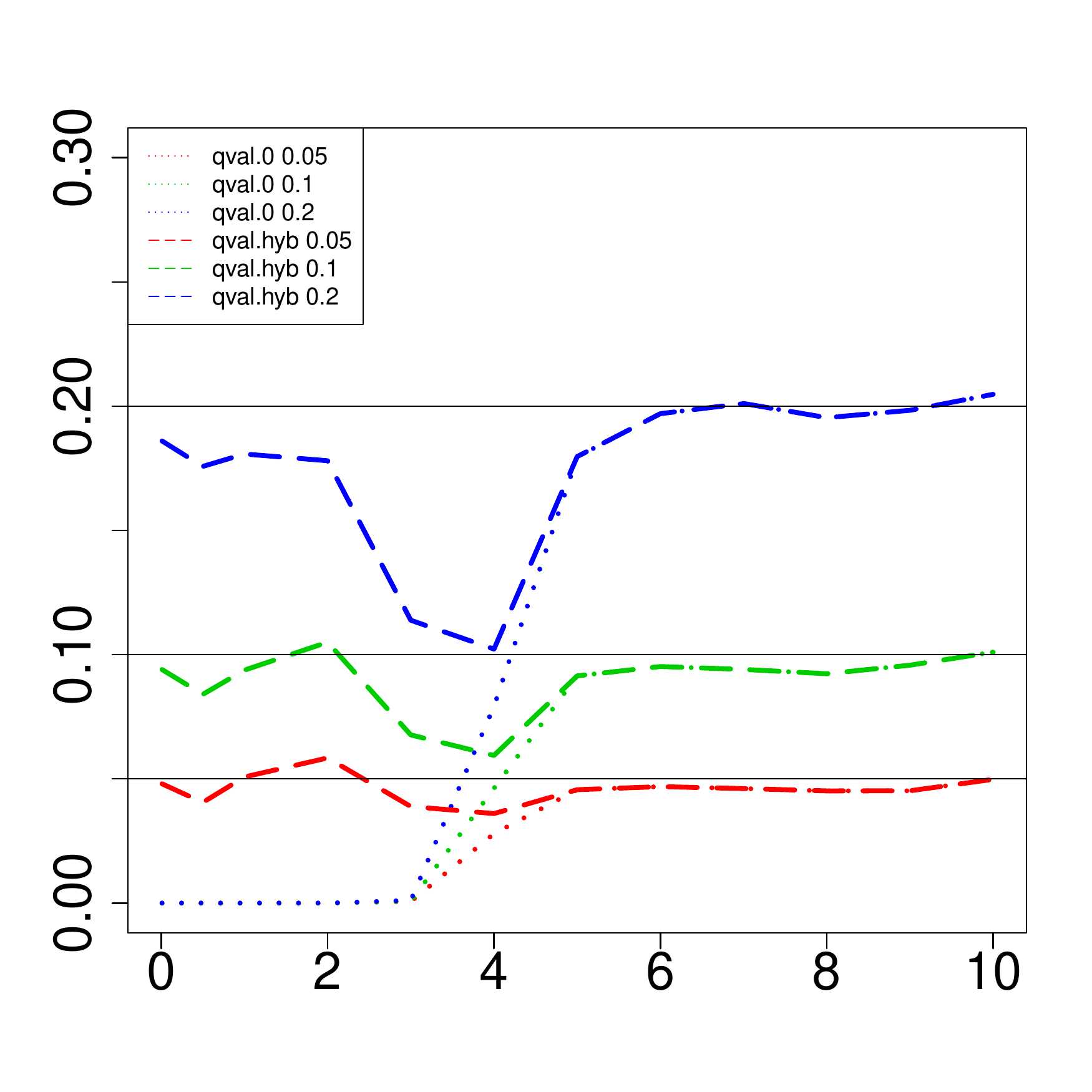}
&\includegraphics[scale=0.35]{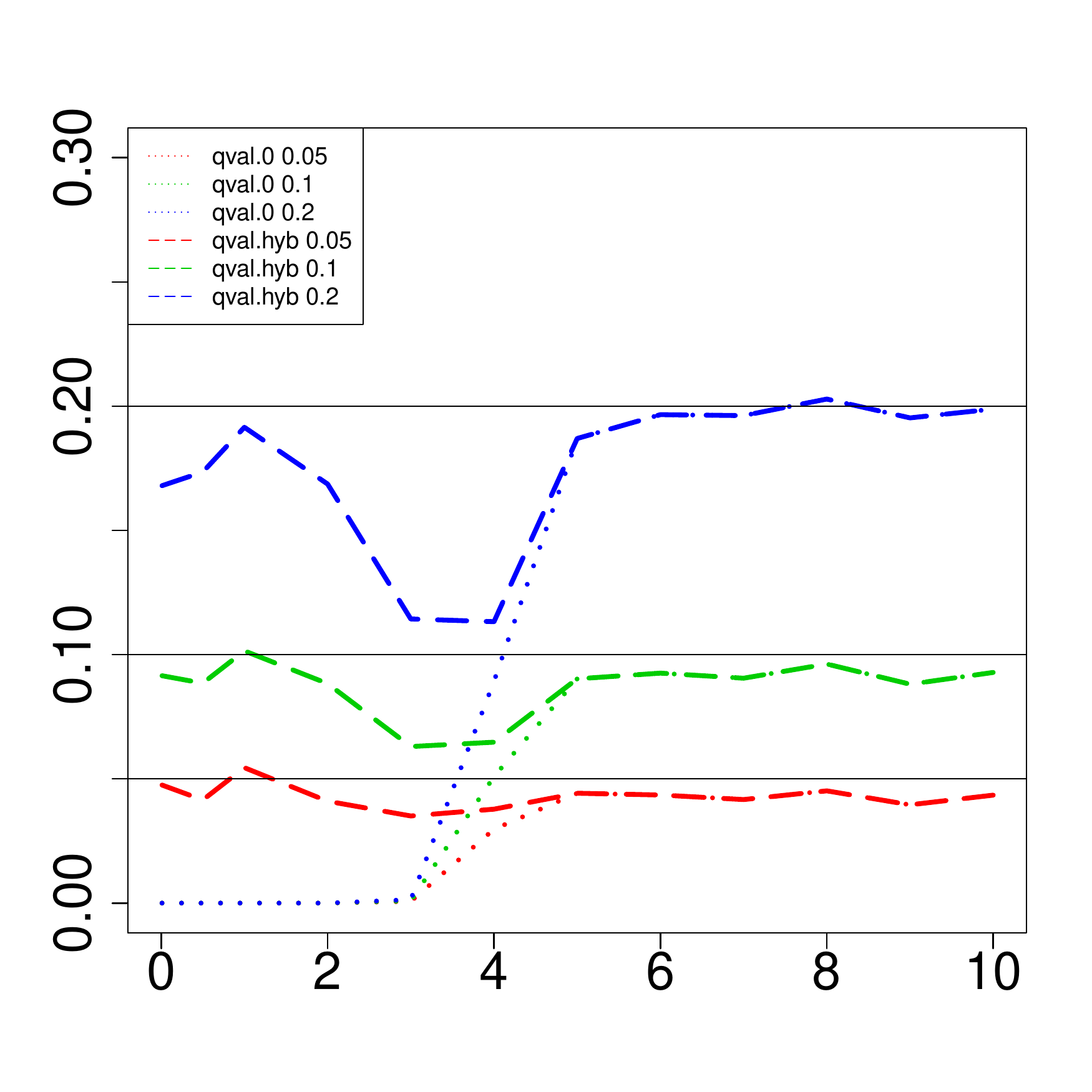}
\end{tabular}
\caption{\label{fig2}
FDR of \EBayesqseuillage and \EBayesqplus procedures with threshold $t\in\{0.05,0.1,0.2\}$.  $n=10,\,000$; $2000$ replications; alternative all equal to $\mu$ (on the $X$-axis).}
\end{figure}

\begin{figure}[h!]
\begin{tabular}{ccc}
\vspace{-0.5cm}
&quasi-Cauchy & Laplace\\
\vspace{-1cm}
\rotatebox{+90}{\hspace{3cm}$s_n/n=0.1$} &\includegraphics[scale=0.35]{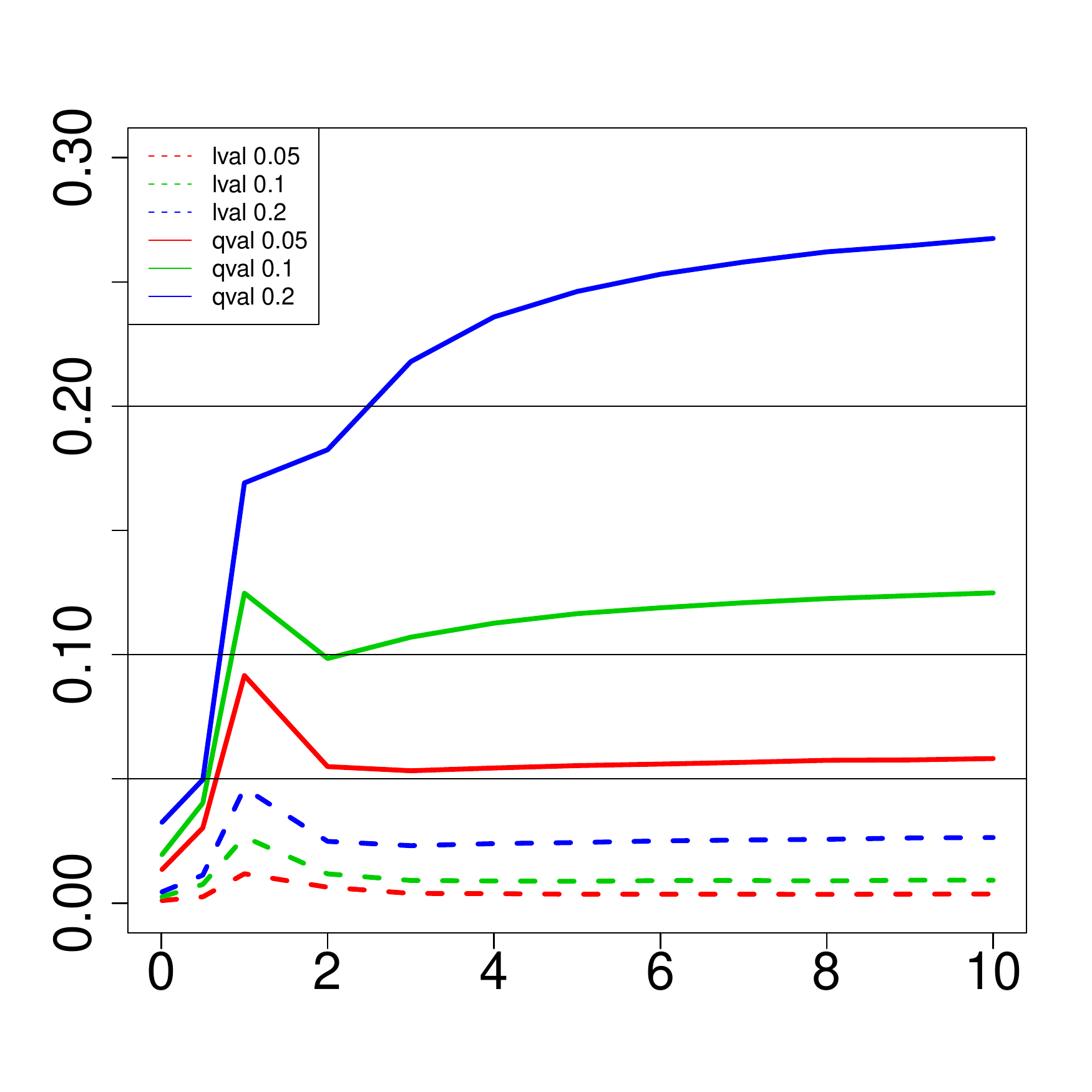}
&\includegraphics[scale=0.35]{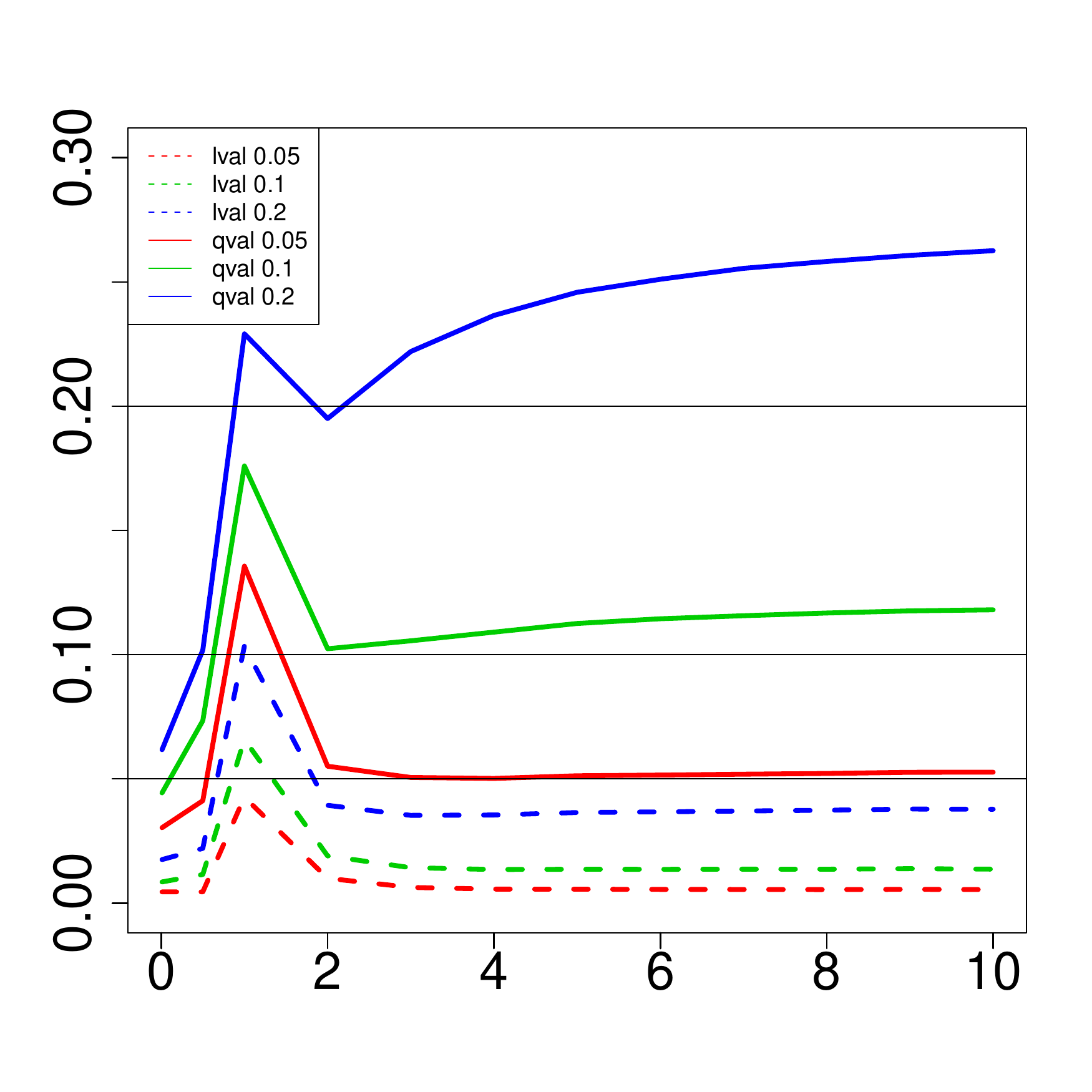}
\\
\vspace{-1cm}
\rotatebox{+90}{\hspace{3cm}$s_n/n=0.01$} &\includegraphics[scale=0.35]{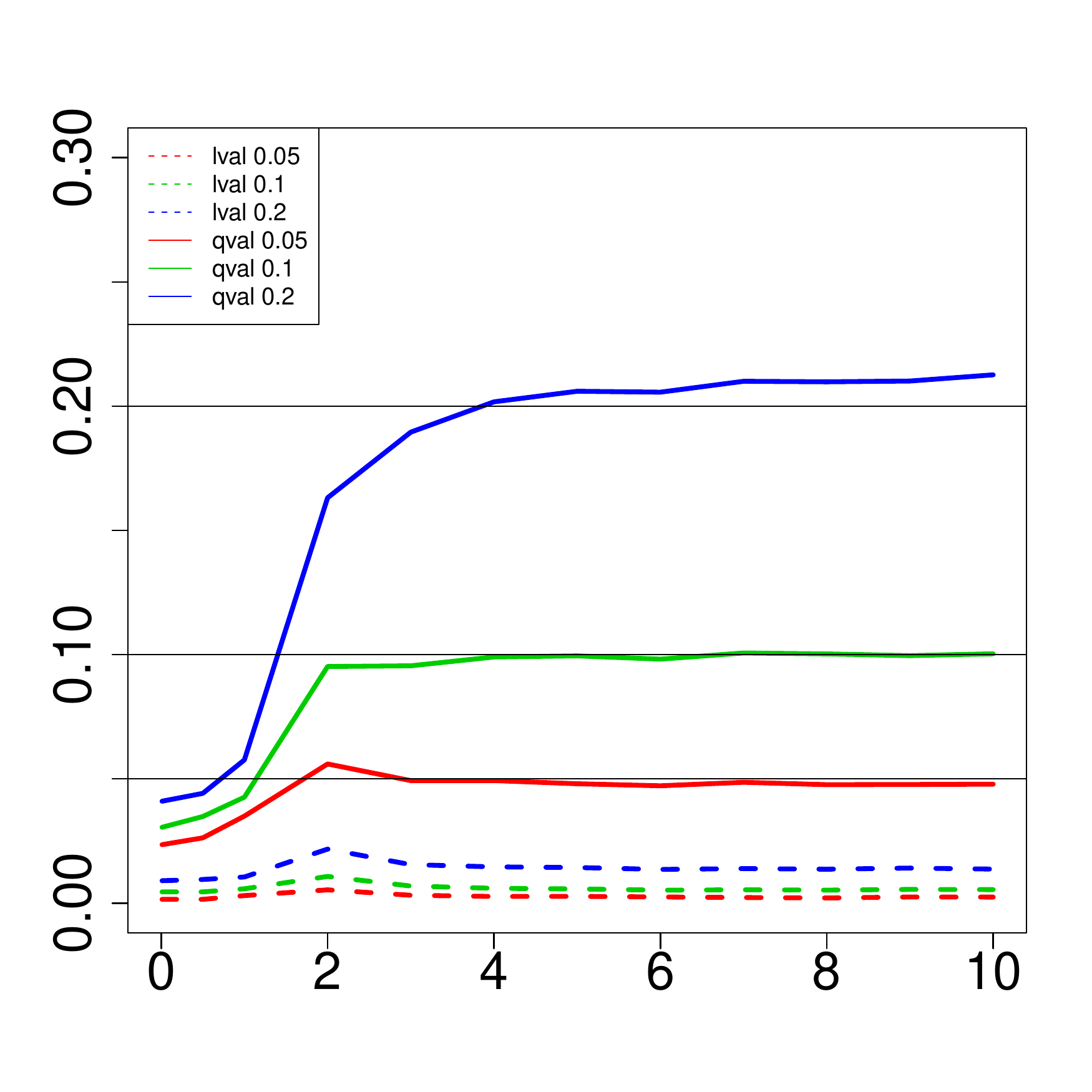}
&\includegraphics[scale=0.35]{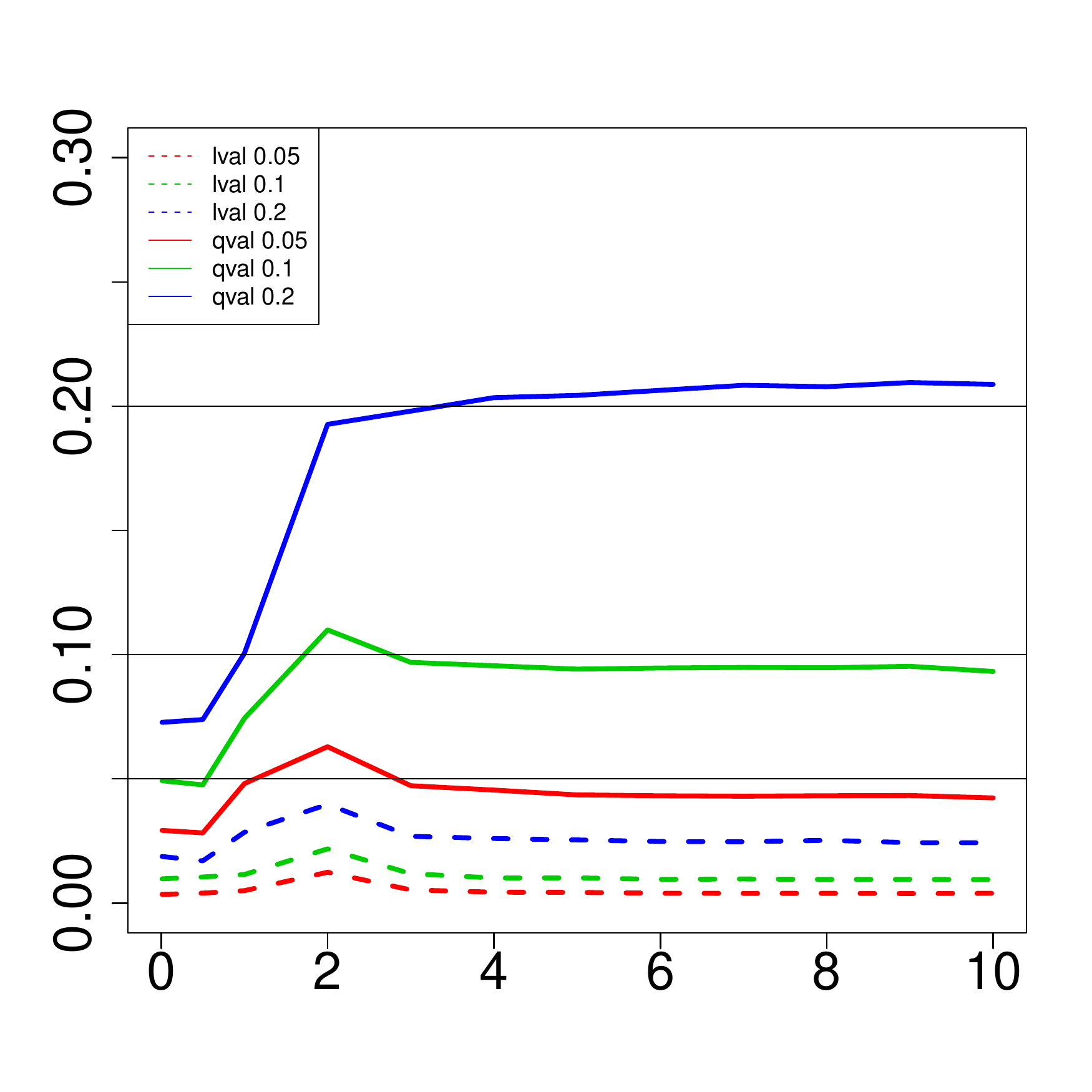}
\\
\vspace{-0.5cm}
\rotatebox{+90}{\hspace{3cm}$s_n/n=0.001$} &\includegraphics[scale=0.35]{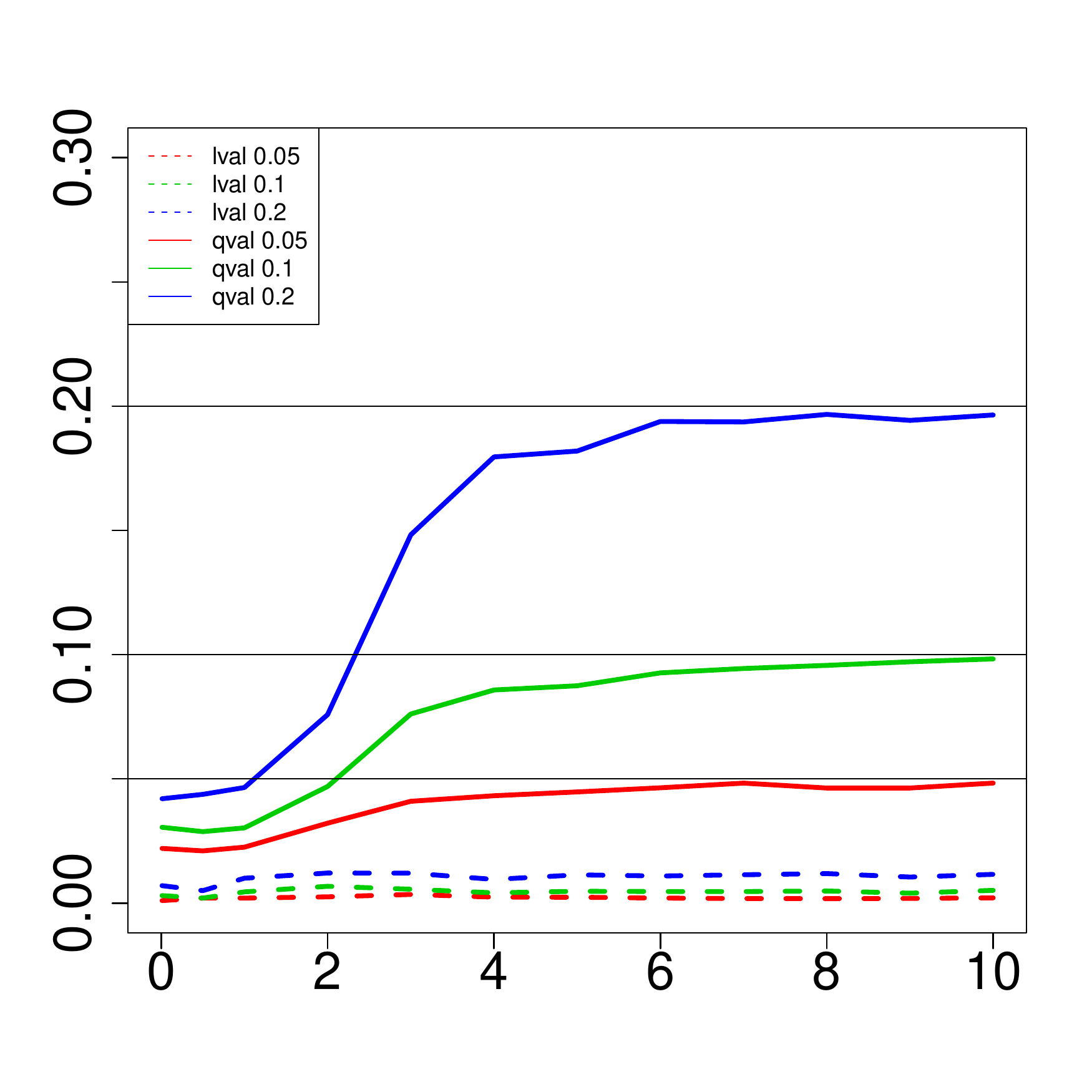}
&\includegraphics[scale=0.35]{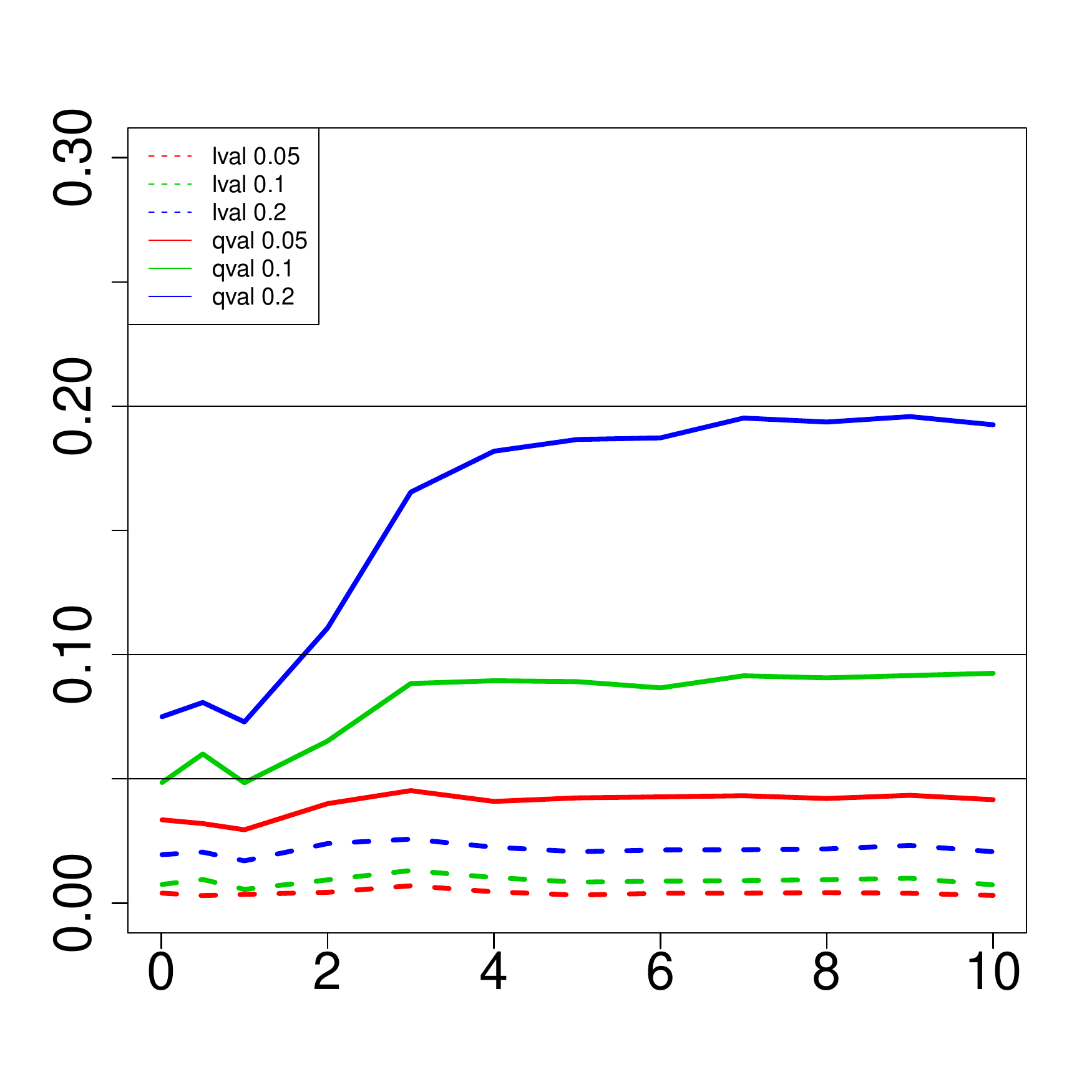}
\end{tabular}
\caption{\label{fig3}
FDR of {\tt EBayesL} and {\tt EBayesq} procedures with threshold $t\in\{0.05,0.1,0.2\}$. $n=10,\,000$; $2000$ replications; alternative values i.i.d. uniformly drawn into $[0,2\mu]$ ($\mu$ on the $X$-axis).}
\end{figure}

\begin{figure}[h!]
\begin{tabular}{ccc}
\vspace{-0.5cm}
&quasi-Cauchy & Laplace\\
\vspace{-1cm}
\rotatebox{+90}{\hspace{3cm}$s_n/n=0.1$} &\includegraphics[scale=0.35]{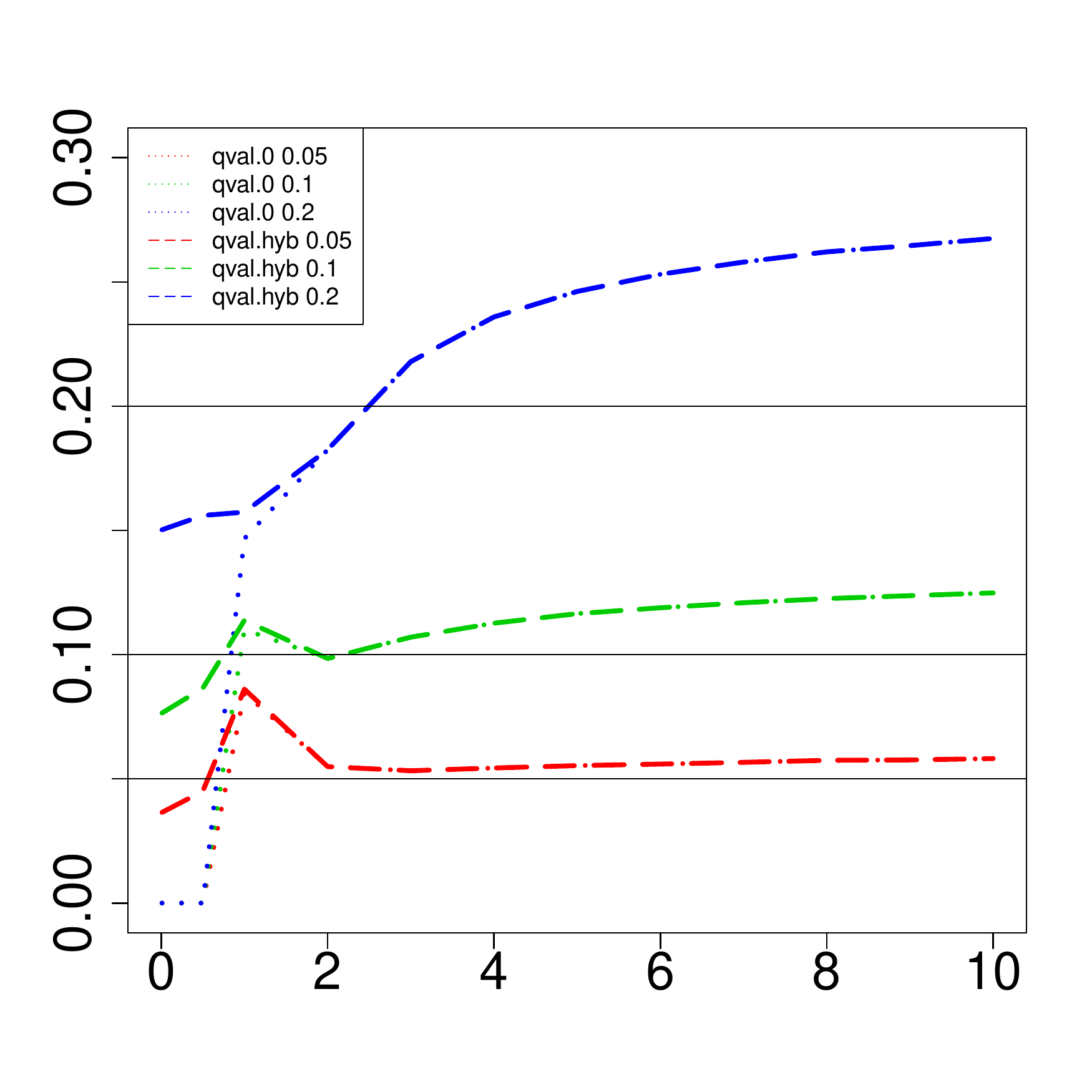}
&\includegraphics[scale=0.35]{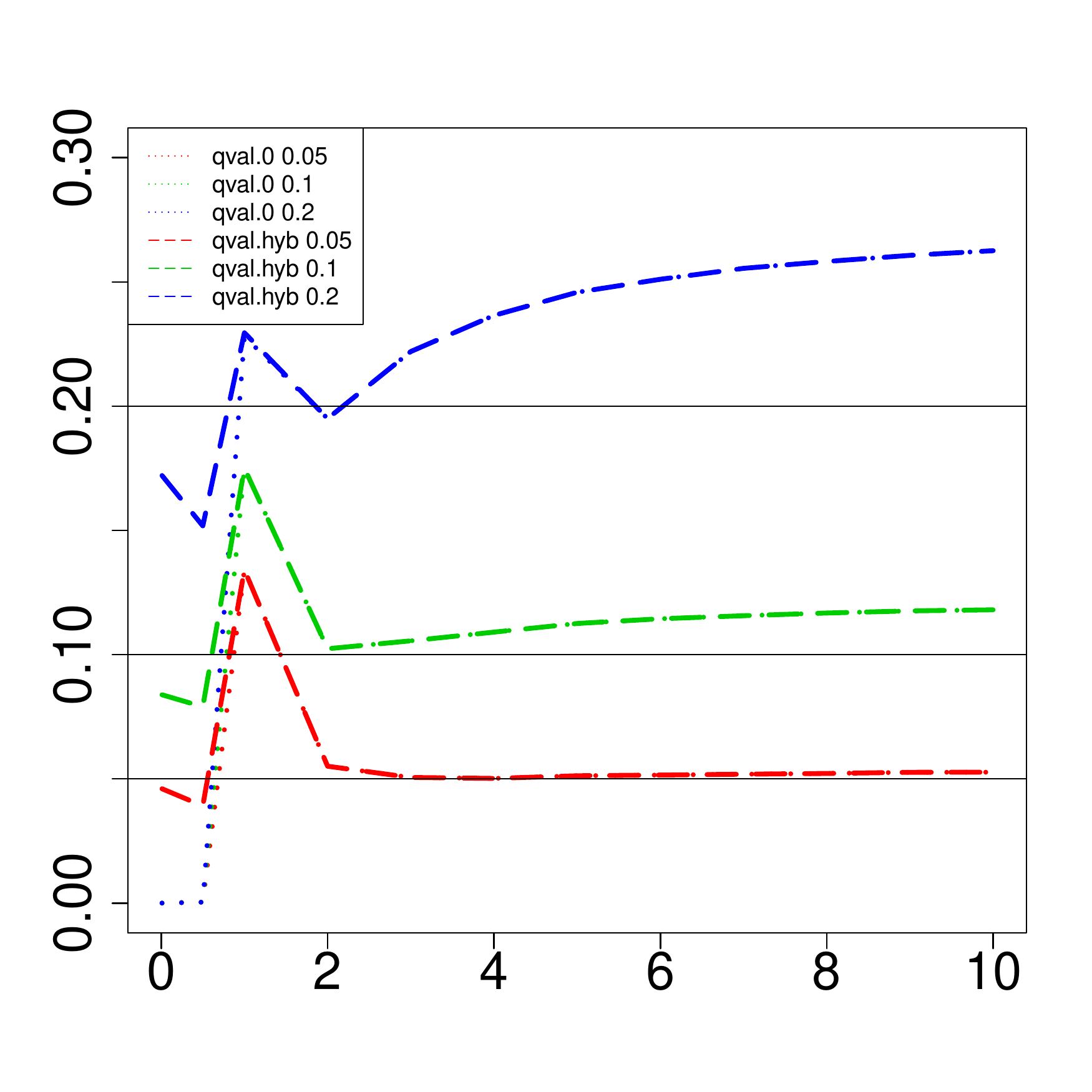}
\\
\vspace{-1cm}
\rotatebox{+90}{\hspace{3cm}$s_n/n=0.01$} &\includegraphics[scale=0.35]{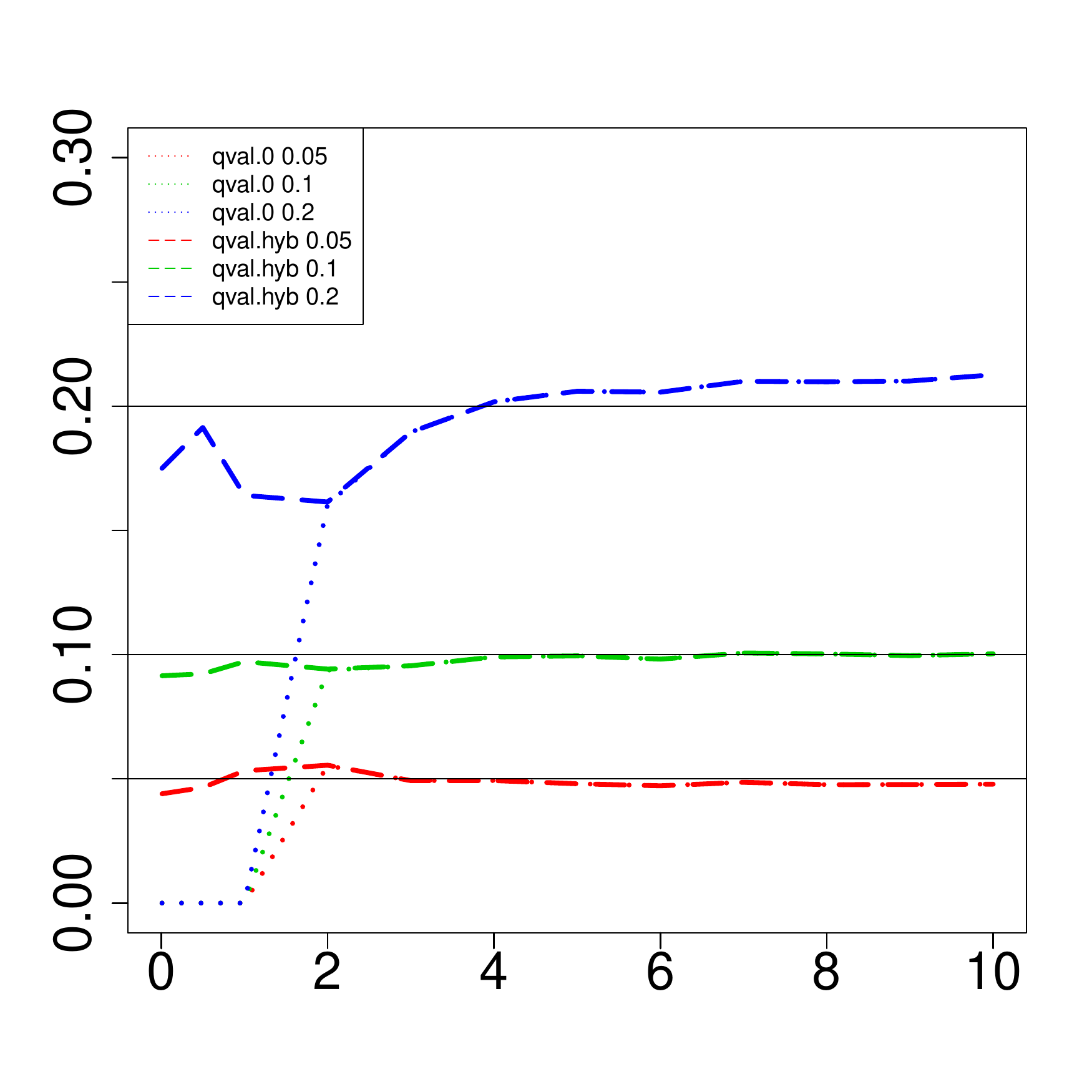}
&\includegraphics[scale=0.35]{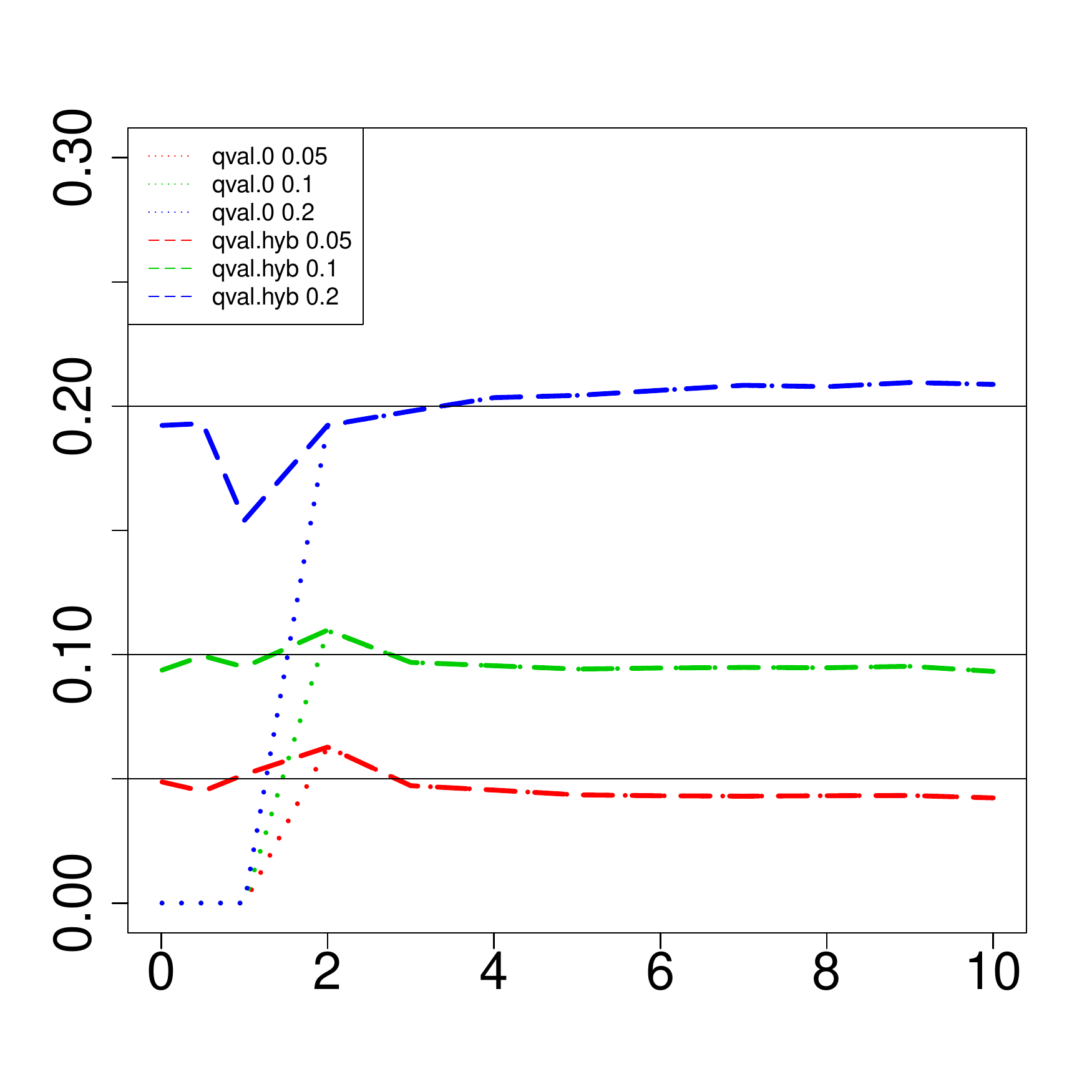}
\\
\vspace{-0.5cm}
\rotatebox{+90}{\hspace{3cm}$s_n/n=0.001$} &\includegraphics[scale=0.35]{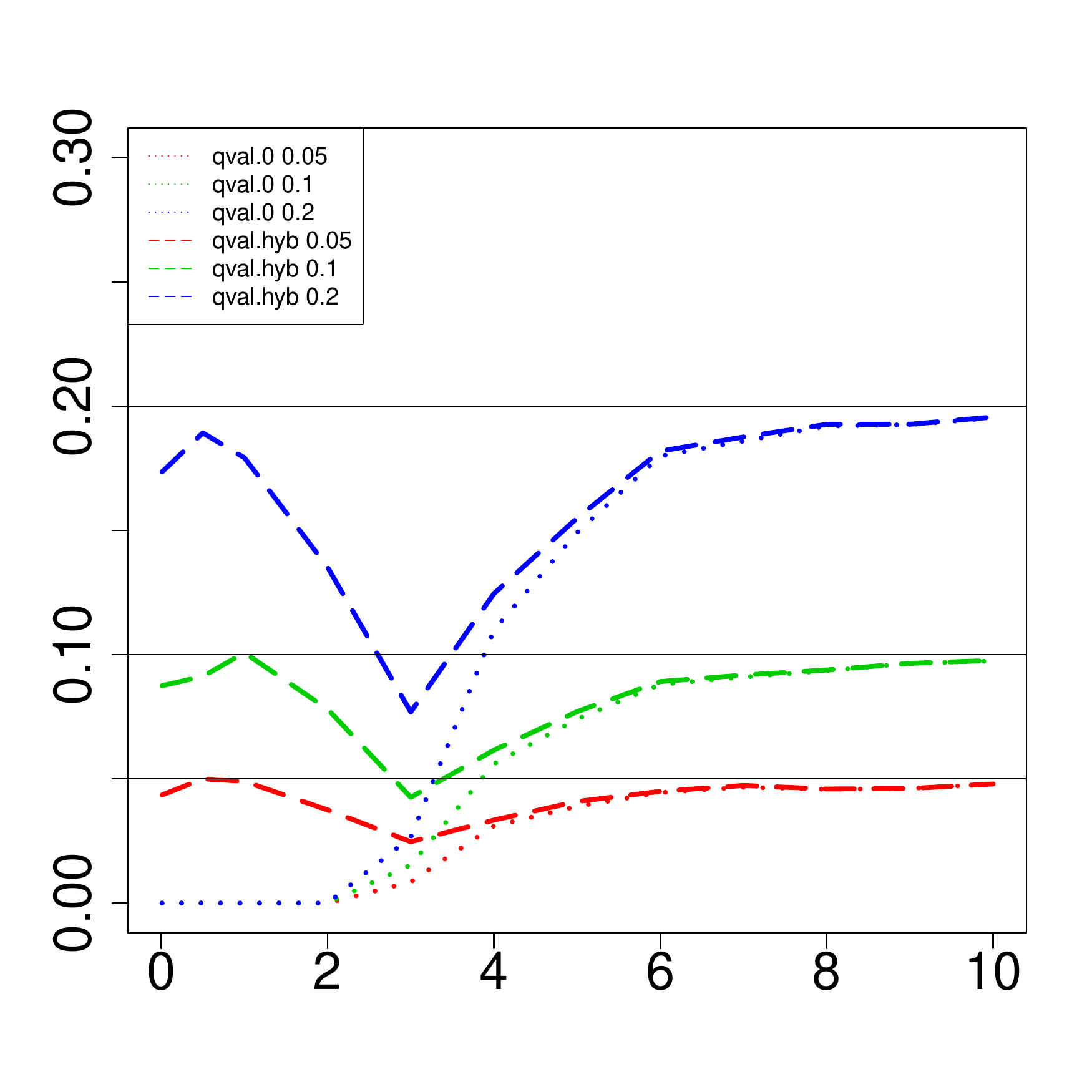}
&\includegraphics[scale=0.35]{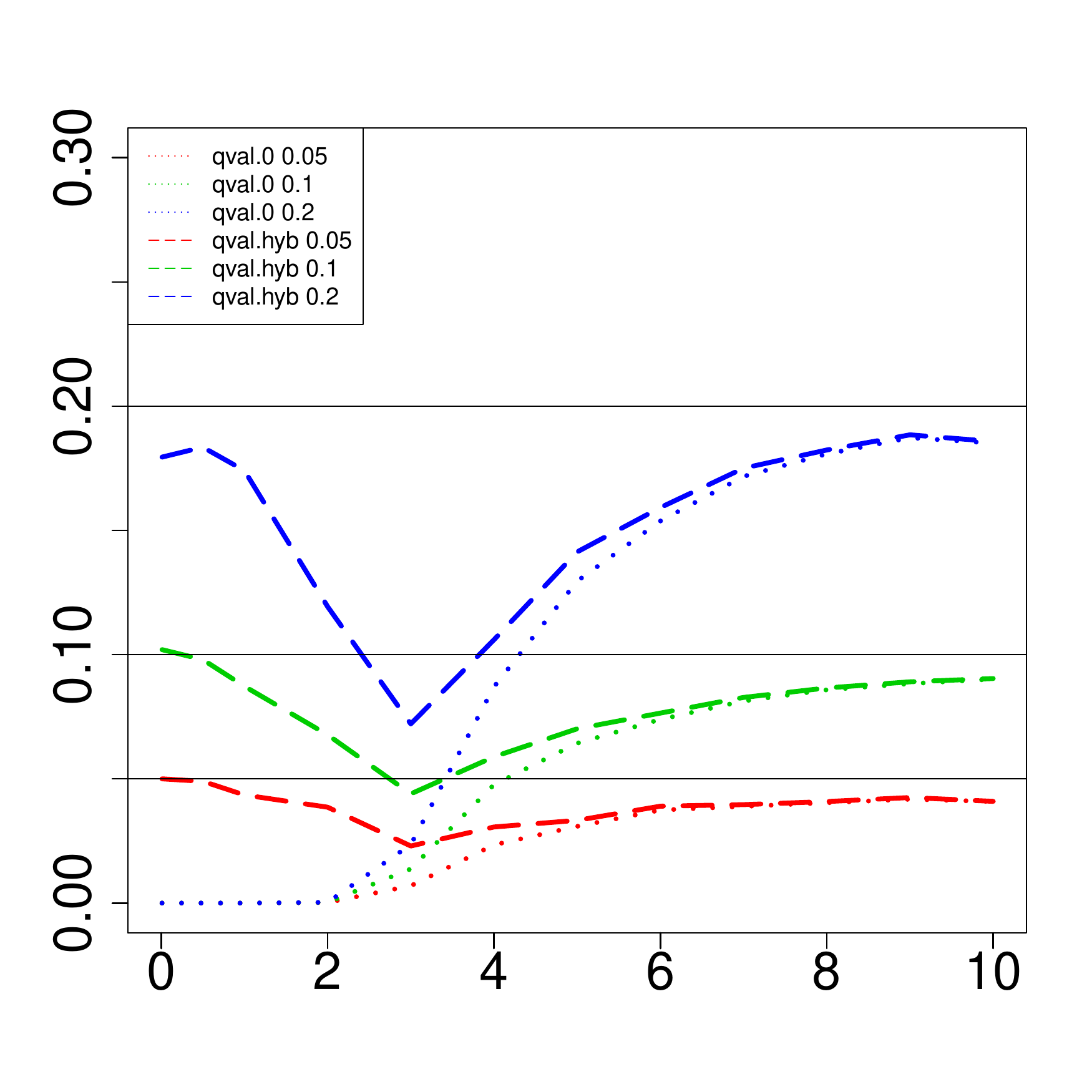}
\end{tabular}
\caption{\label{fig4}
FDR of \EBayesqseuillage and \EBayesqplus procedures with threshold $t\in\{0.05,0.1,0.2\}$. $n=10,\,000$; $2000$ replications; alternative values i.i.d. uniformly drawn into $[0,2\mu]$ ($\mu$ on the $X$-axis).}
\end{figure}

\end{document}